\documentclass{article}
\usepackage{color}
\usepackage{amssymb,amsmath,amsthm,graphicx, color, amsfonts,tabularx}
\usepackage[english]{babel}
\usepackage{authblk, relsize}
\usepackage[export]{adjustbox} 
\usepackage{geometry}

\newtheorem{theorem}{Theorem}[section]

\newtheorem{lemma}[theorem]{Lemma}

\newtheorem{definition}[theorem]{Definition}
\newtheorem{remark}[theorem]{Remark}
\newtheorem{example}[theorem]{Example}

\newcommand{\N}{\mathbb{N}}
\newcommand{\R}{\mathbb{R}}
\newcommand{\btimes}{\mathlarger{\mathlarger{\mathlarger{\times}}}}

\renewenvironment{proof}[1][.]{%
\bigskip\noindent{\bf Proof#1 }}{%
\hfill$\blacksquare$\bigskip}
\newcommand{\mZ}{\mathcal Z}
\newcommand{\mS}{\mathcal S}

\newcommand{\F}{\mathcal F}

\newcommand{\K}{\mathcal K}

\newcommand{\ve}{\varepsilon}
\newcommand{\on}{\operatorname}

\begin{document}
\pagestyle{myheadings}
\title{A multiresolution algorithm to generate images of generalized fuzzy fractal attractors}

\author[0]{Rudnei D. da Cunha
\thanks{
Instituto de Matem\'atica e Estat\'istica - UFRGS\\
Av. Bento Gon\c calves, 9500\\
Porto Alegre - 91500 - RS - Brazil.\\
Email: rudnei.cunha@ufrgs.br}
}
\affil[0]{Universidade Federal do Rio Grande do Sul\\
}
\author[1]{Elismar R. Oliveira
\thanks{
Instituto de Matem\'atica e Estat\'istica - UFRGS\\
Av. Bento Gon\c calves, 9500\\
Porto Alegre - 91500 - RS - Brazil.\\
Email: elismar.oliveira@ufrgs.br}
}
\affil[1]{Universidade Federal do Rio Grande do Sul\\
}
\author[2]{Filip Strobin
\thanks{
Instititute of Mathematics, Lodz University of Technology, W\'olcza\'nska 215, 90-924 {\L}\'od\'z, Poland.}
}
\affil[2]{Lodz University of Technology\\
}
\date{\today}
\maketitle

\noindent\rule{\textwidth}{0.4mm}
\begin{abstract}
 We provide a new algorithm to generate images of  the generalized fuzzy fractal attractors described in \cite{Oliveira-2017}. We also provide some important results on the approximation of fractal operators to discrete subspaces with application to discrete versions of the deterministic algorithm for fractal image generation in the cases of IFS recovering the classical images from Barnsley et al., Fuzzy IFS from \cite{Cabrelli-1992} and GIFS's from \cite{Jaros-2016}.
\end{abstract}



\section{Introduction}

Fractals theory is based on the Hutchinson-Barnsley theorem from early 80's (\cite{BOOK:9592}, \cite{HUT}). Initiated by Miculescu and Mihail \cite{MIHAIL-2007} 2008 (see also \cite{mihail2008recurrent}, \cite{mihail2010generalized} and  \cite{strobin_swaczyna_2013}) considered selfmaps of a $X$, they considered mappings defined on finite Cartesian product of a metric space $X$ (they called systems of such mappings as \emph{generalized IFSs}, GIFSs for short). It turns out that such systems of mappings generate sets which can be considered as fractal sets, and many parts of classical theory have natural counterparts in such a framework. What is also important, the class of GIFS's fractals is essentially wider than the class of classical IFSs' fractals (see Strobin \cite{Strobin-2015}).

Cabrelli, Forte, Molter and Vrscay in 1992,  \cite{Cabrelli-1992}, considered a {fuzzy} version of the theory of iterated function systems (IFS's in short) and their fractals, which now is quite rich and important part of the fractals theory. Inspired by the work of Miculescu and Mihail \cite{MIHAIL-2007} we considered, in 2017, the fuzzy version of GIFs's in \cite{Oliveira-2017} obtaining the analogous of the Barnsley-Hutchinson theorem for this setting claiming that there exist a generalized fuzzy fractal attractor. The remaining question was to provide some procedure to approximate this fuzzy sets by a finite iterative process, that is, an algorithm to draw these images.

Through the years several papers has been published providing algorithms generating attractors  for IFS, see \cite{Amo-2003}, \cite{Chitcescu-2010},\cite{Dubuc-1990}, \cite{Yang_1994} etc, for fuzzy IFS, see \cite{Cabrelli-1992}, etc, and lately for GIFS see \cite{Jaros-2016}, \cite{Miculescu_2019}, etc.

The first problem we had to solve was to find an appropriate discretization of a GIFS an of fuzzy sets. The main difficult on such algorithms for fuzzy GIFS is that the Zadeh's extension principle require the solution of a max-min problem with constraints at each point of the discretization, just to do a single iteration of the fractal operator, which is a monstrous task computationally speaking.

The paper is organized as follow:\\
On the Section~\ref{sec:Bas Hut--Barn theory} we recall some basics on the classical Hutchinson-Barnsley theory, on the Section~\ref{sec:Fuzzyfication of the Hutchinson--Barnsley} we explain the  fuzzy counterpart of the Hutchinson-Barnsley theory (fuzzyfication of IFS) and finally in Section~\ref{sec:discret fixed point thm} we introduce a fairly wide applicable theory to discretization of fixed point theorems for Banach contraction maps on a metric space with respect to a subset. This results will be of fundamental importance in the following sections.  In Section~\ref{sec:Discretization of sets and fuzzy sets} we describe the discretization process for a fuzzy set, which is of main importance to our work.

As a consequence of this preparatory theory we were able to produce discrete version of the deterministic algorithm to generate the attractor of IFS, Section~\ref{sec:IFSDraw} using Theorem~\ref{tt2}, and fuzzy IFS, Section~\ref{sec:FuzzyIFSDraw} using Theorem~\ref{tt3}.  Computational examples are given in Section~\ref  {Examples of discrete IFS fractals} and Section~\ref{Fuzzy IFS examples}.

The final part is to consider the GIFS and the fuzzy GIFS in Section~\ref{sec:Generalized IFSs and their fuzzyfication} and its discretization in the Section~\ref{sec:GIFS and GIFZS discretization}.  Then, in Section~\ref{sec:FuzzyIFSandGIFSDraw} we present the discrete version of the deterministic algorithm to both cases, exhibiting examples of known GIFS fractals being recovered by our algorithm and for the last new fuzzy GIFS attractors as we proposed. Computational examples are given in Section~\ref{sec:FuzzyIFSandGIFSDraw}.

\section{Basics of the Hutchinson--Barnsley theory}\label{sec:Bas Hut--Barn theory}

Let $(X,d)$ be a metric space.
We say that $f:X\to X$ is a \emph{Banach contraction}, if the Lipschitz constant $\on{Lip}(f)<1$.
The classical Banach fixed point theorem states that each Banach contraction on a complete metric space has a unique fixed point $x_*$, and for every $x_0\in X$, the sequence of iterates $(f^k(x_0))_{k=0}^{\infty}$  converges to $x_*$.
\begin{definition}\emph{
An }iterated function system\emph{ (IFS in short)  $\mS=(X,(\phi_j)_{j=1}^{L})$ consists of a finite family  $\phi_1,...,\phi_L$ of continuous selfmaps of $X$. Each IFS $\mS$ generates the map $F_\mS:\K^*(X)\to\K^*(X)$ (where $\K^{*}(X)$ denotes the family of all nonempty and compact subsets of $X$), called }the Hutchinson operator\emph{, defined by
$$
\forall_{K\in\K^*(X)}\;F_\mS(K):=\bigcup_{j=1}^L\phi_j(K).
$$
By the \emph{attractor} of an IFS $\mS$ we mean the unique set $A_\mS\in \K^*(X)$ which satisfies
$$
A_\mS=F_\mS(A_\mS)=\bigcup_{j=1}^L\phi_j(A_\mS)
$$
and such that for every $K\in\K^*(X)$, the sequence of iterates $(F_\mS^k(K))_{k=0}^{\infty}$ converges to $A_\mS$ with respect to the Hausdorff metric $h$ on $\K^*(X)$.
}\end{definition}
The classical Hutchinson--Barnsley theorem \cite{BOOK:9592}, \cite{HUT} states that
each IFS $\mS$ consisting of Banach contractions on a complete metric space $X$ admits the attractor.
This result can be proved with a help of the Banach fixed point theorem as it turns out that $F_\mS$ is a Banach contraction provided each $\phi_j$ is a Banach contraction.
\begin{lemma}\label{lem3}
Let $(X,d)$ be a metric space and $\mS=(X,(\phi_j)_{j=1}^{L})$  be an IFS consisting of Banach contractions. Then $F_\mS$ is a Banach contraction and $\on{Lip}(F_\mS)\leq\max\{\on{Lip}(\phi_j):j=1,...,L\}$.
\end{lemma}

Given an IFS $\mS=(X,(\phi_j)_{j=1}^L)$ consisting of Banach contractions, we will denote $$\alpha_\mS:=\max\{\on{Lip}(\phi_1),...,\on{Lip}(\phi_L)\}.$$

\section{Fuzzyfication of the Hutchinson--Barnsley theorem}\label{sec:Fuzzyfication of the Hutchinson--Barnsley}

Now we recall the fuzzy version of the Hutchinson--Barnsely theorem. Its particular version can be found in \cite{Cabrelli-1992}, but the presented general case can be deduced from \cite{Oliveira-2017}.
Let $X$ be a set; we recall that $u$ is a fuzzy subset of $X$ if $u$ is a function from $X$ to $[0,1]$. The family of fuzzy subsets of $X$ is denoted by $\mathcal{F}_{X}$.  It means that each point $x$ has a grade of membership  $0\leq u(x)\leq 1$ in the set $u$. Here, $u(x)=0$ indicates that $x$ is not in $u$ and $u(x)=0.4$ indicates that $x$ is a member of $u$ with membership degree $0.4$.
\begin{figure}[!ht]
 \centering
 \includegraphics[width=2cm]{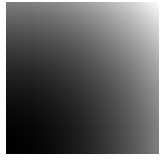}
  \caption{Representation of the fuzzy set $u(x,y)=1/2 (x^2 + y^2)$ in $X=[0,1]^2$ as a grey scale image.  In this case $u(1,1)=1$=a white pixel and $u(0,0)=0$=a black pixel.}\label{example fuzzy set}
\end{figure}

Given a fuzzy set $u:X\to[0,1]$ and $\alpha\in[0,1]$, its $\alpha$-cut is the set
$$[u]^\alpha:=\left\{\begin{array}{ccc}\{x\in X \, | \, u(x)\geq\alpha\}&\mbox{if}&\alpha>0\\\on{cl}(\{x\in X\, | \, u(x)>0\})&\mbox{if}&\alpha=0\end{array}\right.,$$
where $\on{cl}(A)$ denotes the closure of $A$.

To make this theory work we need to restrict $\mathcal{F}_{X}$ to a smaller family,
$$\mathcal{F}_{X}^* := \{ u \in \mathcal{F}_{X} \; | u \text{ is normal, upper semicontinuous and compactly supported}\}. \; $$
We recall that $u:X\to[0,1]$ is:\\
- \emph{upper semicontinuous} (usc) if for any $\alpha>0$, the set $[u]^\alpha$ is closed;\\
- \emph{compactly supported}, if the set $\on{supp}(u):=[u]^0$ is compact;\\
- \emph{normal}, if $u(x)=1$ for some $x \in X$.

In particular, for $u\in\F^*_X$, all $\alpha$-cuts are nonempty and compact.
The metric that makes $\mathcal{F}_{X}^*$ complete is the {version of the} Haussdorf-Pompeiu metric. Since $\K^*(X)$ contains all the $\alpha$-cuts, we can define the distance $d_\infty$ in $\mathcal{F}_{X}^*$ by
$$d_\infty (u,v) := \sup_{\alpha \in [0,1]} h([u]^{\alpha},[v]^{\alpha}),$$
for $u, v \in \mathcal{F}_{X}^*$.
It is known that $d_\infty$ is a metric, which is complete provided $X$ is {complete} (see \cite{Oliveira-2017} or Diamond and Kloeden~\cite{diamond1994metric} for a somewhat more restrictive version). Moreover, as we proved in \cite[Cor. 2.4]{Oliveira-2017},
\begin{equation}\label{dinfty}d_\infty (u,v) := \sup_{\alpha \in (0,1]} h([u]^{\alpha},[v]^{\alpha}).\end{equation}
We also recall the extension principle,
\begin{definition}\label{defext}(Zadeh's Extension Principle) Given a map $\phi: X \to Y$, $ u \in \mathcal{F}_{X}$ we define a new fuzzy set $\phi(u)\in \mathcal{F}_{Y}$ as follows:
{$$\phi(u)(y) :=\sup\{u(x):x\in \phi^{-1}(y)\},$$
under the additional assumption that $\sup\emptyset=0$.}
\end{definition}
It is worth to mention, that the above definitions somehow extends the setting of the hyperspace $\K^*(X)$ with the Hausdorff metric. Indeed, a set $K\subset X$ is an element of $\K^*(X)$ iff the characteristic function $\chi_K\in \F^*_X$, and for $K,D\in\K^*(X)$, it holds $h(K,D)=d_\infty(\chi_K,\chi_D)$. In other words, the map $\K^*(X)\ni K\to \chi_K\in\F^*_X$ is an isometric embedding. Moreover, if $\phi:X\to Y$ and $K\in\K^*(X)$, then $\phi(\chi_K)=\chi_{\phi(K)}$.

A system of grey level maps $(\rho_j)_{j=1}^{L}: [0,1]\to [0,1]$  is admissible if it satisfies all the conditions\\
a) $\rho_j$ is nondecreasing;\\
b) $\rho_j$ is right continuous;\\
c) $\rho_j(0)=0$;\\
d) $\rho_j(1)=1$ for some $j$.

\begin{definition} \label{IFZS definition}
An iterated fuzzy function system \emph{(IFZS in short) $\mathcal{S}:=(X, (\phi_j)_{j=1}^{L}, (\rho_{j})_{j=1}^{L})$ consists of an IFS  $(X, (\phi_j)_{j=1}^{L}),$ with a set of admissible grey level maps $(\rho_j)_{j=1}^{L}$.\\
The operator $Z_\mathcal{S}:  \mathcal{F}_{X}^{*} \to \mathcal{F}_{X}^{*}$ defined by
$${Z_\mathcal{S}({u})}:= \bigvee_{j \in \{1,...,L\}} \rho_{j}(\phi_{j}(u)){:=\max\{ \rho_{j}(\phi_{j}(u)):j=1,...,L\}}$$
is called }the fuzzy {Hutchinson operator (FH) associated to $\mathcal{S}$.\\
A fuzzy set $u_\mS\in \mathcal{F}_{X}^{*}$  is called the }fuzzy fractal of $\mathcal{S}$ if $Z_\mathcal{S}(u_\mS)=u_\mS$, that is
$$u_\mS= \bigvee_{j \in \{1,...,L\}} \rho_{j}(\phi_{j}(u_\mS)),$$
and for every $u\in\mathcal{F}_X^*$, the sequence of iterates $(Z^k_\mathcal{S}(u))$ converges to $u_\mathcal{S}$ with respect to the metric $d_\infty$.
\end{definition}

The following result is a consequence of \cite[Thm. 3.15]{Oliveira-2017} (see Diamond and Kloeden~\cite{diamond1994metric} for more restrictive version).
\begin{lemma}\label{lem3ff}
Let $(X,d)$ be a metric space and $\mathcal{S}=(X, (\phi_j)_{j=1}^{L}, (\rho_{j})_{j=1}^{L})$ be an IFZS consisting of Banach contractions, then $Z_\mathcal{S}$ is a Banach contraction with $\on{Lip}(Z_\mS)\leq\max\{\on{Lip}(\phi_j):j=1,...,L\}$.
\end{lemma}

Given an IFZS $\mS=(X,(\phi_j)_{j=1}^L,(\rho_j)_{j=1}^L)$ consisting of Banach contractions, we will denote $\alpha_\mS:=\max\{\on{Lip}(\phi_1),...,\on{Lip}(\phi_L)\}$.

\section{Discretization of fixed point theorems}\label{sec:discret fixed point thm}

\begin{definition}\emph{
A subset ${\hat{X}}$ of a metric space $(X,d)$ is called an }$\ve$-net\emph{ of $X$, if for every $x\in X$, there is $y\in {\hat{X}}$ such that $d(x,y)\leq \ve$.\\
A map $r:X\to {\hat{X}}$ such that $r(x)=x$ for $x\in {\hat{X}}$ and $d(x,r(x))\leq\ve$ for all $x\in X$ will be called an }$\ve$-projection\emph{ of $X$ to $\hat{X}$.\\
For $f:X\to X$, by its} $r$-discretization\emph{ we will call the map $\hat{f}:=(r\circ f)_{\vert \hat{X}}$
}
\end{definition}
Clearly, for each $\ve$-net $\hat{X}$ of $X$, an $\ve$-projection exists, but it need not be unique.
The following result can be considered as a discrete version of the Banach fixed point theorem.

\begin{theorem}\label{dfp}
Assume that $(X,d)$ is a complete metric space and $f:X\to X$ is a Banach contraction with the unique fixed point $x_*$ and a Lipschitz constant $\alpha$. Let $\ve>0$, ${\hat{X}}$ be an $\ve$-net, $r:X\to{\hat{X}}$ be an $\ve$-projection and $\hat{f}$ be an $r$-discretization of $f$.

For every $x\in {\hat{X}}$ and $n\in\N$,
\begin{equation}\label{approx}
d(\hat{f}^n(x),x_*)\leq\frac{5\ve}{1-\on{Lip}(f)}+\alpha^nd(x,x_*).
\end{equation}
In particular, there exists a point $y\in X$ so that $d(x_*,y)\leq \frac{6\ve}{1-\on{Lip}(f)}$ and which can be reached as an appropriate iteration of $\hat{f}$ of an arbitrary point of $\hat{X}$.

\end{theorem}
\begin{proof}
Set $\alpha:=\on{Lip}(f)$. We first prove that
\begin{equation}\label{red2}
\forall_{x,y\in \hat{X}}\;d(\hat{f}(x),\hat{f}(y))\leq 2\ve+\alpha d(x,y).
\end{equation}

For $x,y\in\hat{X}$, we have
$$
d(\hat{f}(x),\hat{f}(y))\leq d(\hat{f}(x),f(x))+d(f(x),f(y))+d(f(y),\hat{f}(y))\leq$$ $$\leq d(r(f(x)),f(x))+\alpha d(x,y)+d(f(y),r(f(y)))\leq 2\ve+\alpha d(x,y).
$$

Consider two cases:\\
Case 1: $x_*\in {\hat{X}}$.\\
By a  simple induction, we can show that for every $x,y\in {\hat{X}}$ and $n\in\N$,
$$
d(\hat{f}^n(x),\hat{f}^n(y))\leq 2\ve\sum_{i=0}^{n-1}\alpha^i+\alpha^nd(x,y)
$$
As $r(x_*)=x_*$, we have
$$\hat{f}(x_*)=r(f(x_*))=r(x_*)=x_*$$
so $x_*$ is a fixed point of $\hat{f}$. Hence for every $x\in {\hat{X}}$ and $n\in\N$,
$$
d(\hat{f}^n(x),x_*)=d(\hat{f}^n(x),\hat{f}^n(x_*))\leq2\ve\sum_{i=0}^{n-1}\alpha^i+\alpha^nd(x,x_*)\leq \frac{2\ve}{1-\alpha}+\alpha^nd(x,x_*),
$$
and we are done.\\
Case2: $x_*\notin {\hat{X}}$.\\
Set $\tilde{X}:={\hat{X}}\cup\{x_*\}$, and extend $r$ to a map $\tilde{r}:\tilde{X}\to\tilde{X}$ by setting $\tilde{r}(x_*)=x_*$, and the map $\hat{f}$, to the map $\tilde{f}:\tilde{X}\to\tilde{X}$ by setting $\tilde{f}(x_*)=x_*$.\\
Now observe that for every $x,y\in \tilde{X}$,
\begin{equation}\label{11}
d(\tilde{f}(x),\tilde{f}(y))\leq 5\ve+\alpha d(x,y)
\end{equation}
Indeed, if $x,y\neq x_*$, then it follows from that fact that $\tilde{f}(x)=\hat{f}(x)$,  $\tilde{f}(y)=\hat{f}(y)$, and earlier calculations. If $x\neq x_*$, then, choosing $x'\in {\hat{X}}$ so that $d(x',x_*)\leq\ve$, we have
$$
d(\tilde{f}(x_*),\tilde{f}(x))\leq d(\tilde{f}(x_*),f(x'))+d(f(x'),\tilde{f}(x'))+d(\tilde{f}(x'),\tilde{f}(x))\leq$$
$$\leq
d(f(x_*),f(x'))+d(f(x'),r(f(x'))+d(\hat{f}(x'),\hat{f}(x))\leq
d(x_*,x')+\ve+2\ve+\alpha d(x',x)\leq $$ $$\leq 4\ve+\alpha d(x,x_*)+\alpha d(x_*,x')\leq 5\ve+\alpha d(x,x_*).
$$
We proved (\ref{11}). Using simple inductive argument, we can prove that for every $n\in\N$,
\begin{equation}\label{2}
d(\tilde{f}^n(x),\tilde{f}^n(y))\leq 5\ve\sum_{i=0}^{n-1}\alpha^i+\alpha^n d(x,y)
\end{equation}
Now if $x\in {\hat{X}}$, then for every $n\in\N$, $\tilde{f}^n(x)=\hat{f}^n(x)$, so by (\ref{2}) we have
$$
d(\hat{f}^n(x),x_*)=d(\tilde{f}^n(x),\tilde{f}^n(x_*))\leq 5\ve\sum_{i=0}^{n-1}\alpha^i+\alpha^nd(x,x_*)\leq \frac{5\ve}{1-\alpha}+\alpha^nd(x_*,x),
$$
and the proof is finished.
\end{proof}

\section{Discretization of sets and fuzzy sets}\label{sec:Discretization of sets and fuzzy sets}

\begin{definition}\emph{
We say that an $\ve$-net $\hat{X}$ of a metric space $X$ is }proper\emph{, if for every bounded $D\subset X$, the set $D\cap\hat{X}$ is finite.}
\end{definition}
Note that proper $\ve$-nets are discrete (as topological subspaces), but the converse need not be true. For example, there exists an infinite subset $E$ of a unit sphere in an infinite dimensional normed space, so that $||x-y||=1$ for all $x,y\in E$, $x\neq y$.\\
The existence of proper $\ve$-nets for every $\ve>0$ is quaranteed by the assumption that $X$ has so-called \emph{Heine--Borel property}, that is, the assumption that each closed and bounded set is compact. In particular, Euclidean spaces and compact spaces admit such nets.

Now assume that $(X,d)$ is a metric space and ${\hat{X}}$ is a proper $\ve$-net. Clearly, $\K({\hat{X}})$ consists of all finite subsets of ${\hat{X}}$. Now if $r:X\to{\hat{X}}$ is an $\ve$-projection, then by the same letter $r$ we will denote the map $r:\K^*(X)\to\K({\hat{X}})$ defined by $r(K)=\{r(x):x\in K\}$. Finally, let $\mS=(X,(\phi_j)_{j=1}^L)$ be an IFS.

It is routine to check that:
\begin{lemma}\label{lemm2}In the above frame\\
(a) $\K({\hat{X}})$ is an $\ve$-net of $\K^*(X)$;\\
(b) $r$ is an $\ve$-projection of $\K^*(X)$ to $\K({\hat{X}})$;\\
(c) $F_{\hat{\mS}}=\hat{F}_\mS$, where $\hat{F}_\mS=(r\circ F_\mS)_{\vert \K^*(\hat{X})}$ is the discretization of $F_\mS$, and $\hat{\mS}:=(X,(\hat{\phi}_j)_{j=1}^L)$ and $\hat{\phi}_j=(r\circ \phi_j)_{\vert \hat{X}}$ is the discretization of $\phi_j$.
\end{lemma}

Now we extend the considerations to fuzzy sets.
Let $(X,d)$ be a metric space, $\ve>0$, ${\hat{X}}$ be a proper $\ve$-net and $r:X\to {\hat{X}}$ be an $\ve$-projection. By the same letter $r$ we will denote the map which adjust to each $u\in \F^*_X$, the fuzzy set $r(u):X\to[0,1]$
defined by
$$
\forall_{x\in X}\;r(u)(x):=\sup\{u(y)|y\in r^{-1}(x)\},
$$
Observe that $r(u)$ is defined as in Definition \ref{defext}, when considering the map $r$ as the map $r:X\to \hat{X}$.\\
For a fuzzy set $u:\hat{X}\to [0,1]$, let $e(u):X\to[0,1]$ be the natural extension of $u$, that is, $\tilde{u}$ is defined by: be the fuzzy set on $X$ defined by
$$e(u)(x)=\left\{\begin{array}{cl}u(x),&\mbox{if}\;\;x\in\hat{X}\\0,&\mbox{if}\;\;x\in X\setminus\hat{X}\end{array}\right.$$
Then let $\tilde{\F}^*_{\hat{X}}:=\{e(u):u\in \F^*_{\hat{X}}\}$ and $\mS=(X,(\phi_j)_{j=1}^L,(\rho_j)_{j=1}^L)$ be an IFZS.

\begin{lemma}\label{lemma4}
In the above frame:\\
(a) $\tilde{\F}^*_{{\hat{X}}}=\{u\in\F^*_X\;|\;u\mbox{ is normal and }\{x:u(x)>0\}\mbox{ is finite and contained in }\hat{X}\}$;\\
(b) the map $\F^*_{\hat{X}}\ni u\to e(u)\in\tilde{\F}^*_{\hat{X}}$ is an isometric embedding;\\
(c) $\tilde{\F}^*_{{\hat{X}}}$ is an $\ve$-net of $\F^*_X$;\\
(d) the map $r:\F^*_X\to\tilde{\F}^*_{\hat{X}}$ is an $\ve$-projection;\\
(e) for any $u\in\tilde{\F}^*_{\hat{X}}$, it holds $e(Z_{\hat{\mS}}(u_{\vert \hat{X}}))=\hat{Z}_\mS(u)$, where $\hat{Z}_\mS=(r\circ Z_\mS)_{\vert \tilde{\F}^*_{\hat{X}}}$ is the discretization of $Z_\mS$, and $\hat{\mS}:=(\hat{X},(\hat{\phi}_j)_{j=1}^L,(\rho_j)_{j=1}^L)$ and $\hat{\phi}_j=(r\circ \phi_j)_{\vert \hat{X}}$ is the discretization of $\phi_j$.

\end{lemma}
\begin{proof}
(a) Take $u\in\tilde{\F}^*_{\hat{X}}$, that is, $u:X\to[0,1]$, $u(x)=0$ for $x\in X\setminus \hat{X}$ and $u_{\vert\hat{X}}:\hat{X}\to [0,1]$ is normal, usc and compactly supported. Since $\hat{X}$ is discrete, compact supportedness implies that $\{x\in X\;|\;u(x)>0\}=\{x\in\hat{X}:u(x)>0\}$ is finite. This also implies that $u$ is usc (all $\alpha$-cuts are finite), so we also have that $u\in\F^*_X$.\\
Now assume that $u:X\to [0,1]$ is normal and $\{x\in X\;|\;u(x)>0\}$ is finite and contained in $\hat{X}$. Then, obviously, $u_{\vert \hat{X}}:\hat{X}\to[0,1]$ is normal, compactly supported and usc, i.e., $u_{\hat{X}}\in\F^*_{\hat{X}}$, so $u\in\tilde{\F}^*_{\hat{X}}$.

Point (b) is obvious as for every $\alpha\in[0,1],\;[u]^\alpha=[e(u)]^\alpha$.

We will prove points (c) and (d) mutually.\\
Let $u:X\to[0,1]$ be an element of $\F^*_X$. We will prove that $r(u)$ is an element of $\tilde{\F}^*_{\hat{X}}$. Let $y\in X$ be such that $u(y)=1$, and let $x\in \hat{X}$ be such that $r(y)=x$. Then $1\geq r(u)(x)\geq u(y)=1$, so $r(u)(x)=1$ and $r(u)$ is normal. Now suppose that $M:=\{x\;|\;r(u)(x)>0\}$ is infinite. Then it must be unbounded as $M\subset \hat{X}$. For every $x\in M$, there is $y_x\in X$ so that $r(y_x)=x$ and $u(y_x)>0$. As $d(y_x,x)\leq\ve$, the set $\{y_x:x\in M\}$ must be unbounded. Since $\{y_x\;|\;x\in M\}\subset \{y\;|\;u(y)>0\}$, this gives a contradiction with the fact that $u$ is compactly supported. Hence $r(u)\in\tilde{\F}^*_{\hat{X}}$.\\
It remains to prove that $d_\infty(u,r(u))\leq\ve$. Observe that for every $\alpha\in(0,1]$,
$$
[r(u)]^\alpha=\{x\;|\;r(u)(x)\geq\alpha\}=\{x\;|\;\sup\{u(y):y\in r^{-1}(x)\}\geq \alpha\}\subset$$ $$\subset\{x\;|\;\sup\{u(y)|y\in \on{cl}({r^{-1}(x)})\}\geq \alpha\}=$$ $$= \{x\;|\;\sup\{u(y)\;|\;y\in \on{cl}({r^{-1}(x)})\cap\on{supp}(u)\}\geq \alpha\}=$$ $$=  \{x\;|\;\max\{u(y)\;|\;y\in \on{cl}({r^{-1}(x)})\cap\on{supp}(u)\}\geq \alpha\},
$$
where the last equality holds as $u$ is usc and $\on{cl}({r^{-1}(x)})\cap\on{supp}(u)$ is compact.\\
Hence let $x\in [r(u)]^\alpha$. By the above computations, there is $y\in \on{cl}({r^{-1}(x)})$ such that $u(y)\geq \alpha$, so $y\in[u]^\alpha$. Then choose $(y_n)\subset r^{-1}(x)$ so that $y_n\to y$. Since $d(y_n,x)=d(y_n,r(y_n))\leq\ve$, we also have $d(x,y)\leq\ve$.\\
Convesely, choose $y\in [u]^\alpha$. Then $d(x,y)\leq\ve$ for $x=r(y)$, and hence $r(u)(x)\geq u(y)\geq\alpha$, so $x\in[r(u)]^\alpha$. All in all, $h([u]^\alpha,[r(u)]^\alpha)\leq\ve$ and by (\ref{dinfty}), $d_\infty(u,r(u))\leq\ve$.\\
Now we prove (e). Take $u\in\tilde{\F}^*_{\hat{X}}$ and observe that for all $x\in X$,
$$
\hat{Z}_\mS(u)(x)=r(Z_\mS(u))=\sup\{Z_\mS(u)(y)\;|\;y\in r^{-1}(x)\}=$$ $$=\sup\{\max\{\rho_j(\phi_j(u)(y))\;|\;j=1,...,L\}\;|\;y\in r^{-1}(x)\}=$$ $$=\sup\{\rho_j(\sup\{u(z)\;|\;z\in \phi_j^{-1}(y)\})\;|\;j=1,...,L,\;y\in r^{-1}(x)\}=
$$
$$=\sup\{\rho_j(\sup\{u(z)\;|\;z\in \phi_j^{-1}(y)\cap \on{supp}(u)\})\;|\;j=1,...,L,\;y\in r^{-1}(x)\}\overset{\on{supp}(u)\mbox{ is finite}}{=}
$$
$$=\sup\{\rho_j(\max\{u(z)\;|\;z\in \phi_j^{-1}(y)\cap \on{supp}(u)\})\;|\;j=1,...,L,\;y\in r^{-1}(x)\}\;\;\overset{\rho_j\mbox{ is nondecreasing}}{=}
$$
$$=\sup\{\rho_j(u(z))\;|\;j=1,...,L,\;y\in r^{-1}(x),\;z\in \phi_j^{-1}(y)\cap \on{supp}(u)\}\;\;\overset{\rho_j(0)=0}{=}
$$
$$=\sup\{\rho_j(u(z))\;|\;j=1,...,L,\;y\in r^{-1}(x),\;z\in \phi_j^{-1}(y)\}=
$$
$$=\sup\{\rho_j(u(z))\;|\;j=1,...,L,\;z\in \phi_j^{-1}(r^{-1}(x))\}=\sup\{\rho_j(u(z))\;|\;j=1,...,L,\;z\in (r\circ \phi_j)^{-1}(x)\}.
$$
In particular, if $x\notin \hat{X}$, then $\hat{Z}_\mS(u)(x)=0$.

Now we switch to calculating $Z_{\hat{\mS}}(u_{\vert \hat{X}})$.
Since $u(z)=0$ for $z\notin \hat{X}$, it is easy to see that for $x\in\hat{X}$ and $j=1,...,L$,
\begin{equation}\label{reduct}
\sup\{u_{\vert\hat{X}}(z):z\in(r\circ\phi_j)^{-1}_{\vert\hat{X}}(x)\}=\sup\{u(z):z\in(r\circ\phi_j)^{-1}(x)\}.
\end{equation}
Now for every $x\in \hat{X}$, we have
$$
Z_{\hat{\mS}}(u_{\vert \hat{X}})(x)=\sup\{\rho_j((r\circ \phi_j)_{\vert \hat{X}}(u_{\vert \hat{X}})(x))\;|\;j=1,...,L\}=
$$
$$
=\sup\{\rho_j(\sup\{u_{\vert\hat{X}}(z)\;|\;z\in(r\circ \phi_j)_{\vert\hat{X}}^{-1}(x)\})\;|\;j=1,...,L\}\;\;\overset{(\ref{reduct})}{=}
$$
$$
=\sup\{\rho_j(\sup\{u(z)\;|\;z\in(r\circ \phi_j)^{-1}(x)\})\;|\;j=1,...,L\}=
$$
$$
=\sup\{\rho_j(\max\{u(z)\;|\;z\in(r\circ \phi_j)^{-1}(x)\cap\on{supp}(u)\})\;|\;j=1,...,L\}\;\;\overset{\rho_j\mbox{ is nondecreasing}}{=}
$$
$$
=\sup\{\rho_j(u(z))\;|\;j=1,...,k,\;z\in(r\circ \phi_j)^{-1}(x)\cap\on{supp}(u)\}\;\;\overset{\rho_j(0)=0}{=}
$$
$$
=\sup\{\rho_j(u(z))\;|\;j=1,...,k,\;z\in(r\circ \phi_j)^{-1}(x)\}.
$$
All in all, $e(Z_{\hat{S}}(u_{\vert \hat{X}}))=\hat{Z}_\mS(u)$.

\end{proof}

\section{IFS and IFZS discretization}\label{sec:IFS and IFZS discretization}

\begin{definition}\emph{
Given an IFS $\mS$ on a metric space $X$ with the attractor $A_\mS$ and $\delta>0$, a set $A_\delta\in\K^*(X)$ will be called an attractor of $\mS$} with resolution $\delta$\emph{, if $h(A_\delta,A_\mS)\leq\delta$.}
\end{definition}
Lemma \ref{lem3}, Theorem \ref{dfp} used for the Hutchinson operator and Lemma \ref{lemm2} imply the following ``discrete" version of the Hutchinson--Barnsley theorem.
\begin{theorem}\label{tt2}
Let $(X,d)$ be a complete metric space and $\mS=(X,(\phi_j)_{j=1}^L)$ be an IFS on $X$ consisting of Banach contractions. Let $\ve>0$, ${\hat{X}}$ be a proper $\ve$-net, $r:X\to {\hat{X}}$ be an $\ve$-projection on ${\hat{X}}$ and $\hat{\mS}:=(\hat{X},(\hat{\phi}_j)_{j=1}^L)$, where $\hat{\phi}_j=(r\circ \phi_j)_{\vert{\hat{X}}}$ is the discretization of $\phi_j$.\\
For any $K\in\K({\hat{X}})$ and $n\in\N$,
\begin{equation}\label{approx2}
h(F_{\hat{\mS}}^n(K),A_\mS)\leq\frac{5\ve}{1-\alpha_\mS}+\alpha_\mS^nh(K,A_\mS),
\end{equation}
where $A_\mS$ is the attractor of $\mS$.\\
In particular, there is $n_0\in\N$ such that for every $n\geq n_0$, $F_{\hat{\mS}}^n(K)$ is an attractor of $\mS$ with resolution $\frac{6\ve}{1-\alpha_\mS}$.
\end{theorem}
Let us explain the thesis. Starting with an IFS $\mS=(X,(\phi_j)_{j=1}^L)$ consisting of Banach contractions, we switch to $\hat{\mS}:=(\hat{X},(\hat{\phi}_j)_{j=1}^{L})$, which is the IFS consisting of discretizations of maps from $\mS$ to $\ve$-net $\hat{X}$. Then, picking any $K\in\K^*(\hat{X})$ (that is, any finite subset of $\hat{X}$), it turns out that the sequence of iterates $(F_{\hat{\mS}}^n(K))_{k=0}^\infty$ (of the Hutchinson operator $F_{\hat{\mS}}$ adjusted to $\hat{\mS}$) gives an approximations the attractor $A_\mS$ of $\mS$ with resolution $\frac{6\ve}{1-\alpha_\mS}$.

Now we formulate and prove a fuzzy version of Theorem \ref{tt2}.

\begin{theorem}\label{tt3}
Let $(X,d)$ be a complete metric space and $\mathcal{S}=(X, (\phi_j)_{j=1}^{L}, (\rho_{j})_{j=1}^{L})$ be an IFZS on $X$ consisting of Banach contractions. Let $\ve>0$, ${\hat{X}}$ be a proper $\ve$-net, $r:X\to {\hat{X}}$ be an $\ve$-projection on ${\hat{X}}$ and $\hat{\mathcal{S}}=(X, (\hat{\phi}_j)_{j=1}^L, (\rho_{j})_{j=1}^L)$, where $\hat{\phi}_j:=(r\circ \phi_{j})_{\vert {\hat{X}}}$ is the discretization of $\phi_j$.\\
Then for any $u\in{\F}^*_{\hat{X}}$ and $n\in\N$,
$$d_\infty(e({Z}_{\hat{\mS}}^n(u)),u_\mS)\leq\frac{5\ve}{1-\alpha_\mS}+\alpha_\mS^nd_\infty(e(u),u_\mS)$$
where $u_\mS$ is the fuzzy attractor of $\mS$.\\
In particular, there is $n_0\in\N$ such that for every $n\geq n_0$, $e({Z}_{\hat{\mS}}^n(u))$ is an attractor of $\mS$ with resolution $\frac{6\ve}{1-\alpha_\mS}$.
\end{theorem}
\begin{proof}
The result follows directly from Lemma \ref{lem3ff}, Theorem \ref{dfp} and Lemma \ref{lemma4}. It is worth to note that the operator $Z_{\hat{\mS}}:\F^*_{\hat{X}}\to\F^*_{\hat{X}}$. This follows from the fact that the IFSZ $\hat{\mS}$ consists of continuous maps $\hat{\phi}_j:\hat{X}\to\hat{X}$.
\end{proof}

Again, we explain the thesis. Starting with an IFZS $\mS$ consisting of Banach contractions, we switch to $\hat{\mS}$, which is the IFZS consisting of discretizations of maps from $\mS$ to the $\ve$-net $\hat{X}$. Then, picking any ``discrete"  fuzzy set $u\in{\F}^*_{\hat{X}}$, it turns out that the sequence of iterates ${Z}_{\hat{\mS}}^n(u)$ (of the  operator ${Z}_{\hat{\mS}}$ adjusted to $\hat{\mS}$) gives an approximations the attractor $u_\mS$ (more formally, the sequence of extensions $e({Z}_{\hat{\mS}}^n(u))$) with resolution $\delta> \frac{6\ve}{1-\alpha_\mS}$.

\section{Discrete version of the deterministic algorithm for IFSs} \label{sec:IFSDraw}

\subsection{Discrete deterministic algorithm for IFSs} Let $\mathcal{S}=(X,(\phi_j)_{j=1}^L)$ be an IFS satisfying the conditions of Theorem~\ref{tt2}. Then we can obtain a fractal with resolution $\delta$, approximating $A_{\mS}$, as follows:

{\tt
\begin{tabbing}
aaa\=aaa\=aaa\=aaa\=aaa\=aaa\=aaa\=aaa\= \kill
     \> \texttt{IFSDraw($\mS$)}\\
     \> {\bf input}: \\
     \> \>  \> $\delta>0$, the resolution.\\
     \> \>  \> $K \subseteq \hat{X}$, any finite and not empty subset. \\
     \> \>  \> The diameter $D$ of a ball in $(X,d)$ containing $A_{\mS}$ and $K$.\\
     \> {\bf output}: A bitmap  image representing a fractal $W$
     with resolution  at most $\delta$.\\
     \> {\bf Compute} $\displaystyle \alpha_\mS:={\rm Lip}(\mS)$\\
     \> {\bf Compute} $\varepsilon >0$ and $N \in \mathbb{N}$ such that $\frac{5\ve}{1-\alpha_\mS}+\alpha_\mS^N \, D < \delta$\\
     \> W:=K \\
     \> {\bf for n from 1 to N do}\\
     \>\>W:=$F_{\hat{S}}$(W)\\
     \> {\bf end do}\\
     \> {\bf return: Print}(W).\\
\end{tabbing}}

\begin{remark} \label{resolution choices}  After $N$ iterations, we get from Theorem~\ref{tt2} that
$h(F_{\hat{\mS}}^n(K),A_\mS)\leq\frac{5\ve}{1-\alpha_\mS}+\alpha_\mS^N \, h(K,A_\mS)$. Since $A_{\mS}$ and $K$ is contained in a ball of diameter $D$,  we obtain $h(K,A_\mS) \leq D$. Hence, to have $h(F_{\hat{\mS}}^n(K),A_\mS)\leq \delta$ it is enough to choose, for a fixed number $0 < \theta < 1$, a number $\ve>0$ such that $\frac{5\varepsilon}{1-\alpha_\mS} \leq \theta \delta$ and  $\alpha_\mS^N \, D \leq (1-\theta)\delta$, which is which is achieved by choosing
$$\varepsilon = (1-\alpha_\mS)\frac{\theta \delta}{5} \text{ and } N= \left\lceil\frac{\log((1-\theta)\delta/D)}{\log(\alpha_{\mS})}\right\rceil.$$
When $(X,d)$ is compact, we can chose $D={\rm diam}(X)$. For example, if $X=[0,1]^2$ then $D=\sqrt{2}$.
\end{remark}

\subsection{Canonical nets and projections on Euclidean spaces}
We suggest to ways of defining nets and projections on Euclidean spaces and their particular subsets. Let $X=[0,1]^d$ be the $d$-dimensional cube considered with the Euclidean metric and fix an $n\in\N$. Set
$$
\hat{X}:=\{(x[i_1],...,x[i_d]):i_1,...,i_d=0,...,n\},
$$
where $x[i]:=\frac{i}{n}$. We call $\hat{X}$ a \emph{uniform net} of $X$. It is routine to check that it is $\frac{\sqrt{d}}{2n}$-net of $X$. Finally, let $r:X\to\hat{X}$ be defined by
$$
r(z_1,...,z_d):=\left(\frac{\lceil nz_1\rceil}{n},...,\frac{\lceil nz_d\rceil}{n}\right),
$$
where $\lceil a \rceil$ denotes the ceiling of $a$. Clearly, $r$ is an $\frac{\sqrt{d}}{2n}$-projection and $r:X\to\hat{X}$. [FS: now here should be Figure 8].
\begin{remark}\emph{
(1) Instead of ceiling $\lceil\cdot\rceil$ in the definition of $r$, we could take floor $\lfloor\cdot\rfloor$ as well, and the effect will be almost the same.\\
(2) In a similar way, we could consider any ``appropriately regular" closed and bounded set, or even the whole Euclidean space $\R^d$. Is the last case there is a problem with the fact that $\hat{X}$ is infinite, but in the presented algorithms for IFSs, only those points from $\ve$-nets which are actually needed are calculated.}
\end{remark}

\subsection{Examples of discrete IFS fractals} \label{Examples of discrete IFS fractals}
We are going to consider the application of the algorithm \texttt{IFSDraw($\mS$)} to approximate some classical fractals.

Since the algorithm \texttt{IFSDraw($\mS$)} require very little computational effort, we are going to use the following setting on this section: the uniform version, starting from a single point, $700\times 700$ pixels, iterating a fairly number of times until the picture stabilizes. We present always four different iterates and the resolution $\delta$ from the formula $\frac{5\ve}{1-\alpha_\mS}+\alpha_\mS^N \, D < \delta$, where $N$ is the last iterate showed, $\alpha:=\alpha_\mS$ is the Lipschitz constante, computed in each example, $D:=\sqrt{2}$ is the diameter of $X$ itself (with respect to the Euclidean space $(\mathbb{R}^2, d_{e})$) and $\ve$ is the spacing in $\hat{X}$.

\begin{example}\label{discreteIFSex1}  This first example is a very well known fractal, the Barnsley Fern or Black spleenwort fern, see \cite{peruggia1993discrete}. The approximation by the algorithm \texttt{IFSDraw($\mS$)} is presented in the Figure~\ref{Barnsley Fern}.  Consider $(\mathbb{R}^2, d_{e})$ a metric space, $X=[0,1]^2$ and the IFS $\phi_1,..., \phi_4: X \to X$ where
\[\mS:
\left\{
  \begin{array}{ll}
      \phi_1(x,y)   & = (0.856 x + 0.0414 y + 0.07, -0.0205 x+ 0.858 y + 0.147)\\
      \phi_2(x,y)   & = ( 0.244 x - 0.385 y + 0.393, 0.176 x + 0.224 y + 0.102)\\
      \phi_3(x,y)   & = ( -0.144 x + 0.39 y + 0.527, 0.181 x + 0.259 y - 0.014 ) \\
      \phi_4(x,y)   & =( 0.486, 0.031 x + 0.216 y + 0.05)\\
  \end{array}
\right.
\]
 \begin{table}[ht]
   \centering
   \begin{tabular}{c|c|c|c|c|c|c}
     \hline
     $\alpha$ & $[a,b]\times[c,d]$ & $D$ & $n$ & $N$ & $\ve$ & $\delta$ \\
     \hline
      $0.8680563766$ & $[0,1]\times[0,1]$ & $\sqrt{2}$ & $700$ & $24$ & $2.5\times 10^{-4}$ & $0.05686141346$ \\
     \hline
   \end{tabular}
   \caption{Resolution data for the uniform picture in Figure \ref{Barnsley Fern}.}\label{table:Barnsley Fern}
 \end{table}

 \begin{figure}[!ht]
  \centering
  \includegraphics[width=3.5cm,frame]{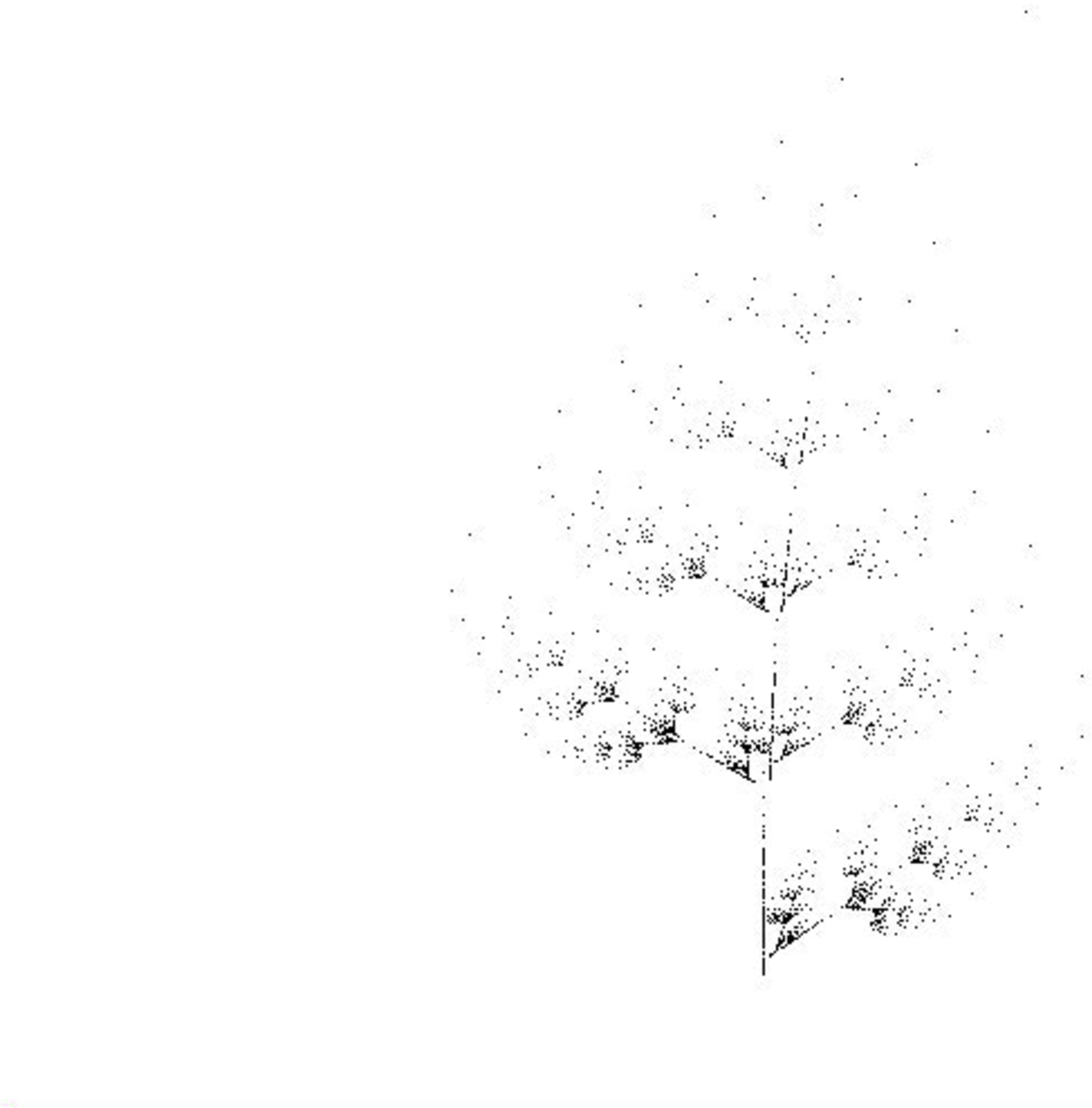}
  \includegraphics[width=3.5cm,frame]{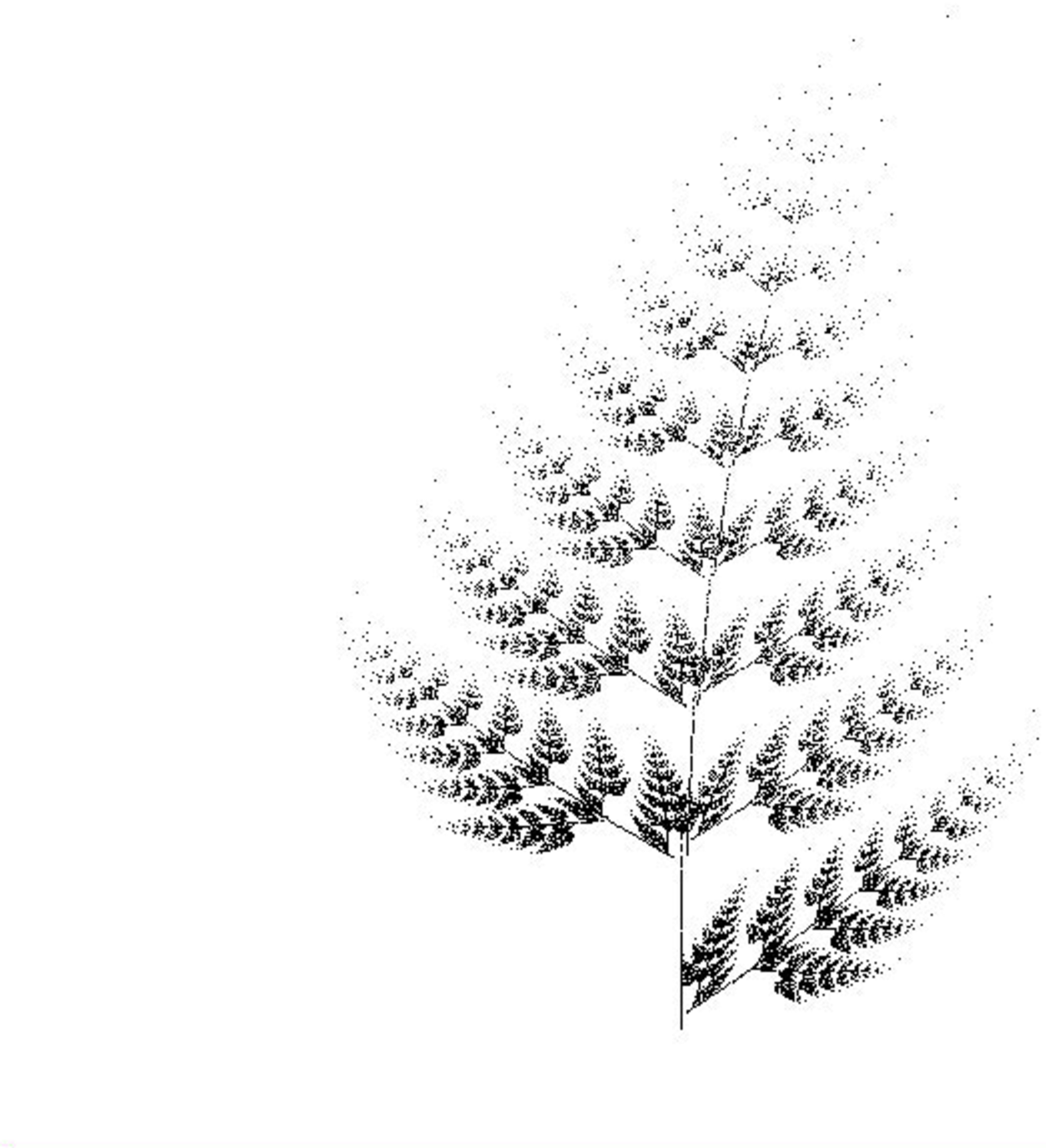}
  \includegraphics[width=3.5cm,frame]{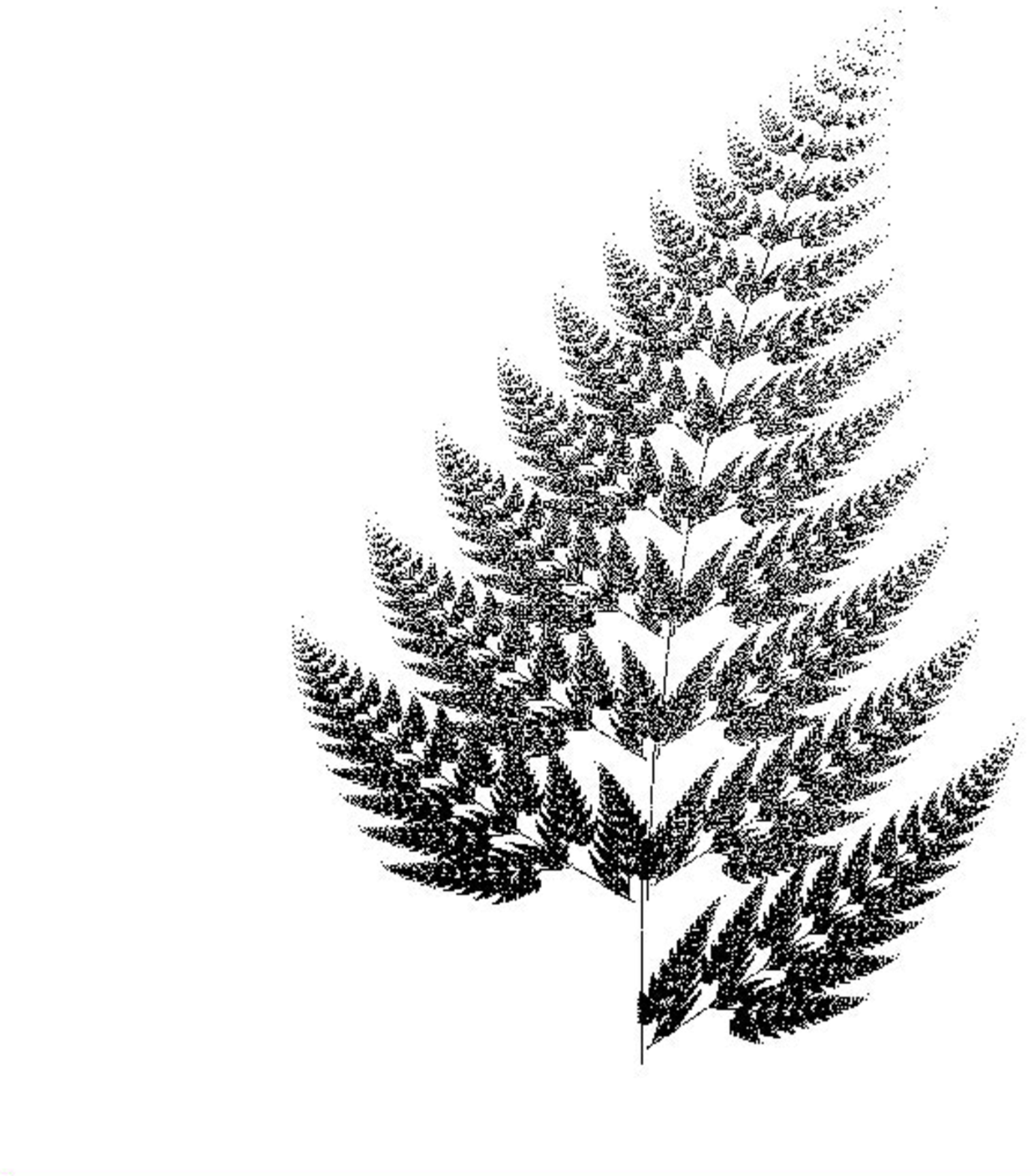}
  \includegraphics[width=3.5cm,frame]{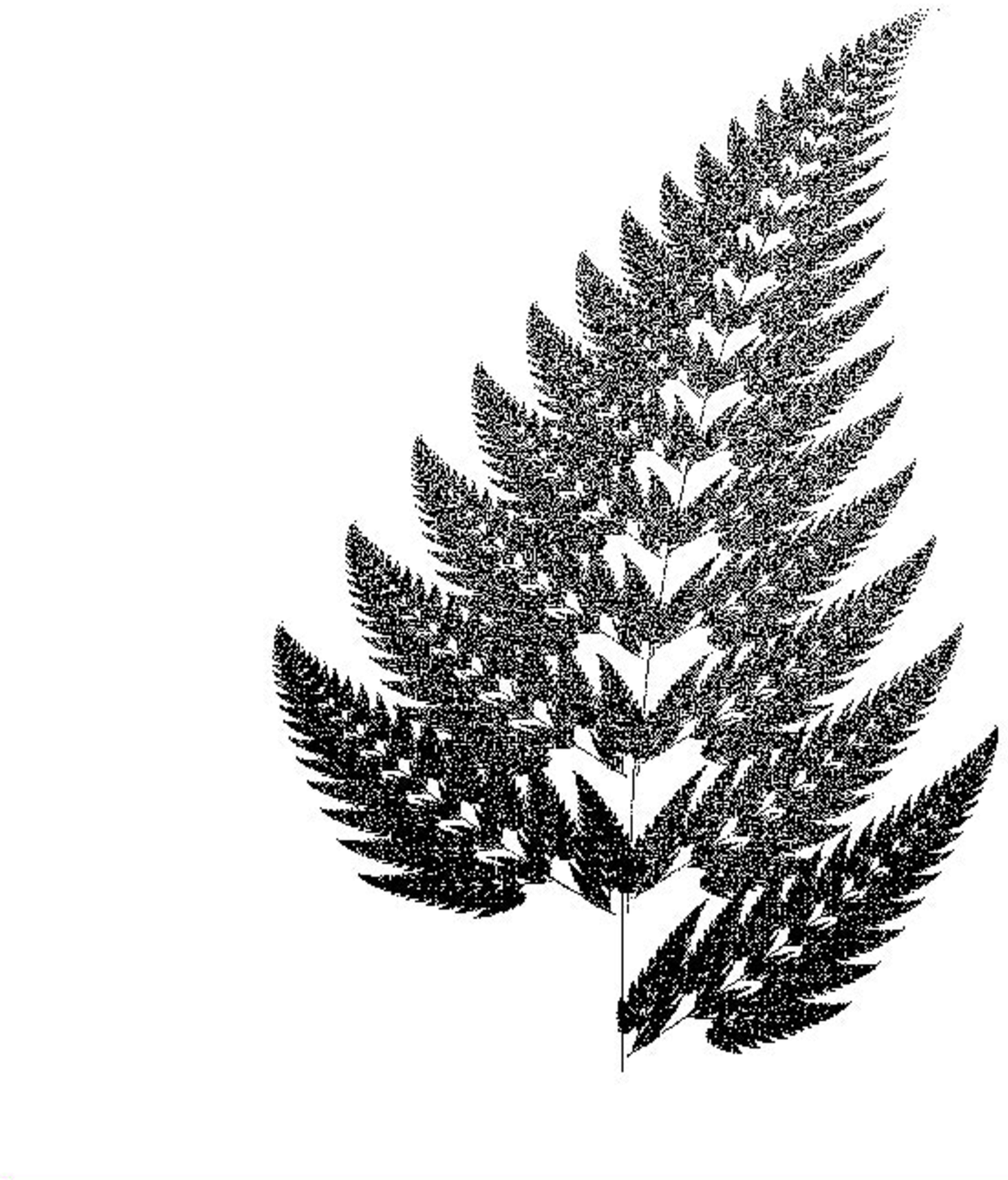}
  \caption{From the left to the right the output of the algorithm \texttt{IFSDraw($\mS$)} after  7, 11, 17 and 24  iterations.}\label{Barnsley Fern}
  \end{figure}

\end{example}


\begin{example}\label{discreteIFSex2} This second example is also a very known fractal, the Sierpinski Triangle, see \cite{peruggia1993discrete}. The approximation by the algorithm \texttt{IFSDraw($\mS$)} is presented in the Figure~\ref{Sierpinski Triangle}.
Consider $(\mathbb{R}^2, d_{e})$ a metric space, $X=[0,1]^2$ and the IFS $\phi_1, \phi_2: X \to X$ where
\[\mS:
\left\{
  \begin{array}{ll}
      \phi_1(x,y)   & =(0.5 x, 0.5 y )\\
      \phi_2(x,y)   & =(0.5 x +0.5, 0.5 y )\\
      \phi_3(x,y)   & =(0.5 x +0.25, 0.5 y + 0.5)\\
  \end{array}
\right.
\]
  \begin{table}[ht]
   \centering
   \begin{tabular}{c|c|c|c|c|c|c}
     \hline
     $\alpha$ & $[a,b]\times[c,d]$ & $D$ & $n$ & $N$ & $\ve$ & $\delta$ \\
     \hline
      $0.5$ & $[0,1]\times[0,1]$ & $\sqrt{2}$ & $700$ & $18$ & $2.5\times 10^{-4}$ & $0.002505394797$ \\
     \hline
   \end{tabular}
   \caption{Resolution data for the uniform picture in \ref{Sierpinski Triangle}.}\label{table:Sierpinski Triangle}
 \end{table}
 \begin{figure}[!ht]
  \centering
  \includegraphics[width=3.5cm,frame]{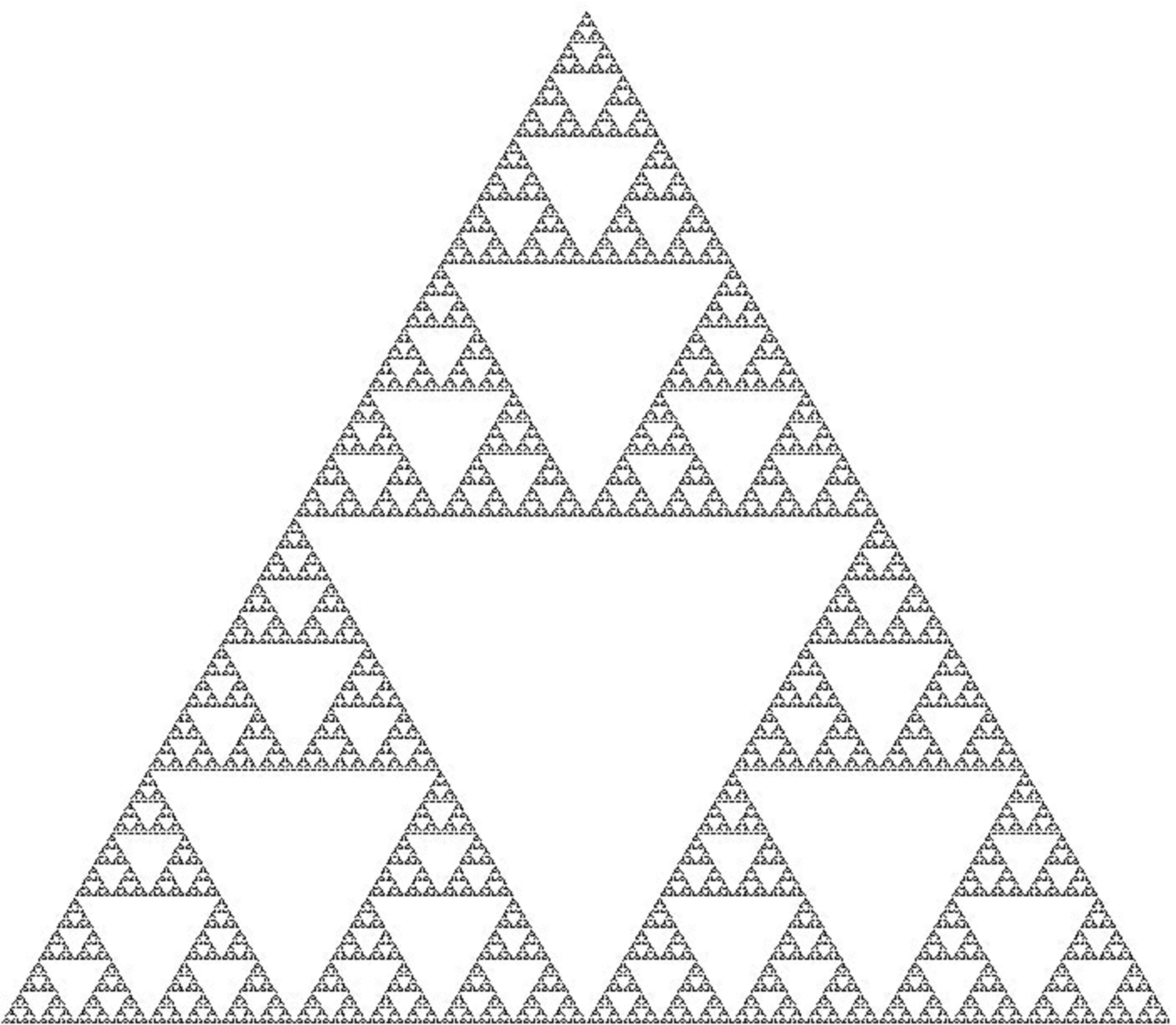}
  \includegraphics[width=3.5cm,frame]{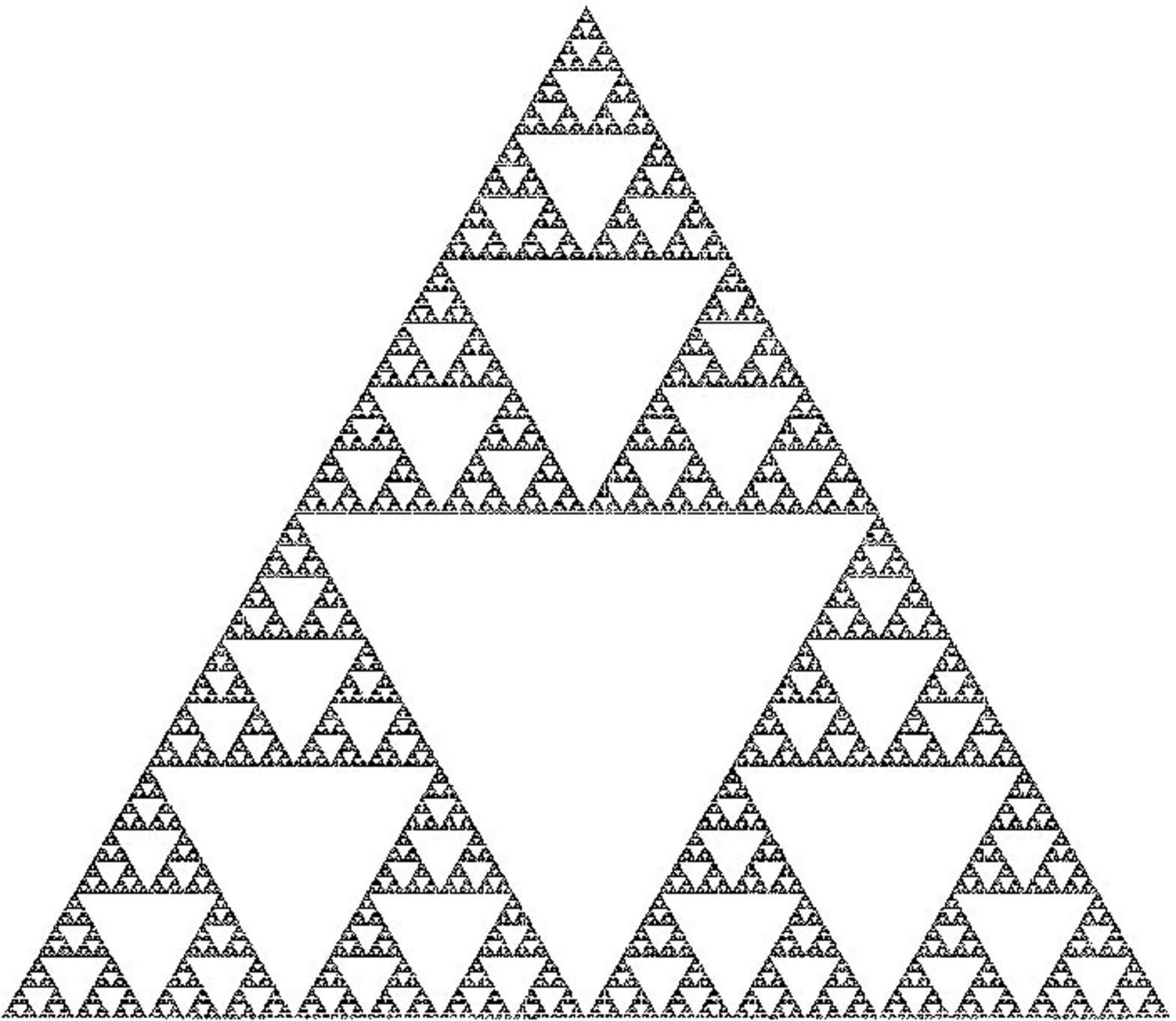}
  \includegraphics[width=3.5cm,frame]{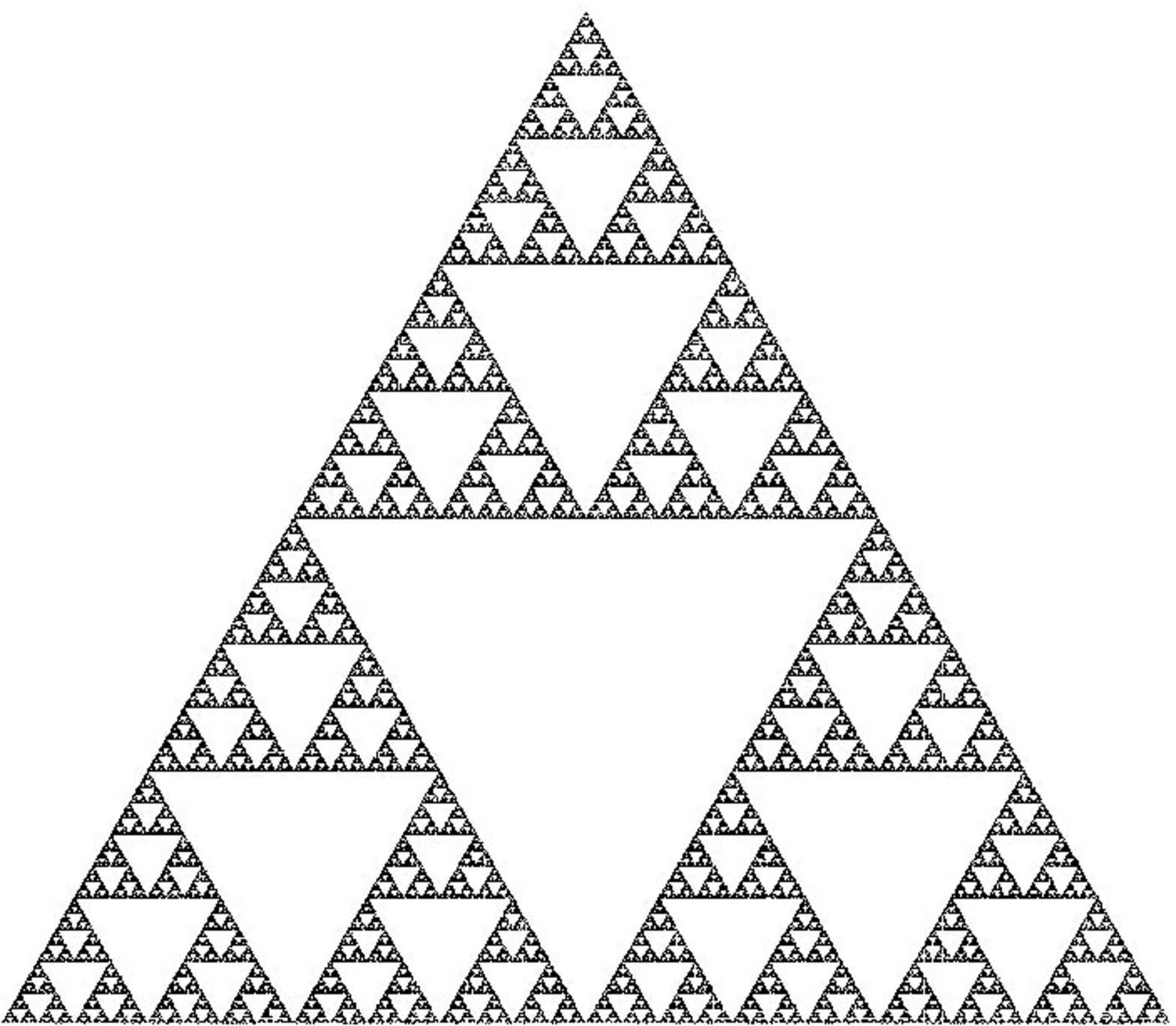}
  \includegraphics[width=3.5cm,frame]{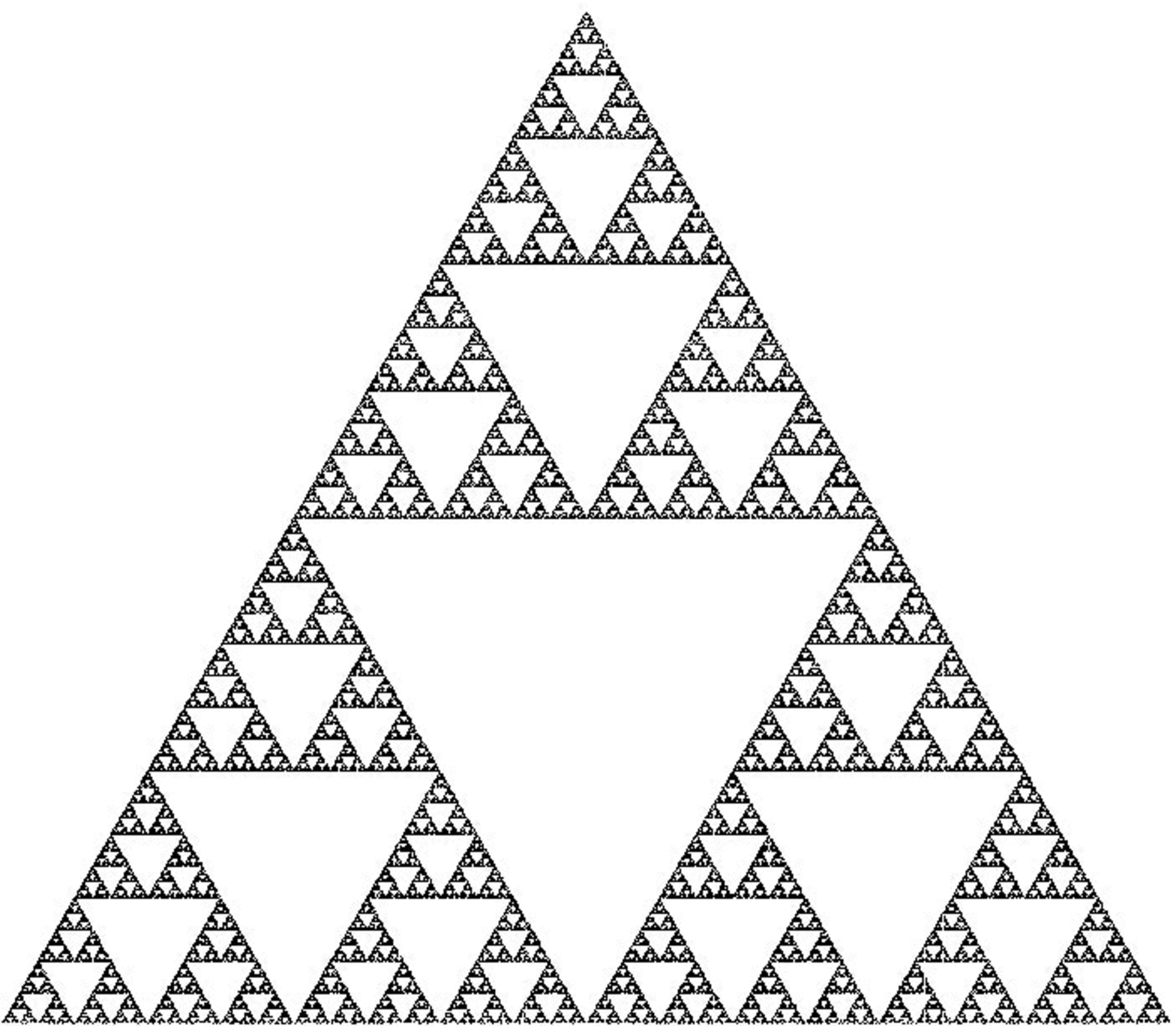}
  \caption{From the left to the right the output of the algorithm \texttt{IFSDraw($\mS$)} after 10, 12, 14 and 18 iterations. }\label{Sierpinski Triangle}
  \end{figure}

\end{example}


\begin{example}\label{discreteIFSex3} This last example is a classic fractal, the Maple Leaf, see \cite{peruggia1993discrete}. The approximation by the algorithm \texttt{IFSDraw($\mS$)} is presented in the Figure~\ref{Maple Leaf}. Consider $(\mathbb{R}^2, d_{e})$ a metric space, $X=[0,1]^2$ and the IFS $\phi_1,..., \phi_4: X \to X$ where
\begin{table}[!ht]
   \centering
   \begin{tabular}{c|c|c|c|c|c|c}
     \hline
     $\alpha$ & $[a,b]\times[c,d]$ & $D$ & $n$ & $N$ & $\ve$ & $\delta$ \\
     \hline
      $0.8$ & $[0,1]\times[0,1]$ & $\sqrt{2}$ & $700$ & $18$ & $2.5\times 10^{-4}$ & $0.03172620668$ \\
     \hline
   \end{tabular}
   \caption{Resolution data for the uniform picture in \ref{Maple Leaf}.}\label{table:Maple Leaf}
 \end{table}
\[\mS:
\left\{
  \begin{array}{ll}
      \phi_1(x,y)   & = (0.8 x + 0.1, 0.8 y + 0.04)  \\
      \phi_2(x,y)   & = (0.5 x + 0.25, 0.5 y + 0.4)\\
      \phi_3(x,y)   & = ( 0.355 x - 0.355 y +0.266,  0.355 x + 0.355 y + 0.078)  \\
      \phi_4(x,y)   & = ( 0.355 x + 0.355 y +0.378,  -0.355 x + 0.355 y + 0.434)  \\
  \end{array}
\right.
\]
 \begin{figure}[!ht]
  \centering
  \includegraphics[width=3.5cm,frame]{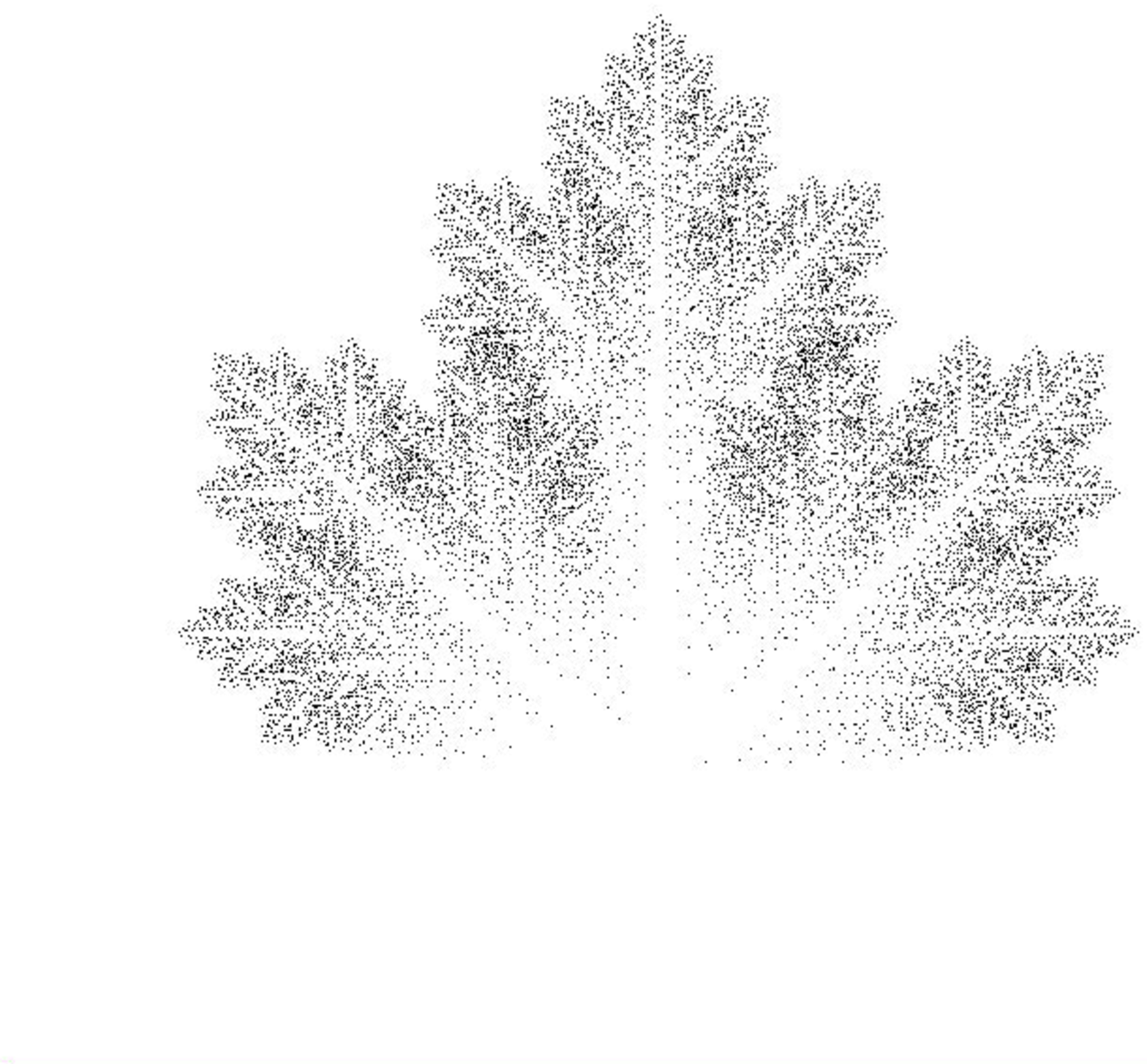}
  \includegraphics[width=3.5cm,frame]{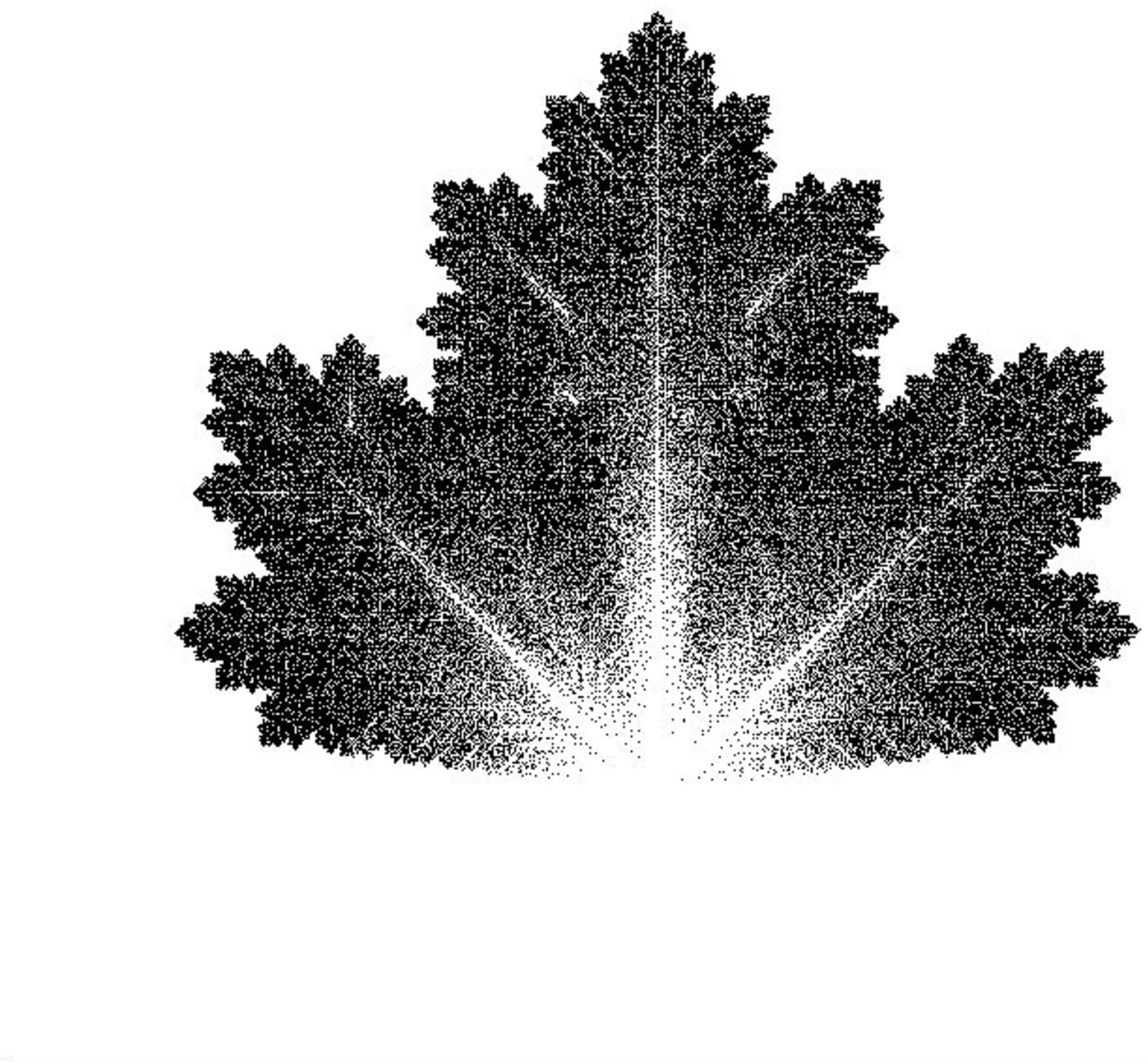}
  \includegraphics[width=3.5cm,frame]{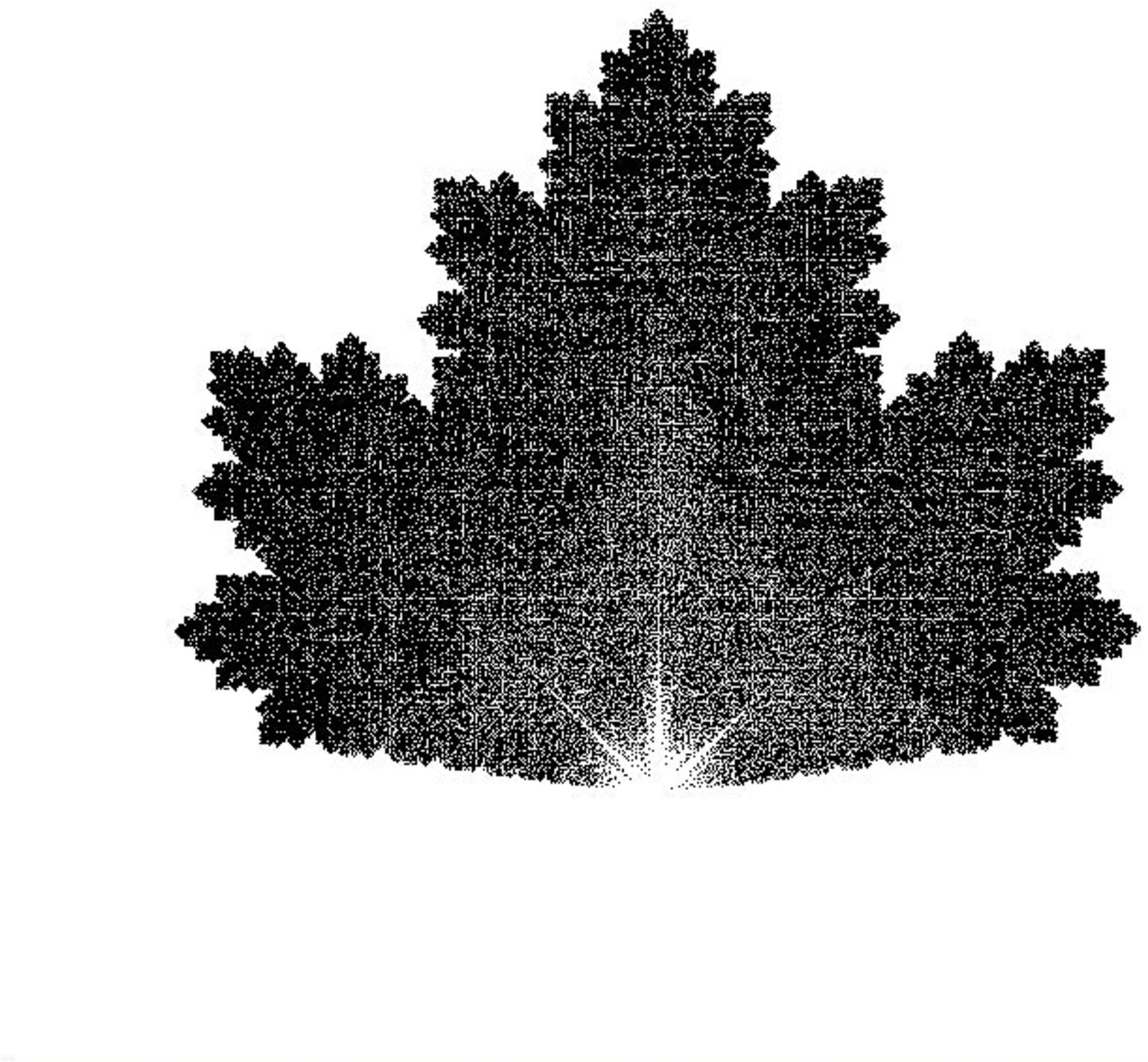}
  \includegraphics[width=3.5cm,frame]{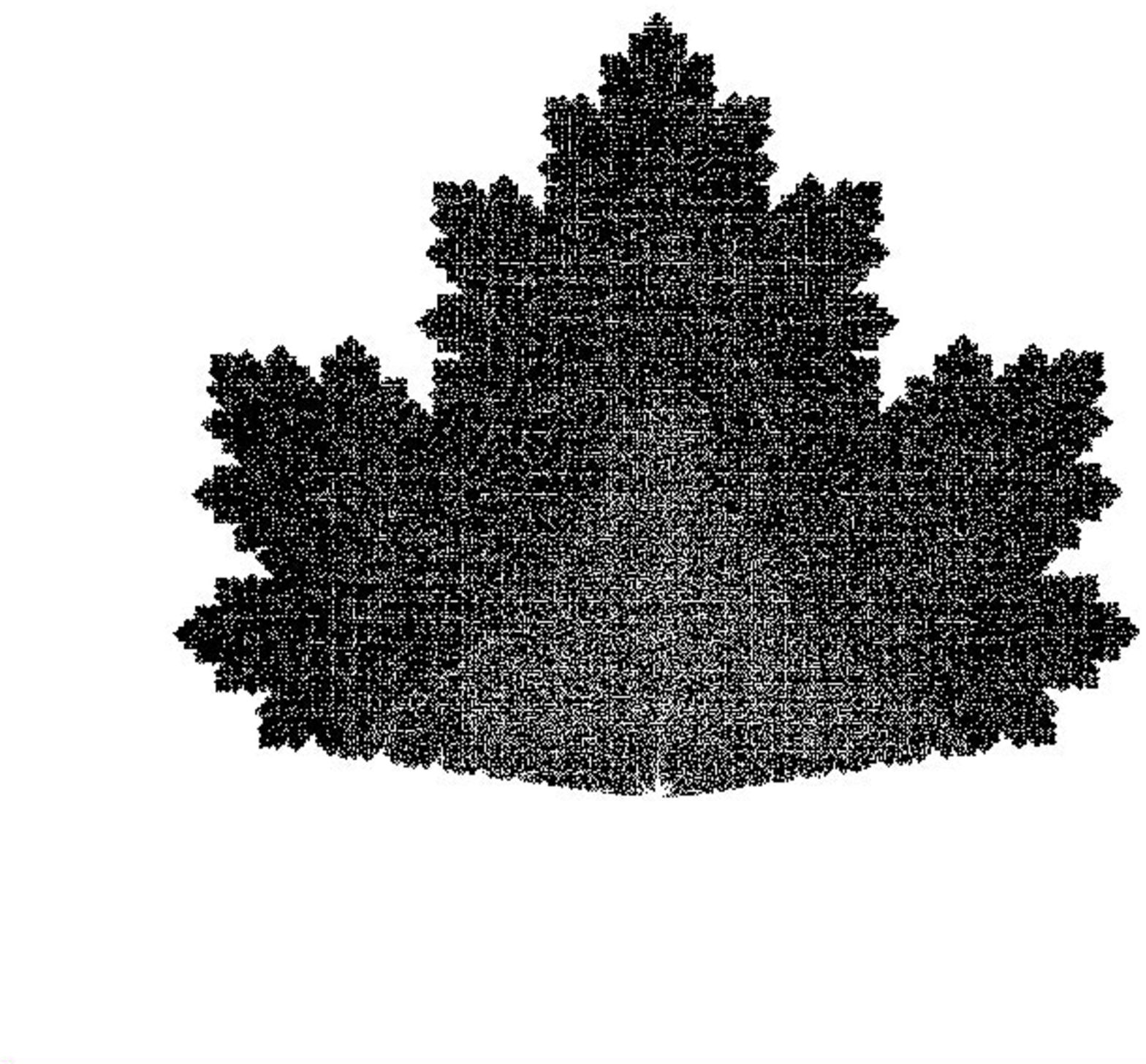}
  \caption{From the left to the right the output of the algorithm \texttt{IFSDraw($\mS$)} after   8, 11, 14 and 18 iterations. }\label{Maple Leaf}
\end{figure}
\end{example}

\section{Discrete version of the deterministic algorithm for fuzzy IFS}\label{sec:FuzzyIFSDraw}

\subsection{Discrete deterministic algorithm for fuzzy IFS} Let $\mathcal{S}=(X,(\phi_j)_{j=1}^L,(\rho_j)_{j=1}^L)$ be a fuzzy IFS satisfying the conditions of Theorem~\ref{tt3}. Then we can obtain a fractal with resolution $\delta$, approximating $u_{\mS}$, as follows:

{\tt
\begin{tabbing}\label{Determ Discrete fuzzy IFS algorithm}
aaa\=aaa\=aaa\=aaa\=aaa\=aaa\=aaa\=aaa\= \kill
     \> \texttt{FuzzyIFSDraw($\mS$)}\\
     \> {\bf input}: \\
     \> \>  \> $\delta$, the resolution.\\
     \> \>  \> $u \in{\F}^*_{\hat{X}}$ a discrete fuzzy set. \\
     \> \>  \> The diameter $D$ of a ball in $({\F}^*_{X},d_{\infty})$ containing $\on{supp}(u_{\mS})\mbox{ and }\on{supp}(u)$.\\
     \> {\bf output}: A grey scale  image representing a fractal $w$
     with resolution  at most $\delta$.\\
     \> {\bf Compute} $\displaystyle \alpha_\mS:={\rm Lip}(\mS)$\\
     \> {\bf Compute} $\varepsilon >0$ and $N \in \mathbb{N}$ such that $\frac{5\ve}{1-\alpha_\mS}+\alpha_\mS^N \, D < \delta$\\
     \>w:=u \\
     \> {\bf for n from 1 to N do}\\
     \> \> w:=$\mZ_{\hat{\mS}}$(w) \\
     \> {\bf end do}\\
     \> {\bf return: Print}(w).
\end{tabbing}}

Following the same reasoning as in Remark~\ref{resolution choices} we can obtain, from Theorem~\ref{tt3}, the appropriate estimates in order to have $\frac{5\ve}{1-\alpha_\mS}+\alpha_\mS^N \, D < \delta$. However the actual computation of $\mZ_{\hat{\mS}}(u)$ requires some additional technical details.

\subsection{Generating the inverse image set}
In the algorithm, \texttt{FuzzyIFSDraw($\mS$)} we have a computational challenge that is to perform the extension principle in a proper $\varepsilon$-net $\hat{X}$, that is, for a given $u \in{\F}^*_{\hat{X}}$ to compute the value
$\hat{\phi}(u)(x)=\max_{y \in \hat{X},\; \hat{\phi}(y)=x}u(y)$
for any $x\in \hat{X}$.  To operationalize this task in the actual algorithm we make an initial process to produce the set
$\Phi^{-1}:=\{(j, x, y) \; | \; \hat{\phi}_{j}(y)=x, \; 1\leq j \leq L\}\subset \{1,...,L\}\times\hat{X}^{2},$
which essentially contains all the possible pre images of the discrete IFS in each point. Later, when we are iterating the main loop on the algorithms \texttt{FuzzyIFSDraw($\mS$)} we just access this arrangement to obtain $\mZ_{\hat{\mS}}(u)$.

\begin{figure}[!ht]
 \centering
 \includegraphics[width=4cm,frame]{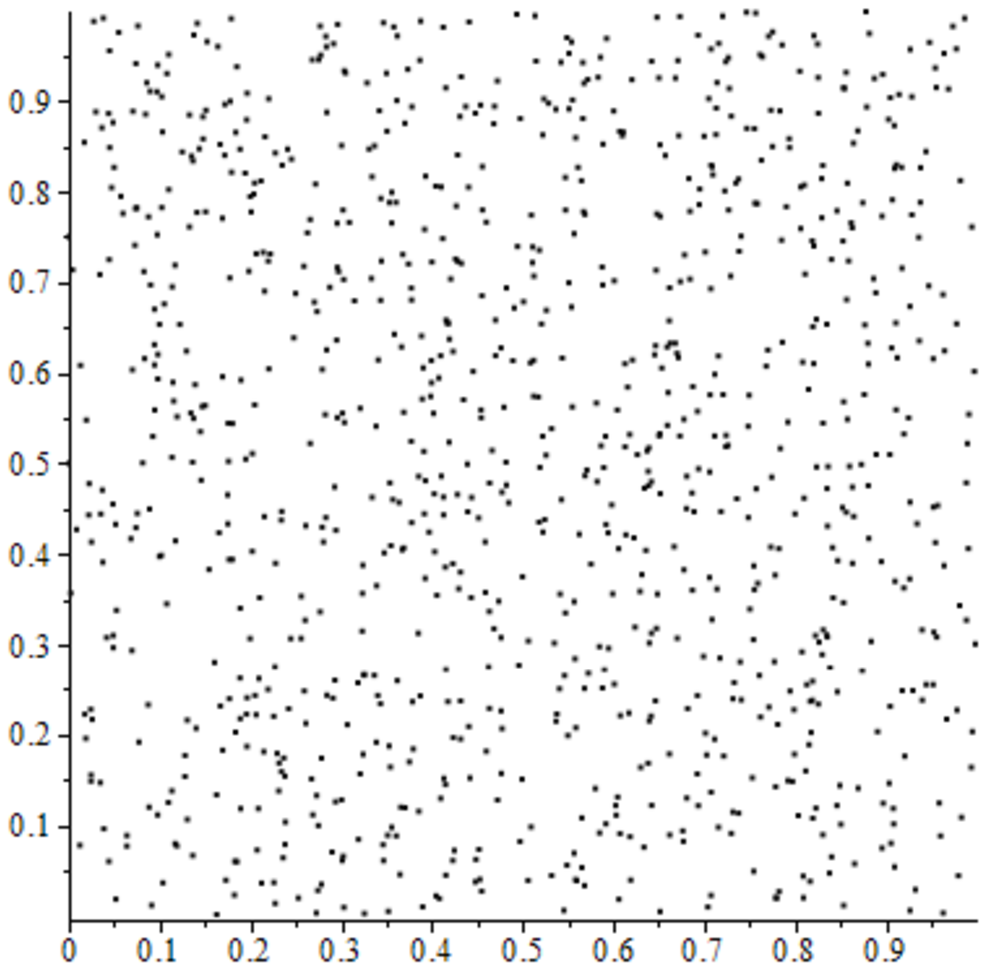}\\
 \caption{$\varepsilon$-net $\hat{X}$  given by an aleatory sequence for $X=[0,1]^2$, with $na=1000$}\label{aleatory1000}
\end{figure}

At this stage we also introduce another way of defining nets on $X$. The idea is to use appropriate chaos game algorithm. Assume $X=[0,1]^d$ (or $X$ is any appropriately regular closed and bounded subset of $\R^d$), $n\in\N$ and $\hat{X}$ be as earlier. For any $t=(t_1,...,t_d)\in \hat{X}$, define $g_t(z):=t$ for each $z\in X$. Then $({X},(g_t)_{t\in\hat{X}})$ is the IFS consisting of Banach contractions on $X$ and $\hat{X}$ is its attractor. Now we let a net $\zeta$ to be an aleatory sequence generated by a Chaos Game procedure:
 $$\zeta:=\{  \zeta_{i}  \in X\; |\; \zeta_{i+1} =g_{j_{i}} (\zeta_{i}), i\geq 0 \; \text{ and } \zeta_0 \text{ is given point}\}$$
 where each $j_i$ is chosen randomly. It is well known (see \cite{BOOK:9592}) that with a probability one, the sequence $\zeta$ is dense in the attractor of the underlying IFS, which in our case means that with a probability one, $\zeta=\hat{X}$. In reality, we stop after $an$ steps, obtaining a subset $\zeta$ of $\hat{X}$ consisting of $an$ points which, with a relatively high probability, is spread uniformly on $\hat{X}$. Then we can define the projection $r:X\to \zeta$ so that
$$
||r(z)-z||=\min\{||x-z||:x\in\zeta\}
$$
Clearly, there is no warranty that $\zeta$ is $\ve$-net for appropriately small $\ve$. However, we will see that such an approach gives satisfactory results, in some cases better than the uniform one.

As we remarked the fundamental operation to compute  $\hat{\phi}_{j}$ is the generation of the set $\Phi^{-1}$. We are going to use two approaches, the \textbf{uniform} one where we run on each point of a uniform $\varepsilon$-net and the \textbf{aleatory} one where we introduce an auxiliary chaos game sequence which also produces an $\varepsilon$-net.

Consider $\mathcal{S}=(X, (\phi_j)_{j=1}^{L}, (\rho_{j})_{j=1}^{L})$ an IFZS and the following algorithm to generate $\Phi^{-1}$:
{\tt
\begin{tabbing}
aaa\=aaa\=aaa\=aaa\=aaa\=aaa\=aaa\=aaa\= \kill
     \> \texttt{Generate($\Phi^{-1}$)}\\
     \> {\bf input}: \\
     \> \>   {An} empty list $\Phi^{-1}$.\\
     \> \>   {A} net  $Y$.\\
     \> \>   {A} projection $r$.\\
     \> {\bf output}: $\Phi^{-1}$\\
     \> {\bf for j from 1 to L do}\\
     \> \> {\bf for $x\in Y$ do}\\
     \> \> \> {\bf $z:=r(\phi_{j}(x))$.}\\
     \> \> \> {\bf $\Phi^{-1}:=\Phi^{-1} \cup (j,x,z)$.}\\
     \> \> {\bf end do}\\
     \> {\bf end do}\\
     \>{\bf return:} $\Phi^{-1}$.
\end{tabbing}}

\subsection{Fuzzy IFS examples}\label{Fuzzy IFS examples}

\subsubsection{Examples with dimension one}

\begin{example}\label{example 1_1IFZS}
  This first example is an exact reproduction of an IFZS from \cite{Cabrelli-1992}, Example 1.  In the following we show the picture drawn by \cite{Cabrelli-1992} and the figure obtained by our algorithm with $n=200$ and $na=1000000$. The resemblance is remarkable.
 Consider $(\mathbb{R}, d_{e})$ a metric space, $X=[0,1]$ and a fuzzy IFS $\mS : \phi_1,..., \phi_4: X \to X$ with grey level functions $\rho_1,..., \rho_4: [0,1] \to [0,1]$ where
$\phi_i(x)=\frac{x}{4}+\frac{i-1}{4}, \; i=1,..., 4$
and
\[\rho_1(t):=
\left\{
  \begin{array}{ll}
      0.25\,t   & 0\leq t < 0.25  \\
      t-0.18 & 0.25 \leq  t \leq 1
  \end{array}
\right.,
\] $\rho_2(t):=t$, $\rho_3(t):=0.33\, t$ and $\rho_4(t):=\sin(t)$.
\begin{table}[ht]
   \centering
   \begin{tabular}{c|c|c|c|c|c|c}
     \hline
     $\alpha$ & $[a,b]\times[c,d]$ & $D$ & $n$ & $N$ & $\ve$ & $\delta$ \\
     \hline
      $0.25$ & $[0,1]\times[0,1]$ & $1$ & $200$ & $9$ & $10^{-6}$ & $0.00001048136393$ \\
     \hline
   \end{tabular}
   \caption{Resolution data for the uniform picture in \ref{example1cabrelliIFZS}.}\label{table:example1cabrelliIFZS}
 \end{table}
 \begin{figure}[!ht]
  \centering
  \includegraphics[width=3cm]{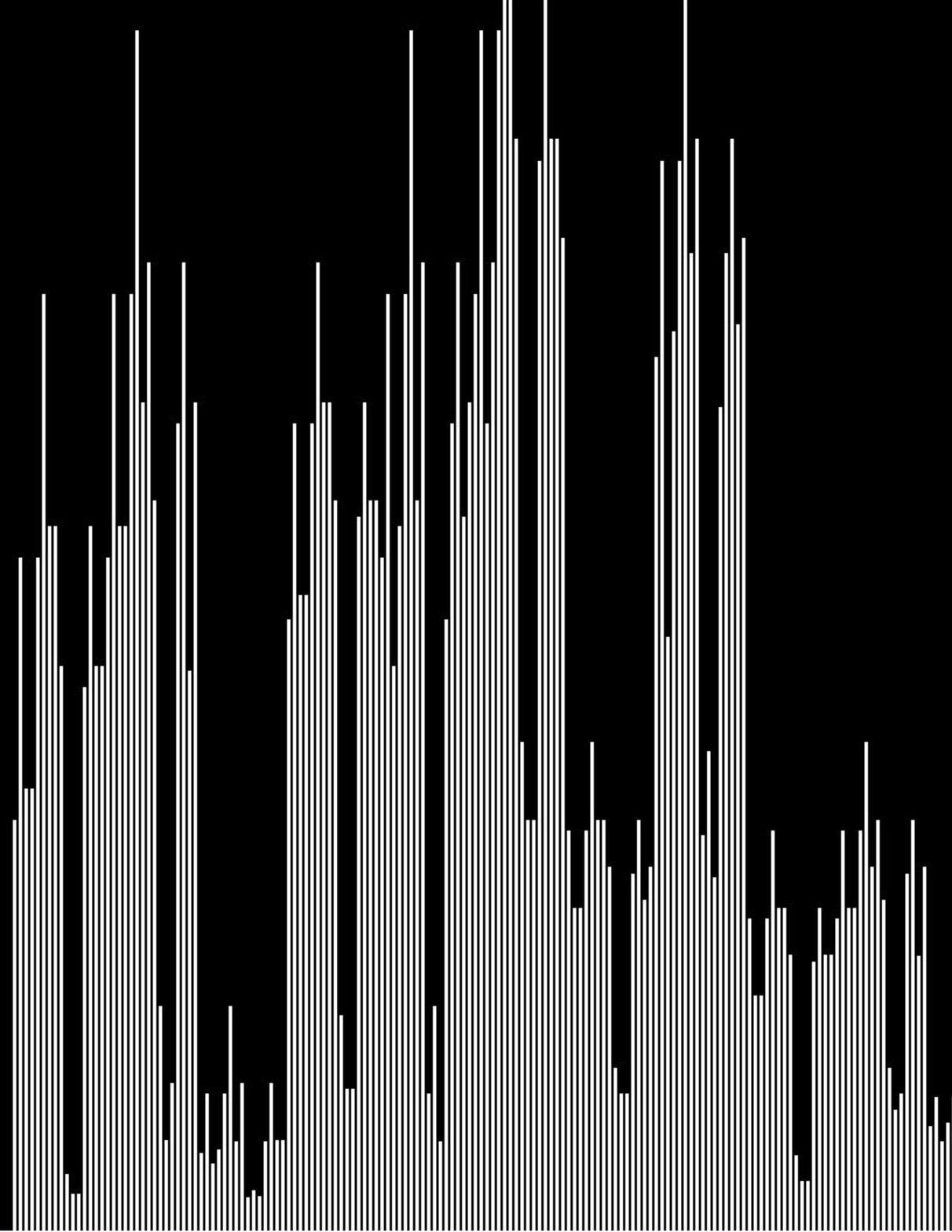}
  \includegraphics[width=3cm]{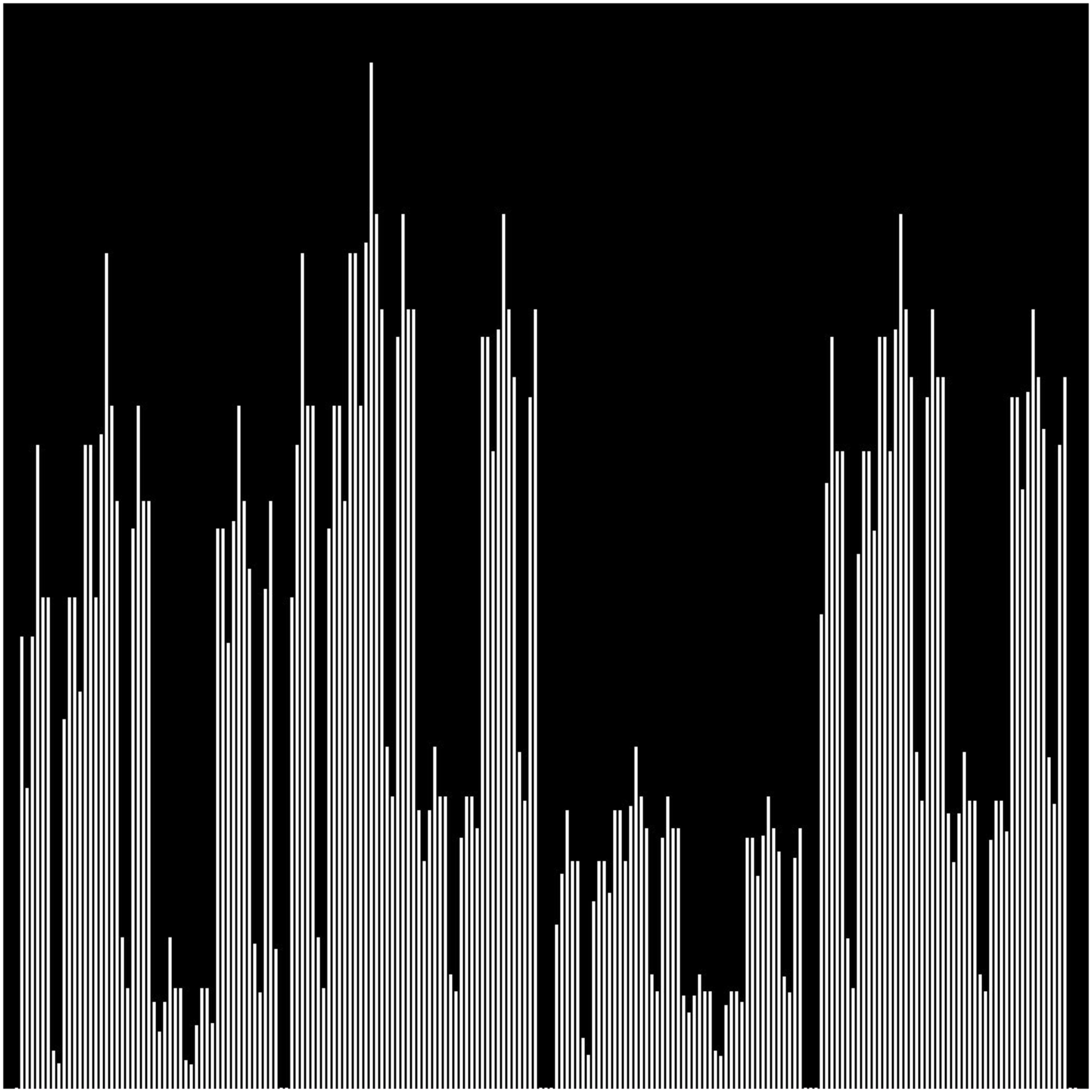}
  \includegraphics[width=3cm]{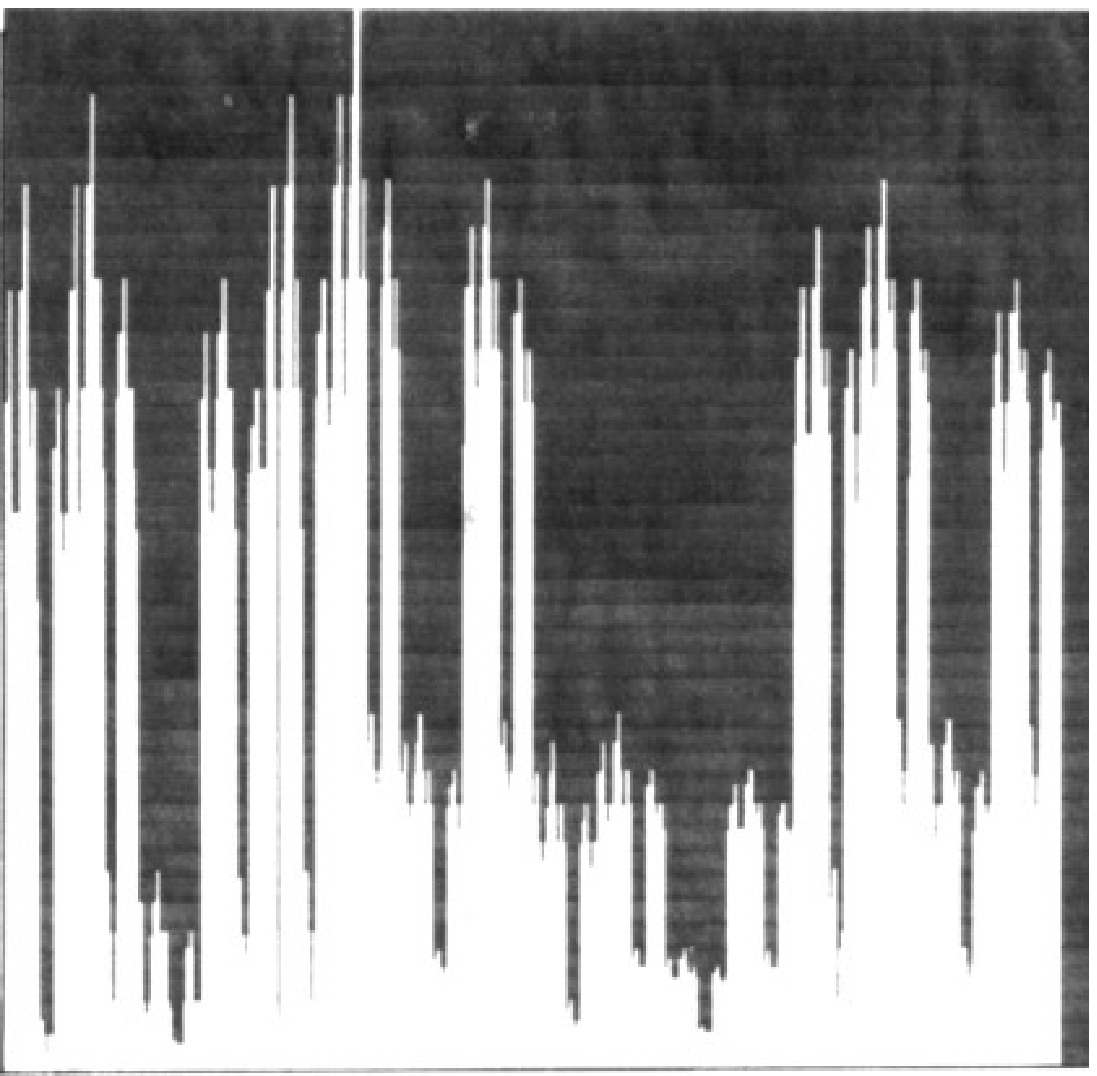}\\
  \caption{From the left to the right, the picture  obtained by the algorithm \texttt{FuzzyIFSDraw($\mS$)} with uniform generation $n=250$, with aleatory generation $n=200$ and $na=1000000$ and the figure drawn in \cite{Cabrelli-1992}.}\label{example1cabrelliIFZS}
\end{figure}

\end{example}


\subsubsection{Examples with dimension 2}

\begin{example}\label{example 1 IFZS}
  Consider $(\mathbb{R}^2, d_{e})$ a metric space, $X=[0,1]^2$ and  a fuzzy IFZS $\phi_1,..., \phi_4: X \to X$ with grey level functions $\rho_1,..., \rho_4: [0,1] \to [0,1]$, where
\[\mS: \left\{
  \begin{array}{ll}
    \phi_1(x_1,y_1)&=(0.5\,x_1, 0.5\,y_1) \\
    \phi_2(x_1,y_1)&=(0.5\,x_1+0.5, 0.5\,y_1) \\
    \phi_3(x_1,y_1)&=(0.5\,x_1, 0.5\,y_1+0.5) \\
    \phi_4(x_1,y_1)&=(0.5\,x_1+0.5, 0.5\,y_1+0.5)
  \end{array}
\right.
\]
and
\[\rho_1(t):=
\left\{
  \begin{array}{ll}
      0    & 0\leq t < 0.2505  \\
      0.25  & 0.2505 \leq t < 0.505  \\
      0.5   & 0.505\leq t < 0.7505  \\
      0.75 & 0.7505 \leq  t \leq 1
  \end{array}
\right.,
\] $\rho_2(t):=t$, $\rho_3(t):= t$ and $\rho_4(t):=t$.\\
\begin{table}[ht]
   \centering
   \begin{tabular}{c|c|c|c|c|c|c}
     \hline
     $\alpha$ & $[a,b]\times[c,d]$ & $D$ & $n$ & $N$ & $\ve$ & $\delta$ \\
     \hline
      $0.5$ & $[0,1]\times[0,1]$ & $\sqrt{2}$ & $200$ & $11$ & $10^{-6}$ & $0.0007005339658$ \\
     \hline
   \end{tabular}
   \caption{Resolution data for the uniform picture in \ref{example2cabrelli_uniform_aleatoryIFZS}.}\label{table:example2cabrelli_uniform_aleatoryIFZS}
 \end{table}

The approximation of the fuzzy attractor, via uniform and aleatory algorithm, is given in the picture Figure~\ref{example2cabrelli_uniform_aleatoryIFZS}.
  \begin{figure}[!ht]
  \centering
  \includegraphics[width=3cm]{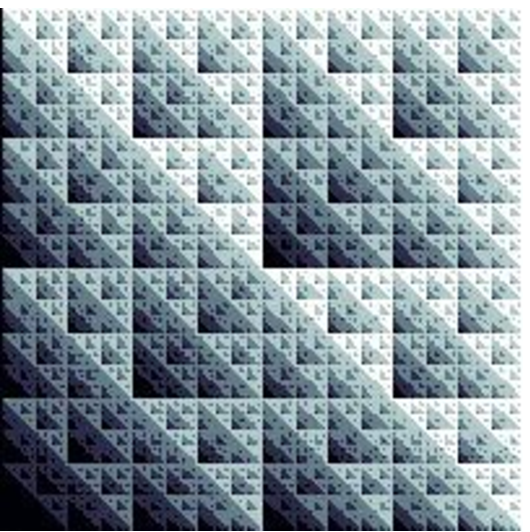}
  \includegraphics[width=3cm]{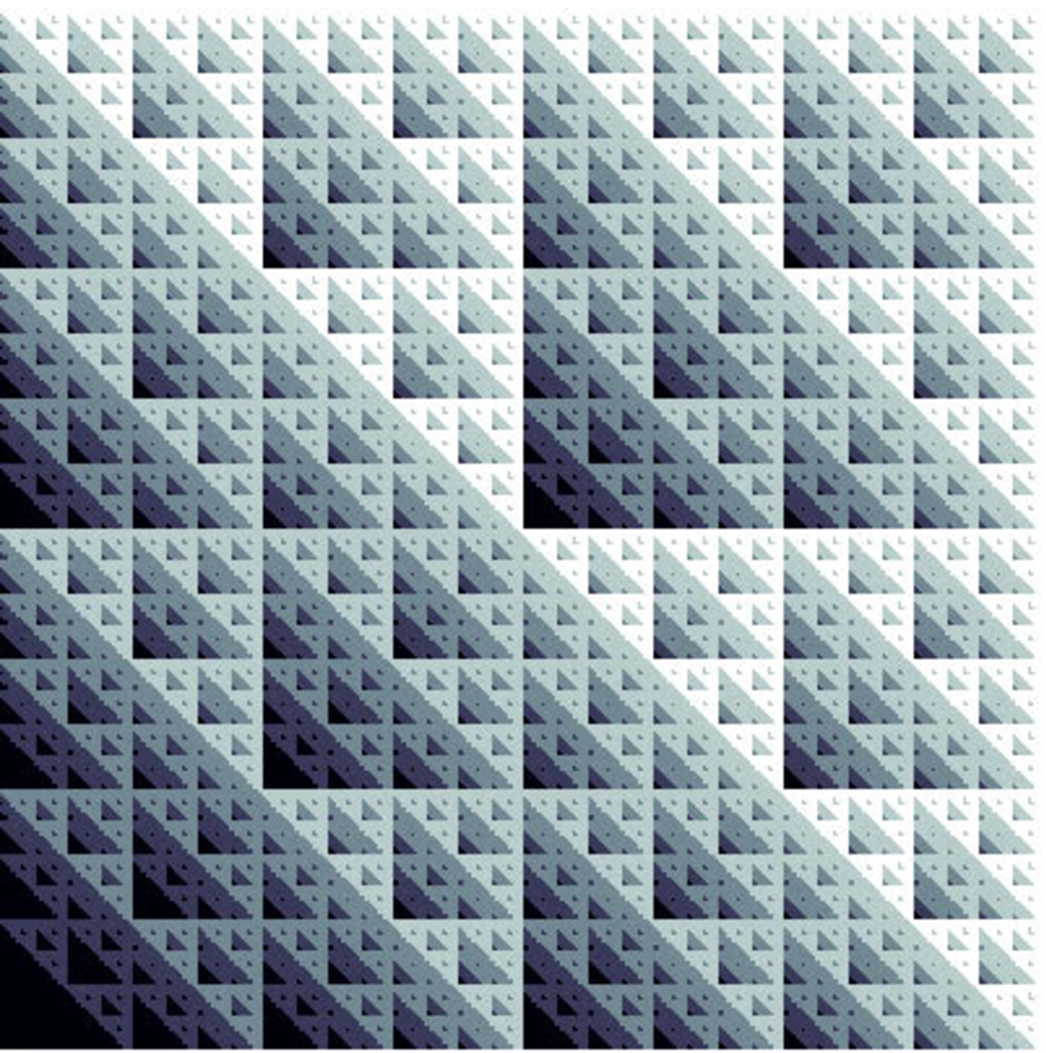}
 \includegraphics[width=3.1cm]{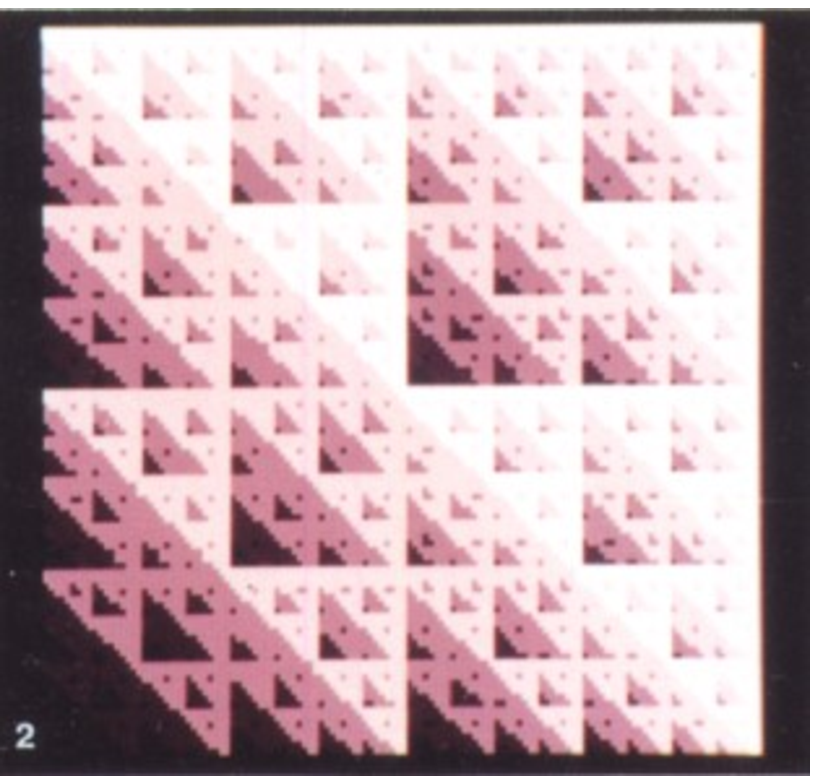}\\
  \caption{From the left to the right the picture obtained by the algorithm \texttt{FuzzyIFSDraw($\mS$)}, with aleatory generation $n=200$ and $11$ iterations, with aleatory generation, with $n=400$, $na=40000000$ and the picture drawn in \cite{Cabrelli-1992}.}\label{example2cabrelli_uniform_aleatoryIFZS}
\end{figure}

The images shown in Figure~\ref{example2cabrelli_uniform_aleatoryIFZS}, left, were produced using the program that implements the uniform algorithm, using the external file version and the specific data regarding its performance are given in Table~\ref{tab:example_1_runs}, where we show the number of iterations, $k$; the time taken to produce the $\Phi^{-1}$ sets, $t_S$; the time taken for the iterations, $t_i$; and the overall execution time, $t$. All times reported are wall-clock times and are given in seconds.
\begin{table}[ht]
\centering
\begin{tabular}{ccccc}
\hline
$n$&$k$&$t_{\Phi^{-1}}$&$t_i$&$t$\\
\hline
$50$&$8$&$57.09$&$777.92$&$835.01$\\
$100$&$10$&$850.33$&$5231.4$&$6081.73$\\
$200$&$10$&$14221.15$&$83506.20$&$97727.35$\\
\hline\\
\end{tabular}
\caption{Number of iterations and execution times $t_S$, $t_i$ and $t$ using the program that implements the deterministic algorithm with the external file version.}\label{tab:example_1_runs}
\end{table}

\end{example}



\begin{example}\label{example 8}
  Consider $(\mathbb{R}^2, d_{e})$ a metric space, $X=[0,1]^2$ and a fuzzy IFS $\mS: \phi_1,..., \phi_4: X \to X$, where $X=[0,1]$ and
\[\mS : \left\{
  \begin{array}{ll}
    \phi_1(x_1,y_1)&=(0.5\,x_1, 0.5\,y_1) \\
    \phi_2(x_1,y_1)&=(0.5\,y_1+0.5, -0.5\,x_1+0.5) \\
    \phi_3(x_1,y_1)&=(-0.5\,y_1+0.5, 0.5\,x_1+0.5) \\
    \phi_4(x_1,y_1)&=(0.5\,x_1+0.5, 0.5\,y_1+0.5)
  \end{array}
\right.
\]
and $\{\rho_1,...,\rho_4\}$ are the same grey level functions used in the  Example~\ref{example 1 IFZS}.
\begin{table}[ht]
   \centering
   \begin{tabular}{c|c|c|c|c|c|c}
     \hline
     $\alpha$ & $[a,b]\times[c,d]$ & $D$ & $n$ & $N$ & $\ve$ & $\delta$ \\
     \hline
      $0.5$ & $[0,1]\times[0,1]$ & $\sqrt{2}$ & $100$ & $11$ & $10^{-6}$ & $0.0007005339658$ \\
     \hline
   \end{tabular}
   \caption{Resolution data for the uniform picture in \ref{example8_uniform and chaos}.}\label{table:example8_uniform and chaos}
 \end{table}

The approximations of the fuzzy attractor, via uniform and aleatory algorithm, are given in Figure~\ref{example8_uniform and chaos}.

\begin{figure}[!ht]
 \centering
  \includegraphics[width=3cm]{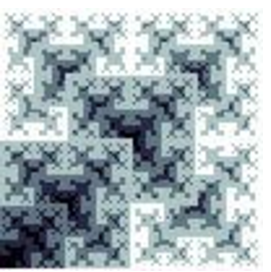} \includegraphics[width=3cm]{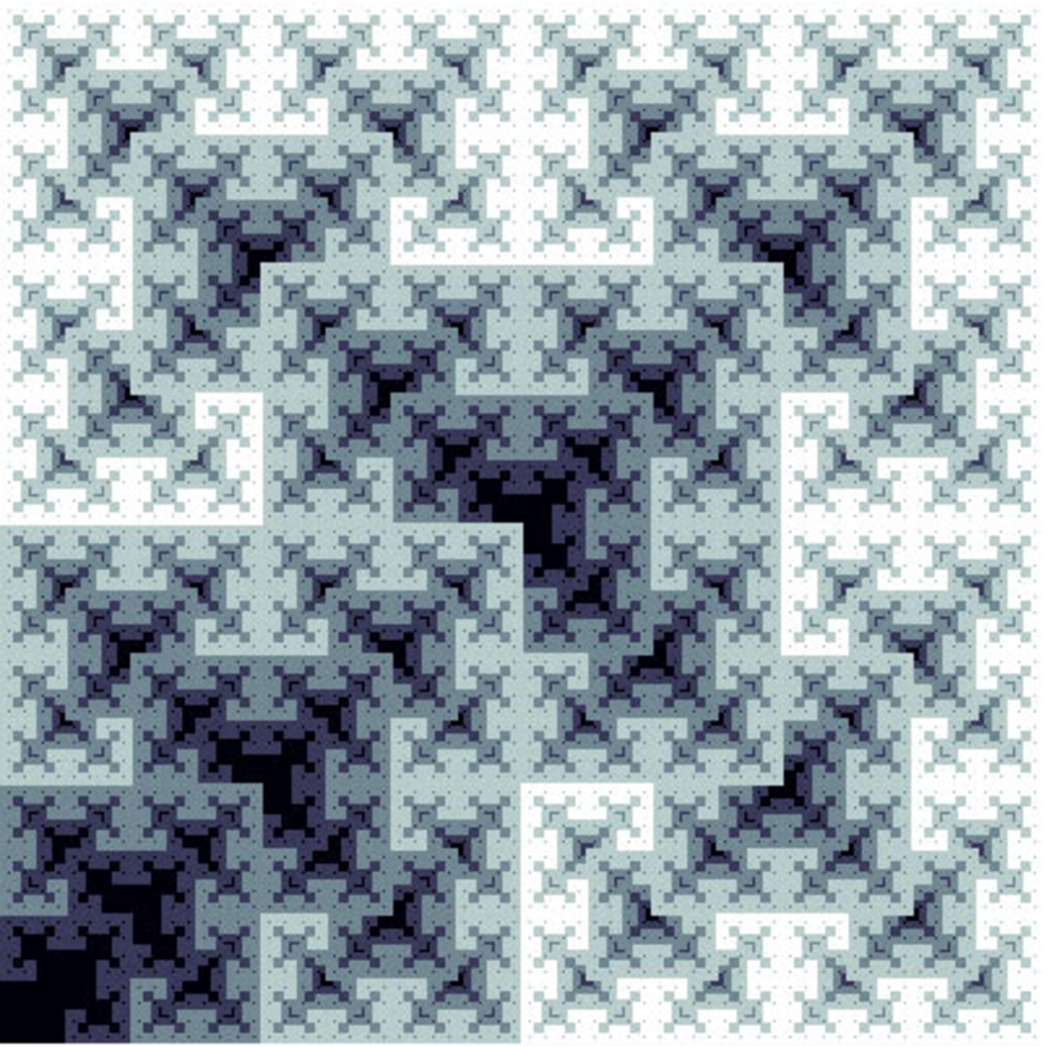}\\
 \caption{On the left, he figure obtained by the uniform algorithm with $n=100$. Performance data: number of iterations, $10$; time for $\Phi^{-1}$ set construction, $851.71$ seconds; time for iterations, $536.88$ seconds; overall running time, $1388.59$ seconds. On the right, he figure obtained by the aleatory algorithm with $n=400$ and $na=160000000$. Performance data: number of iterations, $10$; time for $\Phi^{-1}$ set construction, $1545.66$ seconds; time for iterations,  $19186.91$ seconds; overall running time, $20732.59$ seconds.}\label{example8_uniform and chaos}
  \end{figure}

\end{example}

\begin{example}\label{Fern_IFZS_Counterpart} The next example is the IFZS version of the classic fractal, the Barnsley Fern. The approximation by the algorithm \texttt{FuzzyIFSDraw($\mS$)} is presented in the Figure~\ref{Fern_Fuzzy}. Consider $(\mathbb{R}^2, d_{e})$ a metric space, $X=[0,1]^2$ and the fuzzy IFS $\phi_1,..., \phi_4: X \to X$ given by Example~\ref{discreteIFSex1},
with the grey scale functions
$\rho_1(t):=t$
\[\rho_2(t):=
\left\{
  \begin{array}{ll}
      0    & 0\leq t < 0.2505  \\
      0.25  & 0.2505 \leq t < 0.505  \\
      0.5   & 0.505\leq t < 0.7505  \\
      0.75 & 0.7505 \leq  t \leq 1
  \end{array}
\right.,
\] $\rho_3(t):= t$ and $\rho_4(t):=t$.\\
 \begin{table}[ht]
   \centering
   \begin{tabular}{c|c|c|c|c|c|c}
     \hline
     $\alpha$ & $[a,b]\times[c,d]$ & $D$ & $n$ & $N$ & $\ve$ & $\delta$ \\
     \hline
      $0.8680563766$ & $[0,1]\times[0,1]$ & $\sqrt{2}$ & $100$ & $11$ & $2.5\times 10^{-4}$ & $0.3076975065$ \\
     \hline
   \end{tabular}
   \caption{Resolution data for the uniform picture in \ref{Fern_Fuzzy}.}\label{table:fern_fuzzy_ifs}
 \end{table}

 \begin{figure}[!ht]
  \centering
  \includegraphics[width=3cm,frame]{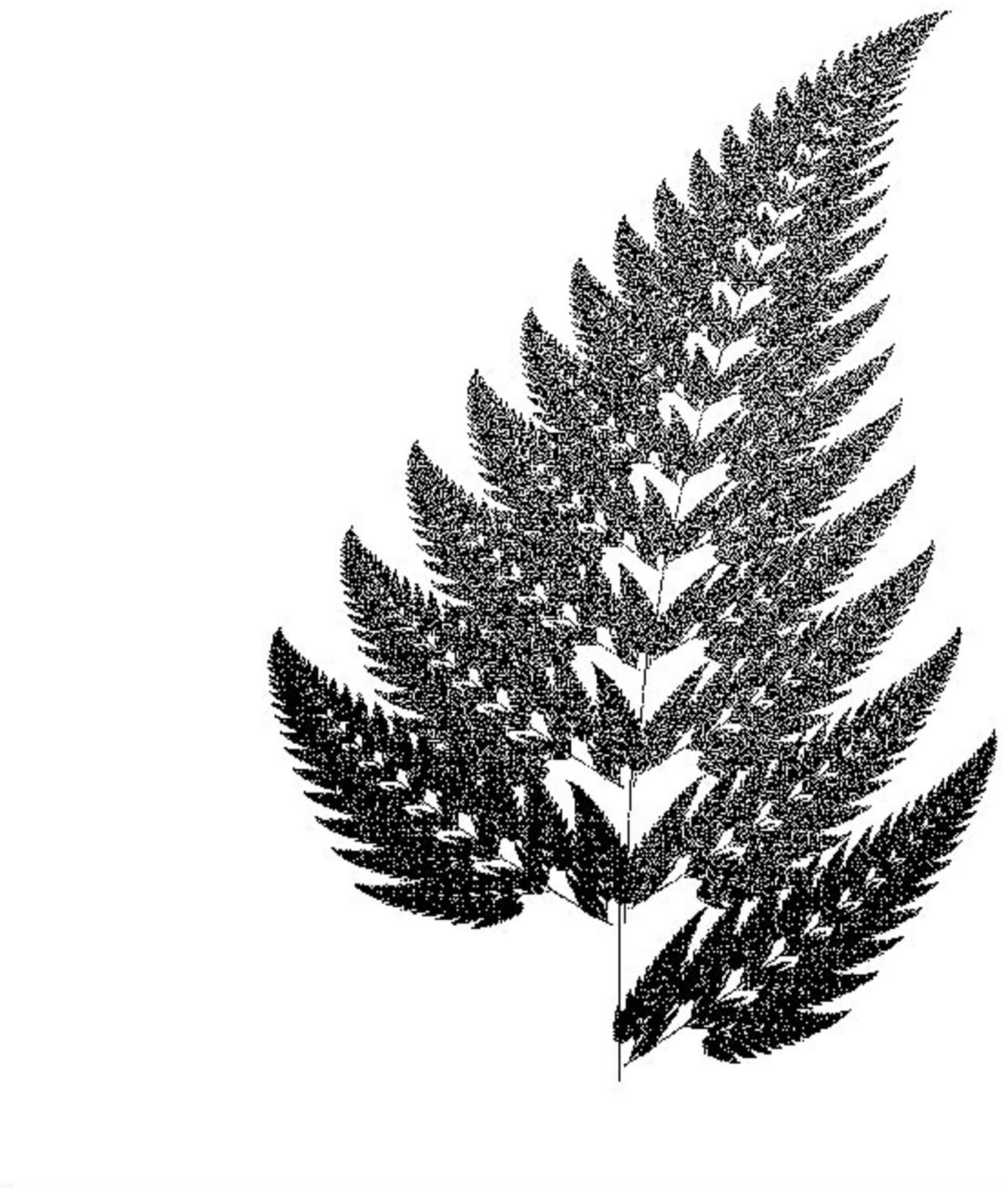}\;
  \includegraphics[width=3cm]{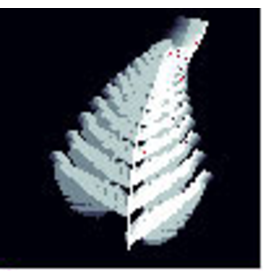}\;
  \includegraphics[width=3cm]{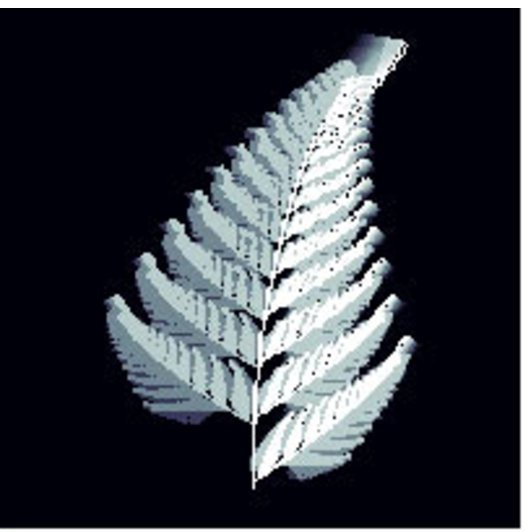}
  \caption{From the left to the right the output of the algorithm \texttt{IFSDraw($\mS$)} and the output of the algorithm \texttt{FuzzyIFSDraw($\mS$)} after 11 iterations on the uniform version and 3 iterations on the aleatory version.}\label{Fern_Fuzzy}
 \end{figure}
\end{example}


\begin{example}\label{discreteIFSex3_IFZS_Counterpart} This last example the IFZS version of the classic fractal, the Maple Leaf, from Example~\ref{discreteIFSex3}. The approximation by the algorithm \texttt{FuzzyIFSDraw($\mS$)} is presented in the Figure~\ref{Maple LeafFuzzy}.  Consider $(\mathbb{R}^2, d_{e})$ a metric space, $X=[0,1]^2$ and the fuzzy IFS $\phi_1,..., \phi_4: X \to X$ where
$\mS$ is the IFS from Example~\ref{discreteIFSex3}, with the grey scale functions
$\rho_1(t):=t$
\[\rho_2(t):=
\left\{
  \begin{array}{ll}
      0    & 0\leq t < 0.2505  \\
      0.25  & 0.2505 \leq t < 0.505  \\
      0.5   & 0.505\leq t < 0.7505  \\
      0.75 & 0.7505 \leq  t \leq 1
  \end{array}
\right.,
\] and $\rho_3(t):=t$.\\
  \begin{table}[ht]
   \centering
   \begin{tabular}{c|c|c|c|c|c|c}
     \hline
     $\alpha$ & $[a,b]\times[c,d]$ & $D$ & $n$ & $N$ & $\ve$ & $\delta$ \\
     \hline
      $0.8$ & $[0,1]\times[0,1]$ & $\sqrt{2}$ & $100$ & $8$ & $2.5\times 10^{-4}$ & $0.2435156640$ \\
     \hline
   \end{tabular}
   \caption{Resolution data for the uniform picture in \ref{Maple LeafFuzzy}.}\label{table:maple_fuzzy_ifs}
 \end{table}

 \begin{figure}[!ht]
  \centering
  \includegraphics[width=3cm]{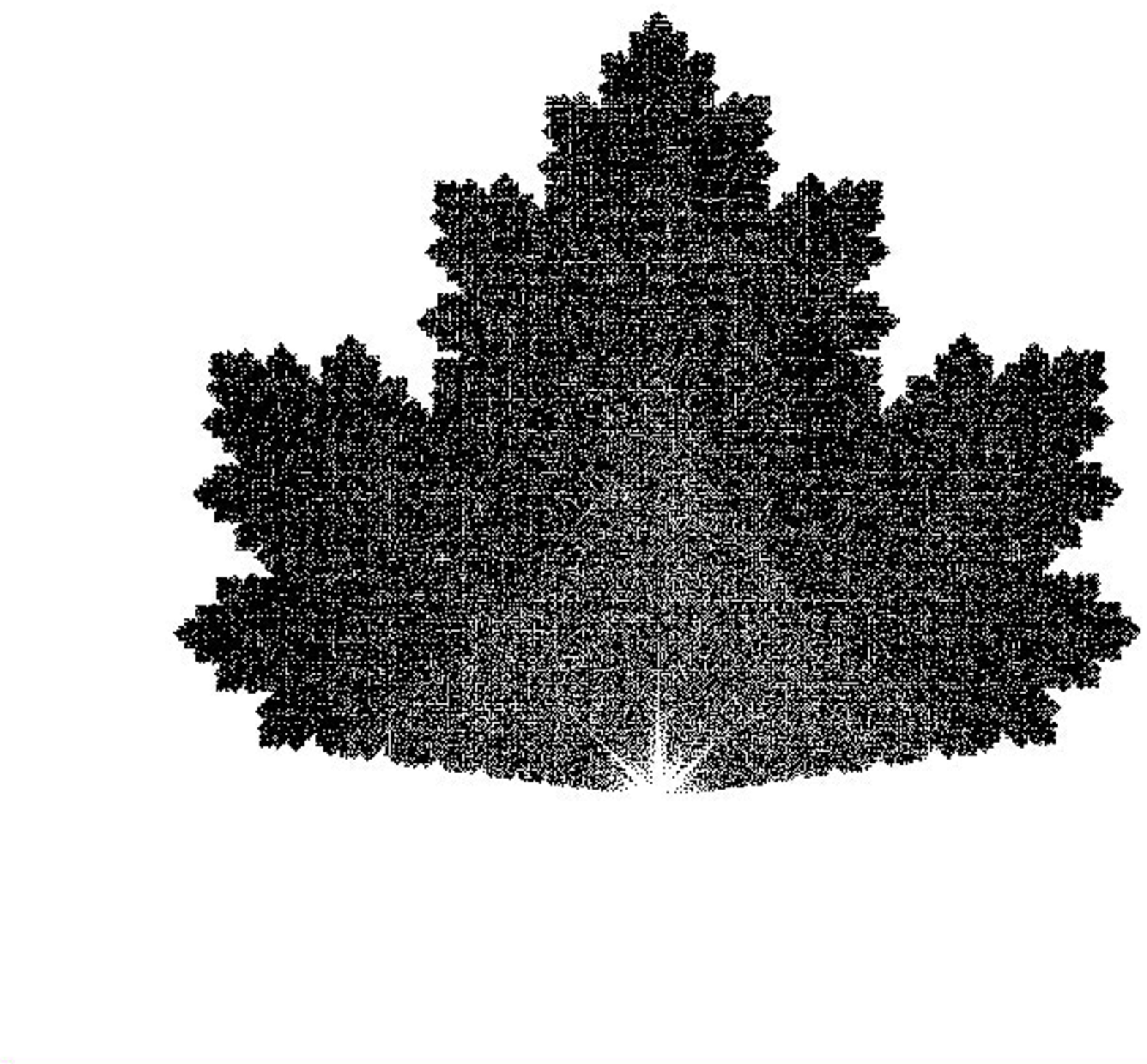}\;
  \includegraphics[width=3cm]{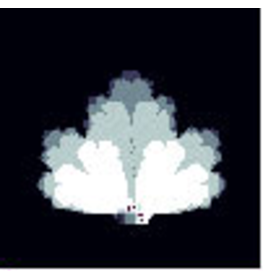}
  \includegraphics[width=3cm]{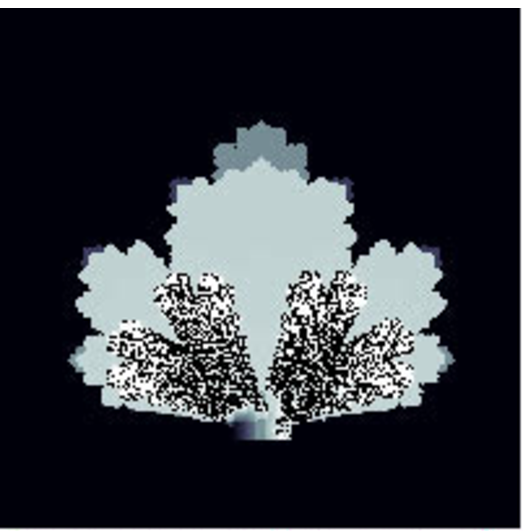}
  \caption{From the left to the right the output of the algorithm \texttt{IFSDraw($\mS$)} after 8 iterations and the output of the algorithm \texttt{FuzzyIFSDraw($\mS$)} uniform after 10 iterations and the aleatory one with $na=800000$ after $11$ iterations.}\label{Maple LeafFuzzy}
 \end{figure}
\end{example}


\section{Generalized IFSs and their fuzzyfication}\label{sec:Generalized IFSs and their fuzzyfication}
We first recall some basics of a generalization of the classical IFS theory introduced by R. Miculescu and A. Mihail in 2008. For references, see \cite{mihail2008recurrent}, \cite{mihail2010generalized}, \cite{strobin_swaczyna_2013} and references therein.\\
If $(X,d)$ is a metric space and $m\in\N$, then by $X^m$ we denote the Cartesian product of $m$ copies of $X$. We consider it as a metric space with the maximum metric
$$
d_m((x_0,...,x_{m-1}),(y_0,...,y_{m-1})):=\max\{d(x_0,y_0),...,d(x_{m-1},y_{m-1})\}.
$$
A map $f:X^m\to X$ is called a \emph{generalized Banach contraction}, if $\on{Lip}(f)<1$.

It turns out that a counterpart of the Banach fixed point theorem holds. Namely,
if $f:X^m\to X$ is a generalized Banach contraction, then there is a unique point $x_*\in X$ (called \emph{a generalized fixed point} of $f$), such that $f(x_*,...,x_*)=x_*$. Moreover, for every $x_0,...,x_{m-1}\in X$, the sequence $(x_k)$ defined by $$x_{k+m}=f(x_k,...,x_{k+m-1}),\;\;k\geq 0,$$
converges to $x_*$.\\
This result can be used to prove a counterpart of the Hutchinson--Barnsley theorem.
\begin{definition}\emph{
A }generalized iterated function system of order $m$\emph{ (GIFS in short) $\mS=(X,(\phi_j)_{j=1}^{L})$ consists of a finite family  $\phi_1,...,\phi_L$ of continuous maps from $X^m$ to $X$. Each GIFS $\mS$ generates the map $F_\mS:\K^*(X)^m\to\K^*(X)$, called }the generalized Hutchinson operator  (GH) \emph{, defined by
$$
\forall_{K_0,...,K_{m-1}\in \K^*(X)}\;F_\mS(K_0,...,K_{m-1}):=
\bigcup_{j=1}^L\phi_j(K_0\times...\times K_{m-1}).
$$
By the \emph{attractor} of a GIFS $\mS$ we mean the unique set $A_\mS\in\K^*(X)$ which satisfies
$$
A_\mS=F_\mS(A_\mS,...,A_\mS)=\bigcup_{j=1}^L\phi_j(A_\mS\times...\times A_\mS)
$$
and such that for every $K_0,...,K_m\in\K^*(X)$, the sequence $(K_k)$ defined by
$$
K_{k+m}:=F_\mS(K_k,...,K_{k+m-1}),\;\;k\geq 0,
$$
converges to $A_\mS$.
}\end{definition}

The following lemma is known:
\begin{lemma}\label{lem3g}
Let $(X,d)$ be a metric space and $\mS=(X,(\phi_j)_{j=1}^{L})$ be a GIFS consisting of generalized Banach contractions. Then $F_\mS$ is a generalized Banach contraction with $\on{Lip}(F_\mS)\leq\max\{\on{Lip}(\phi_j):j=1,...,L\}$.
\end{lemma}

 From now on, given a GIFS $\mS$ consisting of generalized Banach contractions, we denote $\alpha_\mS:=\max\{\on{Lip}(\phi_j):j=1,...,L\}.$\\
From the perspective of the algorithms presented later, it is worth to consider also a bit different approach. Namely, given a GIFS $\mS=(X,(\phi_j)_{j=1}^L)$, define the map $\overline{F}_\mS:\K^*(X)\to\K^*(X)$ by
\begin{equation*}
\forall_{K\in\K^*(X)}\;\overline{F}_\mS(K):=F_\mS(K,...,K)=\bigcup_{j=1}^L\phi_j(K\times...\times K).
\end{equation*}
Lemma \ref{lem3g} implies the following:
\begin{lemma}\label{lem3gg}
Let $(X,d)$ be a metric space and  $\mS=(X,(\phi_j)_{j=1}^{L})$ be a GIFS consisting of generalized Banach contractions. Then $\overline{F}_\mS$ is a Banach contraction with $\on{Lip}(\overline{F}_\mS)\leq \alpha_\mS$.
\end{lemma}
Clearly, in the above frame, if additionally $X$ is complete, then the fixed point of $\overline{F}_\mS$ is the attractor of $\mS$.

Now we recall the fuzzy version of the above setting, which was introduced in \cite{Oliveira-2017}.

Given $m \geq 2$ and $u_{0}, ..., u_{m-1} \in \mathcal{F}_{X}$ we define the Cartesian product  ${\btimes_iu_i}=u_{0}\times ...\times u_{m-1} \in \mathcal{F}_{X^m}$ by
$$\left(\btimes_iu_i\right)(x_0, ..., x_{m-1}){:= \bigwedge_{i=0}^{m-1 }u_{i}(x_{i}):=\min\{u_i(x_i):i=0,...,m-1\}}.$$
As we proved in \cite[Lem. 3.9]{Oliveira-2017},
{if $u_0,...,u_{m-1}\in\F^*_X$, then $\btimes_iu_i\in\F^*_{X^m}$, and, moreover
} if {$u_0,...,u_{m-1} \in\F^*_X$ and  $v_0,...,v_{m-1}\in\F^*_X$}, then {(we denote by the same symbol $d_\infty$ the metric on $\F^*_{X^m}$)}
$$
d_\infty(\btimes_iu_i,\btimes_iv_i)=\max\{d_\infty(u_j,v_j):j=0,...,m-1\}.
$$
{In particular, the map $(\F^*_X)^m\ni(u_0,...,u_{m-1})\to \btimes_iu_i\in\F^*_{X^m}$ is an isometric embedding.}

\begin{definition} \label{GIFZS definition}
A generalized iterated fuzzy function system of degree $m$ \emph{ (GIFZS in short) \\ $\mS=(X,(\phi_j)_{j=1}^{L},(\rho_j)_{j=1}^{L})$ consists of a GIFS $\mS=(X,(\phi_j)_{j=1}^{L})$ with a set of admissible grey level maps {$(\rho_j)_{j \in \{1,...,L\}}$}.\\
The operator $Z_\mathcal{S}:  \left(\mathcal{F}_{X}^{*}\right)^m\to \mathcal{F}_{X}^{*}$ defined by
$${Z_\mathcal{S}(u_0,...,u_{m-1})}:= \bigvee_{j \in \{1,...,L\}} \rho_{j}(\phi_{j}(\btimes_{i}  u_{i})){:=\max\{ \rho_{j}(\phi_{j}(\btimes_{i}  u_{i})):j=1,...,L\}}$$
is called }the generalized fuzzy {Hutchinson operator (GFH)} associated to $\mathcal{Z_S}$.\\
A fuzzy set $u_\mS\in \mathcal{F}_{X}^{*}$  is called a generalized fuzzy fractal of a GIFZS $\mathcal{S}=(X, (\phi_j), (\rho_{j}))_{j \in \{1,...,L\}}$ if $Z_\mathcal{S}(u_\mS,...,u_\mS)=u_\mS$, that is
$$u_\mS= \bigvee_{j \in \{1,...,L\}} \rho_{j}(\phi_{j}(\btimes_{i=0}^{m-1}u_\mS)).$$
\end{definition}

The following result is a consequence of \cite[Thm. 3.14]{Oliveira-2017}.

\begin{lemma}\label{lem3f}
Let $(X,d)$ be a metric space and  $\mS=(X,(\phi_j)_{j=1}^{L},(\rho_j)_{j=1}^{L})$ be a GIFZS consisting of generalized Banach contractions. Then $Z_\mathcal{S}$ is a generalized Banach contraction with $\on{Lip}(Z_\mS)\leq\max\{\on{Lip}(\phi_j):j=1,...,L\}$.
\end{lemma}
 From now on, given a GIFZS $\mS$ consisting of generalized Banach contractions, we denote $\alpha_\mS:=\max\{\on{Lip}(\phi_j):j=1,...,L\}.$

From the perspective of the algorithms presented later, it is worth to consider also a bit different approach. Namely, given a GIFZS $\mS=(X,(\phi_j)_{j=1}^{L},(\rho_j)_{j=1}^{L})$, define the map $\overline{Z}_\mS:\F^*_X\to\F^*_X$ by
\begin{equation*}
\forall_{u\in\F^*_X}\;\overline{Z}_\mS(u):=Z_\mS(u,...,u)=
\bigcup_{j=1}^L\rho_{j}(\phi_j(u\times...\times u)).
\end{equation*}
Lemma \ref{lem3f} implies the following:
\begin{lemma}\label{lem3fff}
Let $(X,d)$ be a metric space and $\mS=(X,(\phi_j)_{j=1}^{L},(\rho_j)_{j=1}^{L})$ be a GIFZS consisting of generalized Banach contractions. Then $\overline{Z}_\mS$ is a Banach contraction and $\on{Lip}(\overline{Z}_\mS)\leq\alpha_\mS$.
\end{lemma}
Clearly, in the above frame, if additionally $X$ is complete, then the fixed point of $\overline{Z}_\mS$ is the attractor of $\mS$.

\section{GIFS and GIFZS discretization}\label{sec:GIFS and GIFZS discretization}

\begin{definition}\emph{
Given a GIFS $\mS$ on a metric space $X$ with the attractor $A_\mS$ and $\delta>0$, a set $A_\delta\in\K^*(X)$ will be called an attractor of $\mS$ }with resolution $\delta$\emph{, if  $h(A_\delta,A_\mS)\leq\delta$.}
\end{definition}
It can be easily proved (similarly as Lemma \ref{lemm2}(c)) that for a GIFS $\mS=(X,(\phi_j)_{j=1}^L,(\rho_j)_{j=1}^L)$, an $\ve$-net $\hat{X}$ and an $\ve$-projection $r:X\to\hat{X}$, the discretization $(r\circ \overline{F}_\mS)_{\vert \K^*(\hat{X})}$ equals the operator $\overline{F}_{\hat{\mS}}$ adjusted to the GIFS $\hat{\mS}=(\hat{X},(\hat{\phi}_j)_{j=1}^L,(\rho_j)_{j=1}^L)$ consisting of discretizations $\hat{\phi}_j:=(r_m\circ \phi_j)_{\vert\hat{X}^m}$, where $r_m(x_1,...,x_m)=(r(x_1),...,r(x_m))$ is the natural projection of $X^m$ to $\hat{X}^m$. Hence
Lemma \ref{lem3f} and Theorem \ref{dfp}  imply the following ``discrete" version of the Hutchinson--Barnsley theorem for GIFSs.
\begin{theorem}\label{ttt3}
Let $(X,d)$ be a complete metric space and $\mathcal{S}=(X, (\phi_j)_{j=1}^L)$ be a GIFS on $X$ consisting of generalized Banach contractions. Let $\ve>0$, ${\hat{X}}$ be a proper $\ve$-net, $r:X\to {\hat{X}}$ be an $\ve$-projection on ${\hat{X}}$ and  $\mS=(X,(\phi_j)_{j=1}^{L},(\rho_j)_{j=1}^{L})$, where $\hat{\phi}_j:=(r_m\circ \phi_{j})_{\vert {\hat{X}}^m}$ for $j=1,...,L$.\\
Then for any $K\in\K^*(\hat{X})$ and $n\in\N$,
$$
d_\infty(\overline{F}_{\hat{\mS}}^n(K),A_\mS)\leq\frac{5\ve}{1-\alpha_\mS}+\alpha_\mS^n d_\infty(K,A_\mS),
$$
where $A_\mS$ is the attractor of $\mS$.\\
In particular, there is $n_0\in\N$ such that for every $n\geq n_0$, $\overline{F}_{\hat{\mS}}^n(K)$ is an attractor of $\mS$ with the resolution $\frac{6\ve}{1-\alpha_\mS}$.
\end{theorem}
Finally, we state a counterpart for the fuzzy version.

\begin{definition}\emph{
Given a GIFZS $\mS$ on a metric space $X$ with the fuzzy attractor $u_\mS$ and $\delta>0$, a fuzzy set $u_\delta\in\F^*_X$ will be called an attractor of $\mS$ }with resolution $\delta$\emph{, if  $d_\infty(u_\delta,u_\mS)\leq\delta$.}
\end{definition}
Similarly as Lemma \ref{lemma4}(e) we can prove that for a GIFZS $\mS=(X,(\phi_j)_{j=1}^L,(\rho_j)_{j=1}^L)$, an $\ve$-net $\hat{X}$ and an $\ve$-projection $r:X\to\hat{X}$, the discretization $(r\circ \overline{Z}_\mS)_{\vert \tilde{\F}^*_{\hat{X}}}$ equals the operator $e(\overline{Z}_{\hat{\mS}})$ adjusted to the GIFZS $\hat{\mS}=(\hat{X},(\hat{\phi}_j)_{j=1}^L,(\rho_j)_{j=1}^L)$. Hence
Lemma \ref{lem3ff} and Theorem \ref{dfp}  imply the following ``discrete" version of the Hutchinson--Barnsley theorem for fuzzy GIFSs.
\begin{theorem}\label{ttt4}
Let $(X,d)$ be a complete metric space and $\mS=(X,(\phi_j)_{j=1}^L,(\rho_j)_{j=1}^L)$ be a GIFS on $X$ consisting of generalized Banach contractions. Let $\ve>0$, ${\hat{X}}$ be a proper $\ve$-net, $r:X\to {\hat{X}}$ be an $\ve$-projection on ${\hat{X}}$ and  $\mS=(X,(\hat{\phi}_j)_{j=1}^L,(\rho_j)_{j=1}^L)$, where $\hat{\phi}_j:=(r_m\circ \phi_{j})_{\vert {\hat{X}}^m}$ for $j=1,...,L$.\\
Then for any $u\in\F^*_{\hat{X}}$ and $n\in\N$,
$$d_\infty(e(\overline{Z}_{\hat{\mS}}^n(u)),u_\mS)\leq
\frac{5\ve}{1-\alpha_\mS}+\alpha_\mS^nd_\infty(e(u),u_\mS),$$
where $u_\mS$ is the fuzzy attractor of $\mS$.\\
In particular, there is $n_0\in\N$ such that for every $n\geq n_0$, $e(\overline{Z}_{\hat{\mS}}^n(u))$ is an attractor of $\mS$ with the resolution $\frac{6\ve}{1-\alpha_\mS}$.
\end{theorem}

\section{Discrete version of the deterministic algorithm for GIFSs and fuzzy GIFSs}\label{sec:FuzzyIFSandGIFSDraw}

Theorems \ref{ttt3} and \ref{ttt4} can be used to get a discrete version of the deterministic algorithms. They are very simlar as IFS's ones - the only difference is that we iterate operators $\overline{F}_{\hat{\mS}}$ or $\overline{Z}_{\hat{\mS}}$ instead of $F_{\hat{\mS}}$ and $Z_{\hat{\mS}}$. Moreover, in the case of fuzzy GIFSs, the set $\Phi^{-1}$ is a subset of $\{1,...,L\}\times \hat{X}^m\times \hat{X}$.

\subsection{Discrete deterministic algorithm for GIFS} Let $\mathcal{S}=(X,(\phi_j)_{j=1}^L)$ be a GIFS satisfying the conditions of Theorem~\ref{ttt3} then we can obtain a fractal with resolution $\delta$, approximating $A_{\mS}$.  The algorithm is the very same as \texttt{IFSDraw($\mS$)}.
{\tt
\begin{tabbing}
aaa\=aaa\=aaa\=aaa\=aaa\=aaa\=aaa\=aaa\= \kill
     \> \texttt{GIFSDraw($\mS$)}\\
     \> {\bf input}: \\
     \> \>  \> $\delta>0$, the resolution.\\
     \> \>  \> $K \subseteq \hat{X}$, any finite and not empty subset. \\
     \> \>  \> The diameter $D$ of a ball in $(X,d)$ containing $A_{\mS}$ and $K$.\\
     \> {\bf output}: A bitmap  image representing a fractal with resolution  at most $\delta$.\\
     \> {\bf Compute} $\displaystyle \alpha_\mS:={\rm Lip}(\mS)$\\
     \> {\bf Compute} $\varepsilon >0$ and $N \in \mathbb{N}$ such that $\frac{5\ve}{1-\alpha_\mS}+\alpha_\mS^N \, D < \delta$\\
     \>W:=K \\
     \> {\bf for n from 1 to N do}\\
     \> \> W:=$\overline{F}_{\hat{\mS}}$(W)\\
     \> {\bf end do}\\
     \> {\bf return: Print}(W).\\
\end{tabbing}}

Following the same reasoning as in Remark~\ref{resolution choices} we can obtain, from Theorem~\ref{ttt3}, the appropriate estimates in order to have $\frac{5\ve}{1-\alpha_\mS}+\alpha_\mS^N \, D < \delta$.

\subsection{Examples of discrete GIFS fractals} \label{Examples of discrete GIFS fractals}
\begin{example}\label{discreteGIFSex1}
   This first example is the GIFS $\mathcal{F}$ appearing in Example 8 from \cite{Jaros-2016}. The approximation by the algorithm \texttt{GIFSDraw($\mS$)} is presented in the Figure~\ref{ex5.3}. Consider $(\mathbb{R}^2, d_{e})$ a metric space, $X=[0,2.1]\times[0.1,2.4]$ and the  IFS $\phi_1, \phi_2: X^2 \to X$ where
\[\mS:
\left\{
  \begin{array}{ll}
           \phi_1(x_1,y_1,x_2,y_2)&=(0.1 x_1+0.15 x_2+0.04 y_2, 0.16 y_1-0.04 x_2+0.15 y_2 +1.6)\\
           \phi_2(x_1,y_1,x_2,y_2)&= (0.1 x_1-0.15 y_1-0.1 x_2+0.15 y_2+1.6, 0.15 x_1+0.15 y_1 +0.15 x_2+0.07 ) \\
  \end{array}
\right.
\]
  \begin{table}[ht]
   \centering
   \begin{tabular}{c|c|c|c|c|c|c}
     \hline
     $\alpha$ & $[a,b]\times[c,d]$ & $D$ & $n$ & $N$ & $\ve$ & $\delta$ \\
     \hline
      $0.4209556069$ & $[0,2.1]\times[0.1,2.4]$ & $3.114482300$ & $700$ & $5$ & $2.5\times 10^{-4}$ & $0.04332744866$ \\
     \hline
   \end{tabular}
   \caption{Resolution data for the uniform picture in \ref{ex5.3}.}\label{table:ex5.3_fuzzy_gifs}
 \end{table}
 \begin{figure}[!ht]
  \centering
  \includegraphics[width=3cm,frame]{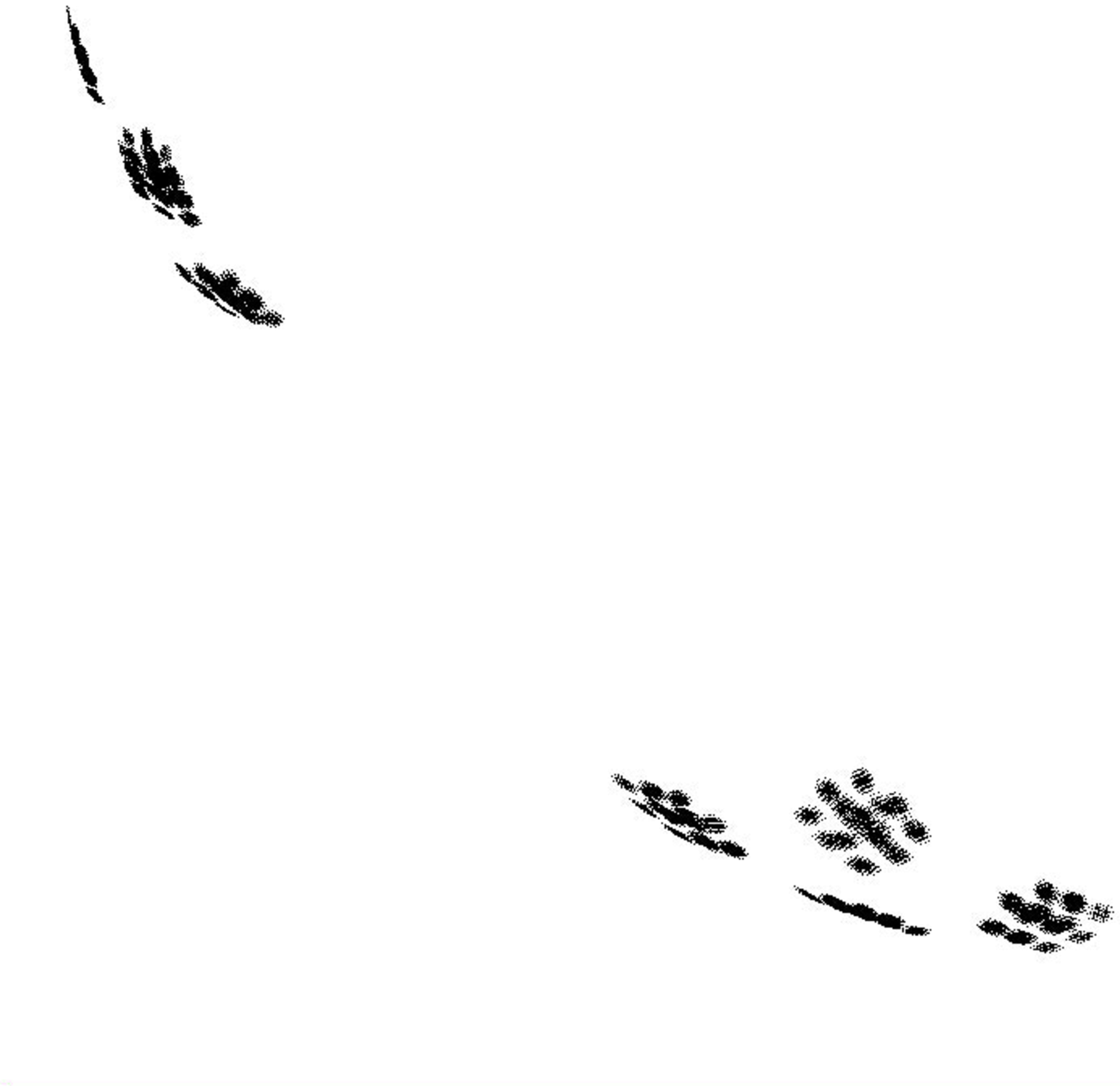}
  \includegraphics[width=3cm,frame]{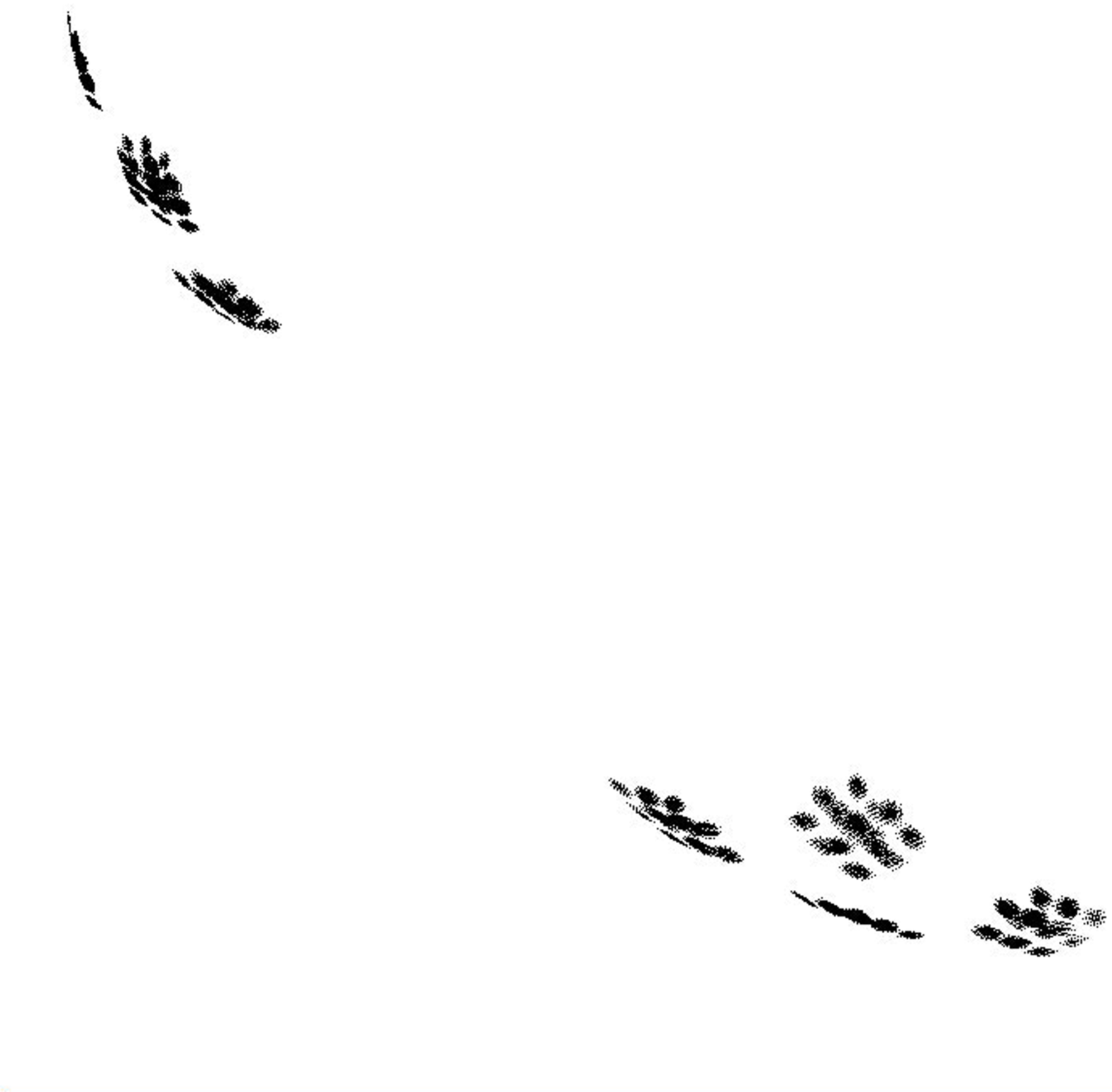}
  \includegraphics[width=3cm, height=3cm, frame]{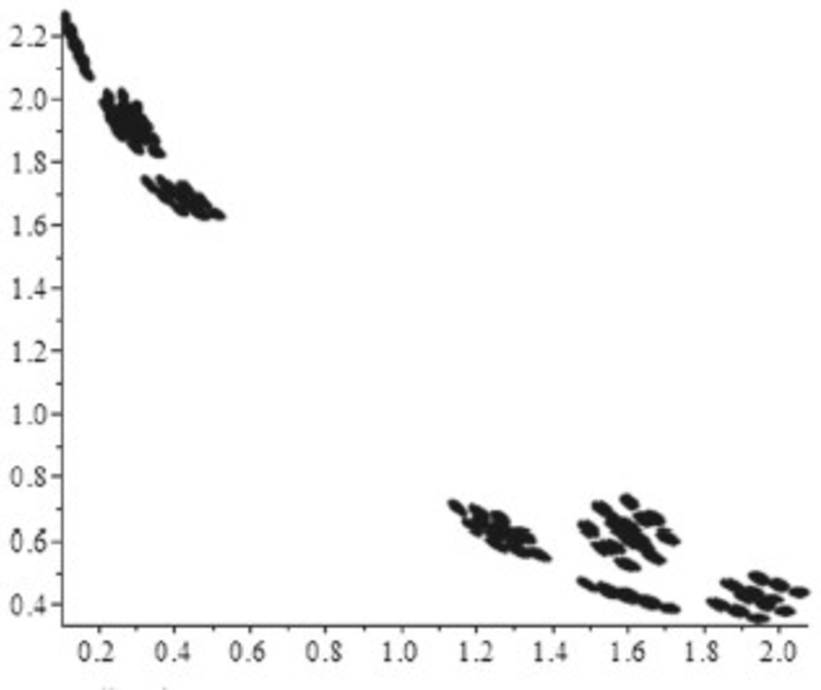}
  \caption{From the left to the right the output of the algorithm \texttt{GIFSDraw($\mS$)} uniform and aleatory after 5 iterations. On the right, the picture performed by the deterministic algorithm for GIFS from \cite{Jaros-2016}.}\label{ex5.3}
\end{figure}

\end{example}


\begin{example}\label{discreteGIFSex2}   Our second example is the GIFS $\mathcal{G}$ appearing in Example 8 from \cite{Jaros-2016}. The approximation by the algorithm \texttt{GIFSDraw($\mS$)} is presented in the Figure~\ref{ex5.4}. Consider $(\mathbb{R}^2, d_{e})$ a metric space, $X=[0.7,0.8]\times[0.65,1.0]$ and the GIFS $\phi_1, \phi_2: X^2 \to X$ where
\[\mS:
\left\{
  \begin{array}{ll}
       \phi_1(x_1,y_1,x_2,y_2)&=\left(0.05 x_{1}+0.02 x_{2}+0.1 y_{2}+0.635 , 0.1 x_{1}+0.2 y_{1}+0.08 x_{2}+0.15 y_{2}+0.5\right) \\
       \phi_2(x_1,y_1,x_2,y_2)&=\left(0.15 x_{1}+0.05 x_{2}+0.1 y_{2}+0.5 , 0.15 x_{1}+0.15 y_{1}+0.45\right) \\
  \end{array}
\right.
\]
  \begin{table}[ht]
   \centering
   \begin{tabular}{c|c|c|c|c|c|c}
     \hline
     $\alpha$ & $[a,b]\times[c,d]$ & $D$ & $n$ & $N$ & $\ve$ & $\delta$ \\
     \hline
      $0.4117654113$ & $[0.7,0.8]\times[0.65,1.0]$ & $0.3640054945$ & $700$ & $5$ & $2.5\times 10^{-4}$ & $0.006433811638$ \\
     \hline
   \end{tabular}
   \caption{Resolution data for the uniform picture in \ref{ex5.4}.}\label{table:ex5.4_fuzzy_gifs}
 \end{table}
\begin{figure}[!ht]
  \centering
  \includegraphics[width=3cm,frame]{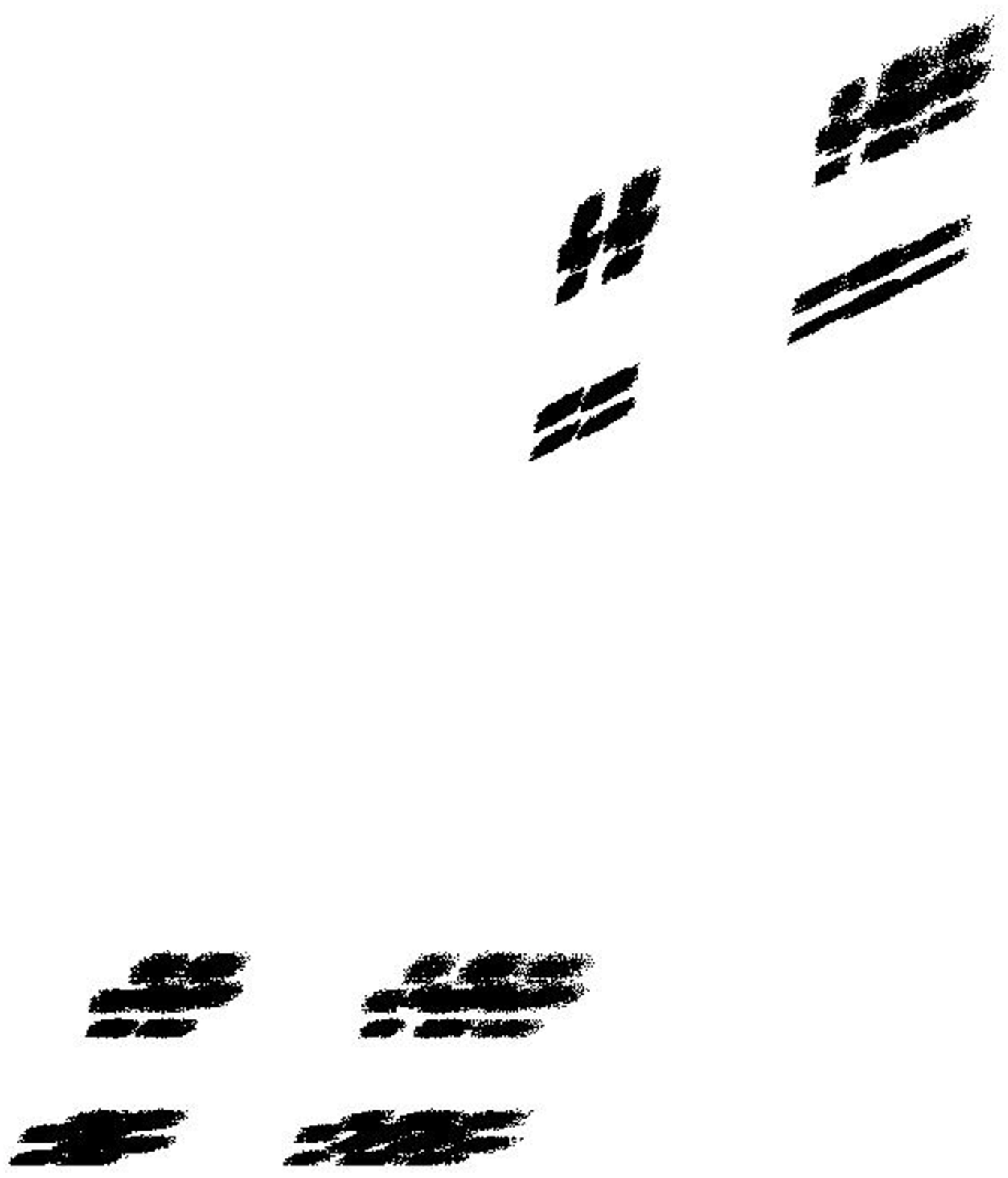}
  \includegraphics[width=3cm,frame]{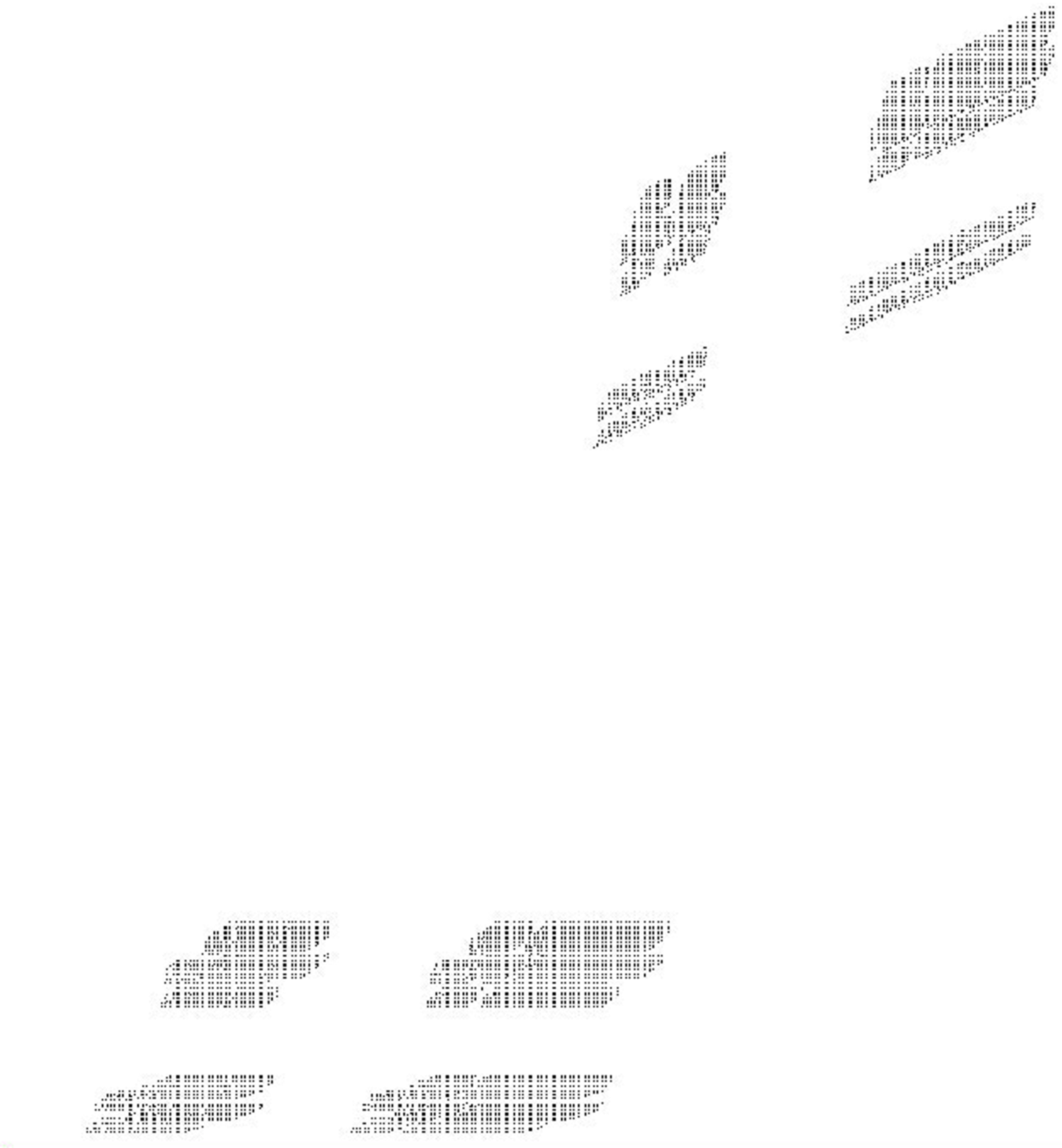}
  \includegraphics[width=3cm, height=3cm, frame]{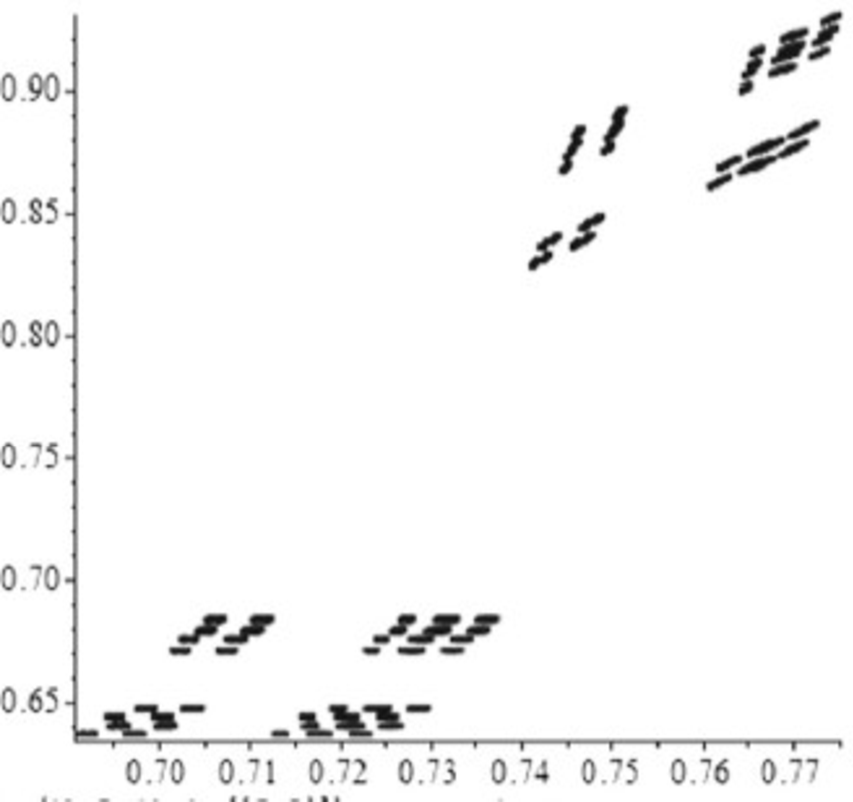}
  \caption{From the left to the right the output of the algorithm \texttt{GIFSDraw($\mS$)} uniform and aleatory after 5 iterations. On the right, the picture performed by the deterministic algorithm for GIFS from \cite{Jaros-2016}. }\label{ex5.4}
\end{figure}

\end{example}


\begin{example}\label{discreteGIFSex3}   Our last example is also based on a normalization of the GIFS $\mathcal{G}$ appearing in Example 16 from \cite{Jaros-2016}. The approximation by the algorithm \texttt{GIFSDraw($\mS$)} is presented in the Figure~\ref{ex5.5}. Consider $(\mathbb{R}^2, d_{e})$ a metric space, $X=[0 , 0.77]^2$ and the GIFS $\phi_1, \phi_2: X^2 \to X$ where
\[\mS:
\left\{
  \begin{array}{ll}
       \phi_1(x_1,y_1,x_2,y_2)&=(0.25 x_1+0.2 y_2, 0.25 y_1+0.2 y_2) \\
    \phi_2(x_1,y_1,x_2,y_2)&=(0.25 x_1+0.2 x_2, 0.25 y_1+0.1 y_2+0.5) \\
    \phi_3(x_1,y_1,x_2,y_2)&=(0.25 x_1+0.1 x_2+0.5, 0.25 y_1+0.2 y_2)) \\
  \end{array}
\right.
\]
  \begin{table}[ht]
   \centering
   \begin{tabular}{c|c|c|c|c|c|c}
     \hline
     $\alpha$ & $[a,b]\times[c,d]$ & $D$ & $n$ & $N$ & $\ve$ & $\delta$ \\
     \hline
      $0.5325006068$ & $[0,0.77]\times[0,0.77]$ & $1.088944443$ & $700$ & $5$ & $2.5\times 10^{-4}$ & $0.04929744899$ \\
     \hline
   \end{tabular}
   \caption{Resolution data for the uniform picture in \ref{ex5.5}.}\label{table:ex5.5_fuzzy_gifs}
 \end{table}
 \begin{figure}[!ht]
  \centering
  \includegraphics[width=3cm, frame]{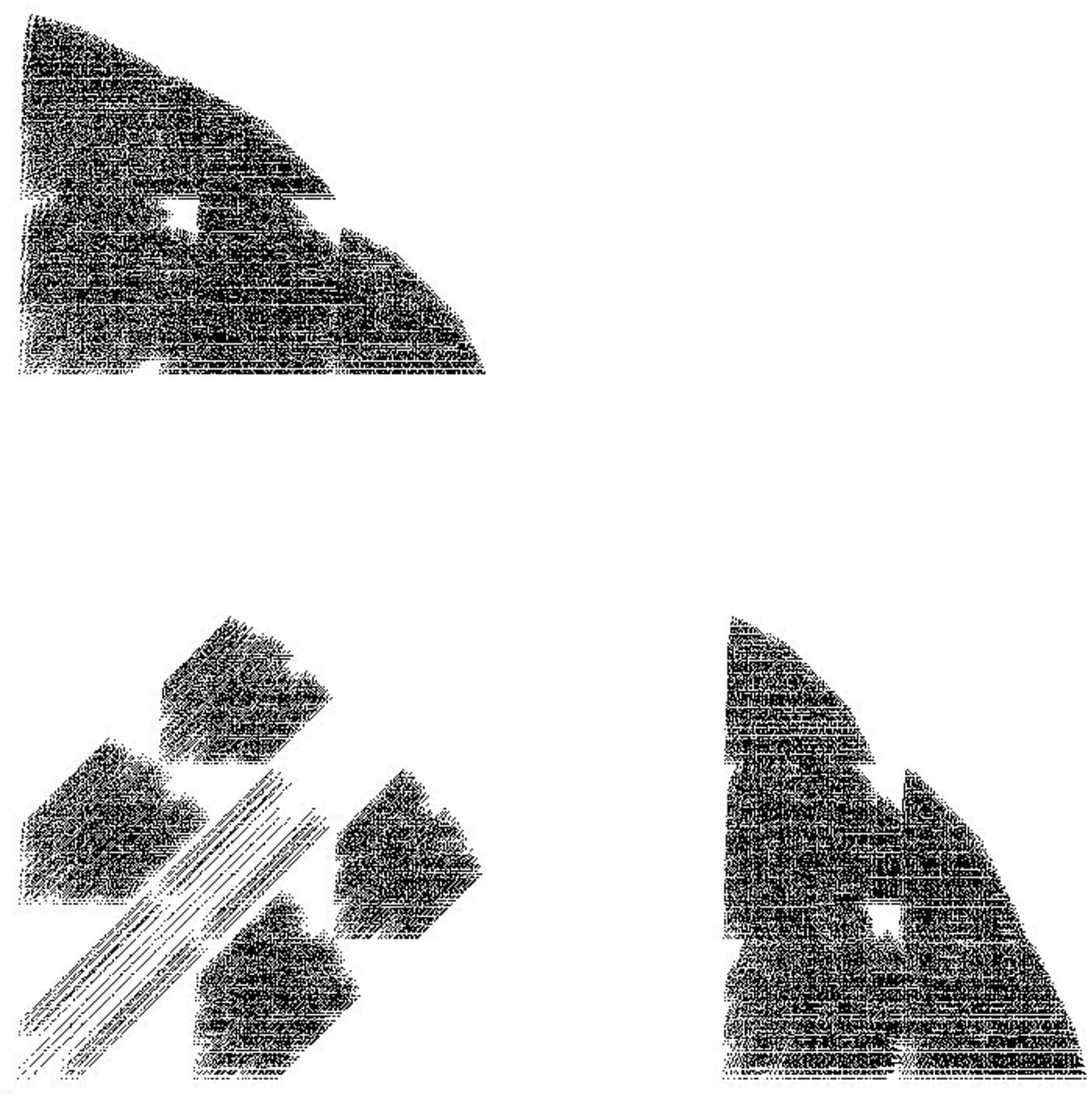}
  \includegraphics[width=3cm, frame]{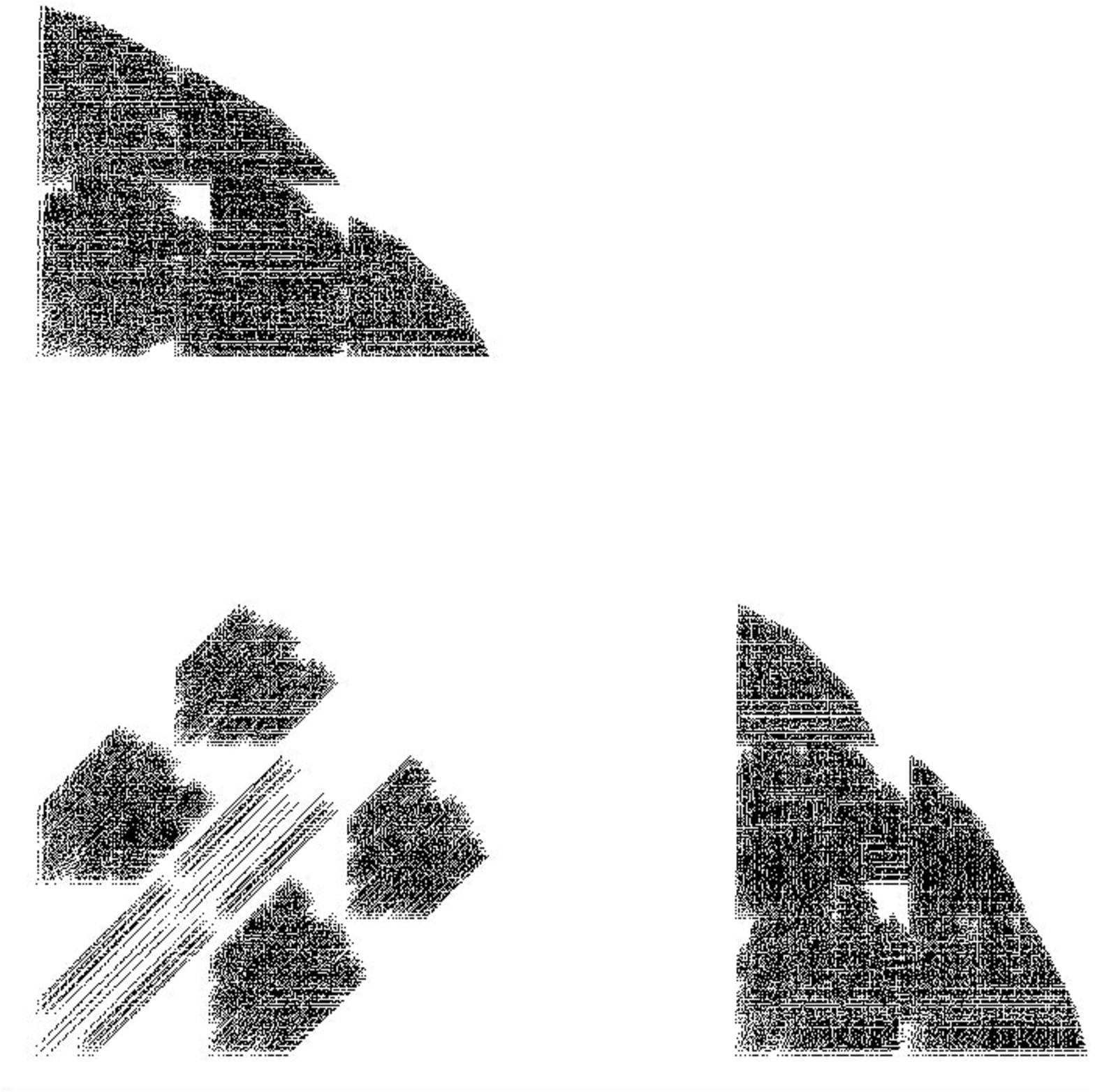}
 \includegraphics[width=3cm, height=3cm, frame]{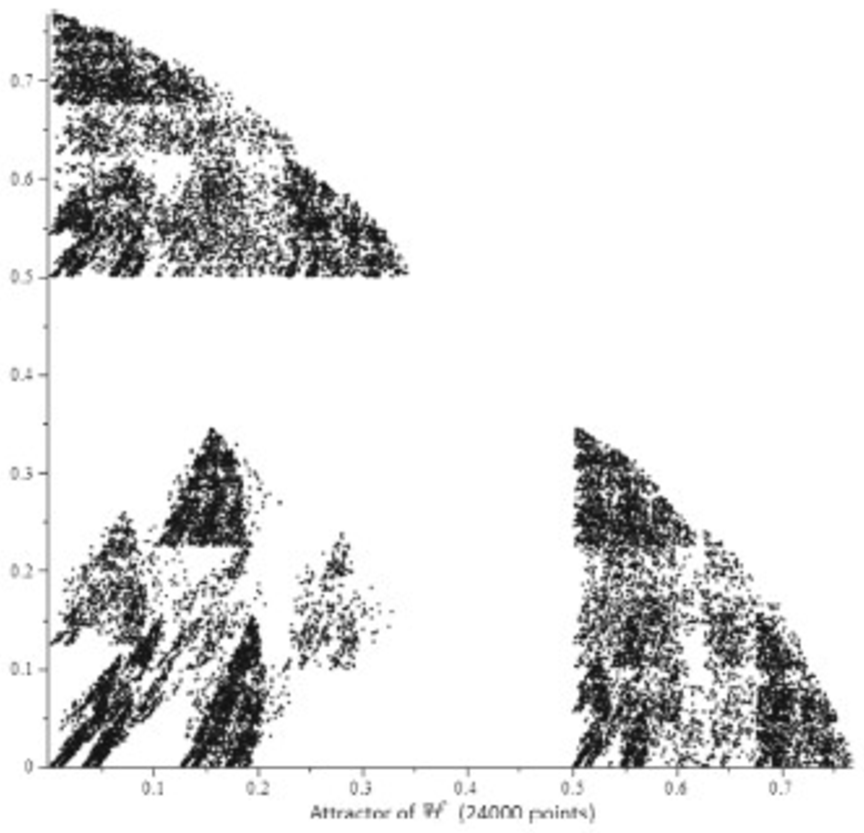}
  \caption{From the left to the right the output of the algorithm \texttt{GIFSDraw($\mS$)} uniform and aleatory after 5 iterations. On the right, the picture performed by the chaos game algorithm for GIFS from \cite{Jaros-2016}.}\label{ex5.5}
  \end{figure}

\end{example}


\subsection{Discrete deterministic algorithm for fuzzy GIFS} Let $\mathcal{S}=(X,(\phi_j)_{j=1}^L,(\rho_j)_{j=1}^L)$ be a fuzzy GIFS satisfying the conditions of Theorem~\ref{ttt4}. Then we can obtain a fractal with resolution $\delta$, approximating $u_{\mS}$, as follows:

{\tt
\begin{tabbing}\label{Determ Discrete fuzzy GIFS algorithm}
aaa\=aaa\=aaa\=aaa\=aaa\=aaa\=aaa\=aaa\= \kill
     \> \texttt{FuzzyGIFSDraw($\mS$)}\\
     \> {\bf input}: \\
     \> \>  \> $\delta> 0$, the resolution.\\
     \> \>  \> $u \in{\F}^*_{\hat{X}}$ a discrete fuzzy set. \\
     \> \>  \> The diameter $D$ of a ball in $({\F}^*_{X},d_{\infty})$ containing $\on{supp}(u_{\mS})\mbox{ and }\on{supp}(u)$.\\
     \> {\bf output}: A grey scale  image representing a fractal
     with resolution  at most $\delta$.\\
     \> {\bf Compute} $\displaystyle \alpha_\mS:={\rm Lip}(\mS):=\max_{1\leq j\leq L}{\rm Lip}(\hat{\phi}_{j}) $;\\
     \> {\bf Compute} $\varepsilon >0$ and $N \in \mathbb{N}$ such that $\frac{5\ve}{1-\alpha_\mS}+\alpha_\mS^N \, D < \delta$\\
     \>w:=u \\
     \> {\bf for n from 1 to N do}\\
     \> \> w:=$\overline{Z}_{\hat{\mS}}$(w) \\
     \> {\bf end do}\\
     \> {\bf return: Print}(w).
\end{tabbing}}

Following the same reasoning as in the  Remark~\ref{resolution choices} we can obtain, from Theorem~\ref{ttt4}, the appropriate estimates in order to have $\frac{5\ve}{1-\alpha_\mS}+\alpha_\mS^N \, D < \delta$. However the actual computation of $\overline{Z}_{\hat{\mS}}(u)$ requires some additional technical details.

\subsection{Fuzzy GIFS examples}
The algorithm \texttt{GIFZSDraw} will take unnecessary computational work when it is used for simpler cases compared with the respective algorithms. Therefore in this section we are going to consider exclusively ``genuine'' GIFS.

\subsubsection{Examples with dimension two}

The  next example is the GIFS $\mathcal{G}$ appearing in Example 16 from \cite{Jaros-2016}.

\begin{example}\label{example 4}
  Consider $(\mathbb{R}^2, d_{e})$ a metric space, $X=[0, 0.77]^2$ and a fuzzy GIFS $\phi_1,..., \phi_3: X^2 \to X$, where
\[\mS :\left\{
  \begin{array}{ll}
    \phi_1(x_1,y_1,x_2,y_2)&=(0.25 x_1+0.2 y_2, 0.25 y_1+0.2 y_2) \\
    \phi_2(x_1,y_1,x_2,y_2)&=(0.25 x_1+0.2 x_2, 0.25 y_1+0.1 y_2+0.5) \\
    \phi_3(x_1,y_1,x_2,y_2)&=(0.25 x_1+0.1 x_2+0.5, 0.25 y_1+0.2 y_2)
  \end{array}
\right.
\]
with the grey level functions
\[\rho_1(t):=
\left\{
  \begin{array}{ll}
      0    & 0\leq t < 0.2505  \\
      0.25  & 0.2505 \leq t < 0.505  \\
      0.5   & 0.505\leq t < 0.7505  \\
      0.75 & 0.7505 \leq  t \leq 1
  \end{array}
\right.,
\] $\rho_2(t):=t$, $\rho_3(t):= t$ and $\rho_4(t):=t$.

  \begin{table}[ht]
   \centering
   \begin{tabular}{c|c|c|c|c|c|c}
     \hline
     $\alpha$ & $[a,b]\times[c,d]$ & $D$ & $n$ & $N$ & $\ve$ & $\delta$ \\
     \hline
      $0.5325006068$ & $[0,0.77]\times[0,0.77]$ & $1.088944443$ & $700$ & $5$ & $2.5\times 10^{-4}$ & $0.04929744899$ \\
     \hline
   \end{tabular}
   \caption{Resolution data for the uniform picture in \ref{example4_determ}.}\label{table:example4_determ_fuzzy_gifs}
 \end{table}

The approximations of the fuzzy attractor, via uniform and aleatory algorithm, are given in the pictures, Figure~\ref{example4_determ} and Figure~\ref{example4_chaos}, respectively.

  \begin{figure}[!ht]
  \centering
  \includegraphics[width=3cm]{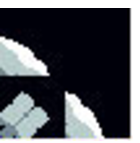}
  \includegraphics[width=3cm]{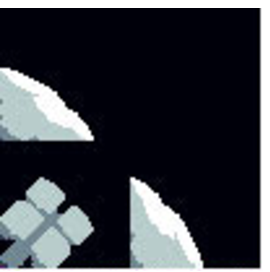}
  \includegraphics[width=3cm]{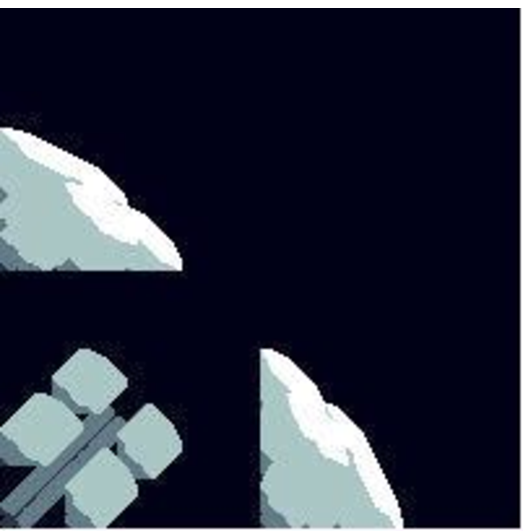}
  \includegraphics[width=3.2cm,height=3cm, frame]{exemplo4-2d_orig.eps}\\
  \caption{From the left to the right, the figure obtained by the uniform algorithm with $n=50$, $100$, $200$ and the picture drawn by \cite{Jaros-2016}.}\label{example4_determ}
  \end{figure}

  \begin{figure}[!ht]
  \centering
  \includegraphics[width=3cm]{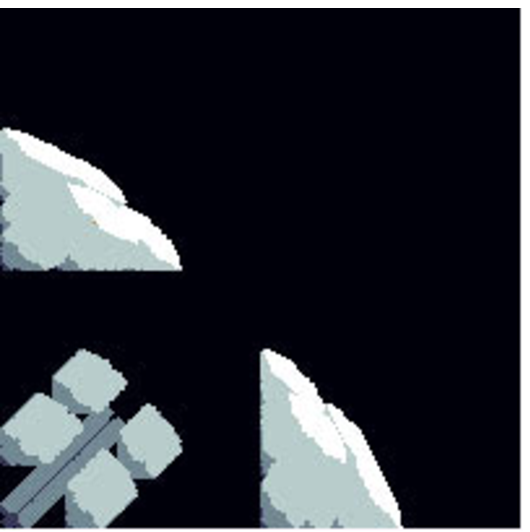} \quad
  \includegraphics[width=3cm]{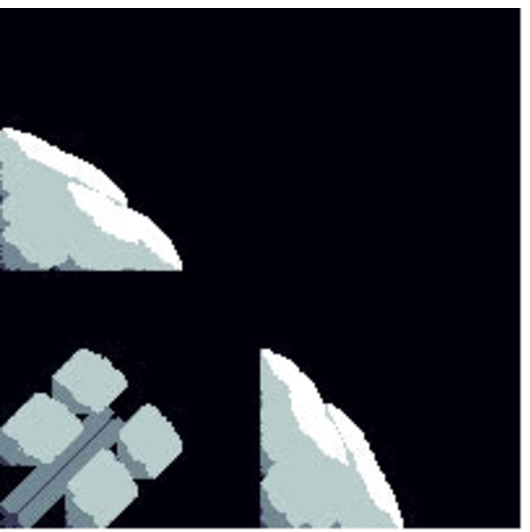}\\
  \caption{From the left to the right, the figure obtained by the chaos-game algorithm with $n=200$ and $na=40000000$, $n=200$ and $na=80000000$, aleatory sequences. Performance data for $na=40000000$: number of iterations, $9$; time for $\Phi^{-1}$ set construction, $305.68$ seconds; time for iterations, $1486.73$ seconds; overall running time, $1792.43$ seconds. Performance data for $na=80000000$: number of iterations, $3$; time for $\Phi^{-1}$ set construction, $601.00$ seconds; time for iterations, $935.89$ seconds; overall running time, $1536.92$ seconds.}\label{example4_chaos}
  \end{figure}

\end{example}


\begin{example}\label{example 6new}
Consider $(\mathbb{R}^2, d_{e})$ a metric space, $X=[0, 0.77]^2$ and a fuzzy GIFS $\phi_1,..., \phi_3: X^2 \to X$, where
\[\mS :\left\{
  \begin{array}{ll}
    \phi_1(x_1,y_1,x_2,y_2)&=(0.35 \cdot x_1+0.3 \cdot y_2+0.35,0.35 \cdot y_1+0.3 \cdot y_2+0.35) \\
    \phi_2(x_1,y_1,x_2,y_2)&=(0.35 \cdot x_1+0.3 \cdot x_2,0.35 \cdot y_1+0.2 \cdot y_2+0.3) \\
    \phi_3(x_1,y_1,x_2,y_2)&=(0.35 \cdot x_1+0.2 \cdot x_2+0.3,0.35 \cdot y_1+0.3 \cdot y_2)\\
  \end{array}
\right.
\]
with the grey scale functions
$\rho_1(t):=t$\\
\[\rho_2(t):=
\left\{
  \begin{array}{ll}
      0    & 0\leq t < 0.2505  \\
      0.25  & 0.2505 \leq t < 0.505  \\
      0.5   & 0.505\leq t < 0.7505  \\
      0.75 & 0.7505 \leq  t \leq 1
  \end{array}
\right.,
\] and $\rho_3(t):=t$.\\
    \begin{table}[ht]
   \centering
   \begin{tabular}{c|c|c|c|c|c|c}
     \hline
     $\alpha$ & $[a,b]\times[c,d]$ & $D$ & $n$ & $N$ & $\ve$ & $\delta$ \\
     \hline
      $0.7737553383$ & $[0,0.77]\times[0,0.77]$ & $1.088944443$ & $700$ & $5$ & $2.5\times 10^{-4}$ & $0.3075368638$ \\
     \hline
   \end{tabular}
   \caption{Resolution data for the uniform picture in \ref{example6new}.}\label{table:example6new}
 \end{table}

The approximations of the fuzzy attractor, via uniform and aleatory algorithm, are given in the Figure~\ref{example6new}.

  \begin{figure}[!ht]
  \centering
  \includegraphics[width=3cm,frame]{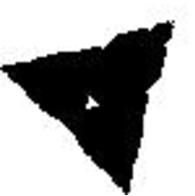}
  \includegraphics[width=3cm]{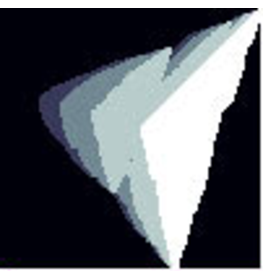}
  \includegraphics[width=3cm]{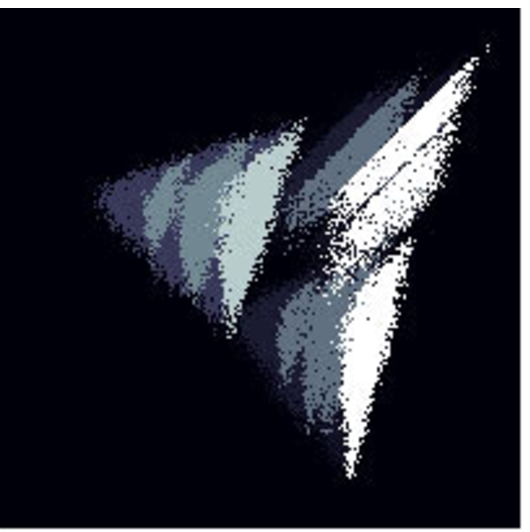}
  \caption{From the left to the right the output of the algorithm \texttt{GIFSDraw($\mS$)} and the output of the algorithm \texttt{FuzzyGIFSDraw($\mS$)} after 5 iterations on both uniform and aleatory versions.}\label{example6new}
  \end{figure}
\end{example}


\begin{example}\label{example 7new}
  Consider $(\mathbb{R}^2, d_{e})$ a metric space, $X=[-0.1,2.1]\times[-0.1,2.1]^2$ and a fuzzy GIFS $\phi_1,..., \phi_3: X^2 \to X$, where
\[\mS :\left\{
  \begin{array}{ll}
    \phi_1(x_1,y_1,x_2,y_2)&=(0.2 x_1+0.25 x_2+0.04 y_ 2,0.16 y_1-0.14 x_2+0.20 y_2+1.3) \\
    \phi_2(x_1,y_1,x_2,y_2)&=(0.2 x_1-0.15 y_1-0.21 x_2+0.15 y_2+1.3,0.25 x_1+0.15 y_1+0.25 x_2+0.17) \\
    \phi_3(x_1,y_1,x_2,y_2)&=(0.355  x_1+0.355  y_ 1+0.378,-0.355 x_1+0.355  y_1+0.434-0.03  y_2)\\
  \end{array}
\right.
\]
with the grey scale functions
\[\rho_1(t):=
\left\{
  \begin{array}{ll}
      0    & 0\leq t < 0.2505  \\
      0.25  & 0.2505 \leq t < 0.505  \\
      0.5   & 0.505\leq t < 0.7505  \\
      0.75 & 0.7505 \leq  t \leq 1
  \end{array}
\right.,
\] $\rho_2(t):= t$ and $\rho_3(t):=t$.\\

  \begin{table}[ht]
   \centering
   \begin{tabular}{c|c|c|c|c|c|c}
     \hline
     $\alpha$ & $[a,b]\times[c,d]$ & $D$ & $n$ & $N$ & $\ve$ & $\delta$ \\
     \hline
      $0.5654836644$ & $[-0.1,2.1]\times[-0.1,2.1]$ & $3.111269837$ & $700$ & $5$ & $2.5\times 10^{-4}$ & $0.1827795924$ \\
     \hline
   \end{tabular}
   \caption{Resolution data for the uniform picture in \ref{example7new}.}\label{table:example7new}
 \end{table}

The approximations of the fuzzy attractor, via uniform and aleatory algorithm, are given in the Figure~\ref{example7new}.

  \begin{figure}[!ht]
  \centering
  \includegraphics[width=3cm,frame]{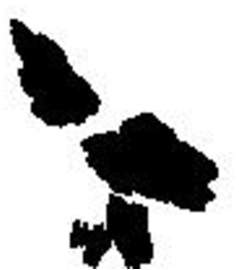}
  \includegraphics[width=3cm]{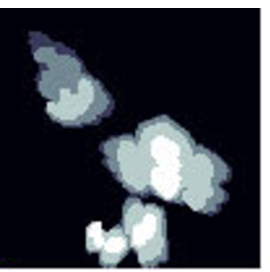}
  \includegraphics[width=3cm]{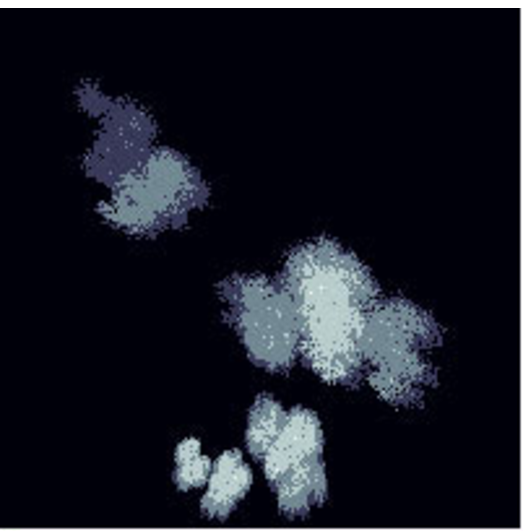}
  \caption{From the left to the right the output of the algorithm \texttt{GIFSDraw($\mS$)} and the output of the algorithm \texttt{FuzzyGIFSDraw($\mS$)} after 5 iterations on both uniform and aleatory versions.}\label{example7new}
  \end{figure}
\end{example}

\section{Final remarks}

\subsection{Numerical analysis and complexity}
  In this section we are going to explain some detail about the performance of the algorithm and its computational limitations.

  The algorithms described in Sections~\ref{sec:IFSDraw} \ref{sec:FuzzyIFSDraw} and \ref{sec:FuzzyIFSandGIFSDraw} were implemented in {\sc Fortran 2003} and were compiled using the {\sc GNU gfortran} compiler, version 5.4.0 with optimization {\tt -O3} turned on. The executable codes were run on a computer equipped with an {\sc Intel(R) Core(TM)} i5-6400T 2.21 GHz processor and 6 GBytes of RAM under the Cygwin environment for Windows 10, 64-bit operating system.

To provide an easier use of our programs, the description of each example is stored in a configuration file which is read and processed prior to the computations proper. In particular, the description of each $\phi$ and $\rho$ functions is written as an arithmetic expression using a syntax similar to that used in {\sc Fortran} source codes and these expressions are then processed and evaluated as needed during the computations.

Due to the fact that the set $\Phi^{-1}$ may require a substantial amount of memory to be stored, even for a small value of $n$, we have implemented two versions for using the deterministic and chaos-game approaches: one that stores the set $\Phi^{-1}$ in the computer's RAM as a linked list of nodes each containing a six-tuple (i.e. $i$, $j$, $i_1$, $j_1$, $i_2$ and $j_2$ representing each point in $\Phi^{-1}$) and another that stores these six-tuples on an external file on disk. These files are repeatedly read during the iteration to compute $F_{\mS}$.

The metric $d_\infty$ is used to decide when to stop the iterations; more specifically, a tolerance $\epsilon$ is specified in the configuration file and the iterations proceed until $d_\infty<\epsilon$. The experiments carried with our implementations also showed that in some cases $d_\infty$ are not reduced further from some iteration onwards. Therefore the iterations will also stop whenever $d_\infty$ remains the same for three consecutive iterations.


\begin{thebibliography}{99}

\bibitem[Bar88]{BOOK:9592}
Michael~Fielding Barnsley.
\newblock {\em Fractals everywhere}.
\newblock Academic Press, 1988.

\bibitem[dAICMDCNAS03]{Amo-2003}
Enrique de~Amo; I. Chitescu; Manuel Diaz Carrillo; Nicolae Adrian~Secelean.
\newblock A new approximation procedure for fractals.
\newblock {\em Journal of Computational and Applied Mathematics}, 151:355--370,
  2003.

\bibitem[DK94]{diamond1994metric}
Phil Diamond and Peter~E Kloeden.
\newblock {\em Metric spaces of fuzzy sets: theory and applications}.
\newblock World scientific, 1994.

\bibitem[Elq90]{Dubuc-1990}
Serge Dubuc;~Abdelkader Elqortobi.
\newblock Approximations of fractal sets.
\newblock {\em Journal of Computational and Applied Mathematics}, 29:79--89,
  1990.

\bibitem[Hut81]{HUT}
John Hutchinson.
\newblock Fractals and self-similarity.
\newblock {\em Indiana Univ. Math. J.}, 30:713--747, 1981.

\bibitem[Mic10a]{mihail2010generalized}
Alexandru Mihail;~Radu Miculescu.
\newblock Generalized ifss on noncompact spaces.
\newblock {\em Fixed Point Theory and Applications}, 2010(1):584215, 2010.

\bibitem[Mic10b]{Chitcescu-2010}
I.~Chit\c{c}escu; H. Georgescu;~R. Miculescu.
\newblock Approximation of infinite dimensional fractals generated by integral
  equations.
\newblock {\em Journal of Computational and Applied Mathematics},
  234:1417--1425, 2010.

\bibitem[Mih07]{MIHAIL-2007}
Radu Mihail, Alexandru;~Miculescu.
\newblock Applications of fixed point theorems in the theory of generalized
  ifs.
\newblock {\em Fixed Point Theory and Applications}, 2008:312876, 2007.

\bibitem[Mih08]{mihail2008recurrent}
Alexandru Mihail.
\newblock Recurrent iterated function systems.
\newblock {\em Revue Roumaine de Mathematiques Pures et Appliquees},
  53(1):43--54, 2008.

\bibitem[MMU19]{Miculescu_2019}
Radu Miculescu, Alexandru Mihail, and Silviu-Aurelian Urziceanu.
\newblock A new algorithm that generates the image of the attractor of a
  generalized iterated function system.
\newblock {\em Numerical Algorithms}, may 2019.

\bibitem[Per93]{peruggia1993discrete}
Mario Peruggia.
\newblock {\em Discrete iterated function systems}.
\newblock AK Peters/CRC Press, 1993.

\bibitem[Str15]{Strobin-2015}
Filip Strobin.
\newblock Attractors of generalized ifss that are not attractors of ifss.
\newblock {\em Journal of Mathematical Analysis and Applications}, 422:99--108,
  2015.

\bibitem[Str16]{Jaros-2016}
Patrycja Jaros ; {\L}ukasz Ma\'{s}lanka ;~Filip Strobin.
\newblock Algorithms generating images of attractors of generalized iterated
  function systems.
\newblock {\em Numerical Algorithms}, 73:477--499, 2016.

\bibitem[Str17]{Oliveira-2017}
Elismar R. Oliveira;~Filip Strobin.
\newblock Fuzzy attractors appearing from gifzs.
\newblock {\em Fuzzy Sets and Systems}, page S0165011417302038, 2017.

\bibitem[Swa13]{strobin_swaczyna_2013}
Filip Strobin;~Jaros{\l}aw Swaczyna.
\newblock On a certain generalisation of the iterated function system.
\newblock {\em Bulletin of the Australian Mathematical Society}, 87(1):37–54,
  2013.

\bibitem[Vrs92]{Cabrelli-1992}
Carlos A. Cabrelli; Bruno Forte; Ursula M. Molter; Edward~R. Vrscay.
\newblock Iterated fuzzy set systems: A new approach to the inverse problem for
  fractals and other sets.
\newblock {\em Journal of Mathematical Analysis and Applications}, 171:79--100,
  1992.

\bibitem[Yan94]{Yang_1994}
Hailang Yang.
\newblock A projective algorithm for approximation of fractal sets.
\newblock {\em Applied Mathematics and Computation}, 63(2-3):201--212, jul
  1994.

\end{thebibliography}
\end{document}